\documentclass[11pt, twoside]{preprint}

\usepackage{times}
\usepackage{amssymb,amsmath,amsfonts,amsthm}
\usepackage{mhequ} 
\usepackage{mathtools, mathrsfs}
\usepackage[title]{appendix}
\usepackage{microtype}

\usepackage{hyperref}
\usepackage[inline]{enumitem}
\usepackage{bm}
\usepackage{dcolumn}
\usepackage{tabu}
\usepackage{subfig}
\usepackage{breakurl}
\usepackage{cases}
\usepackage{wasysym}
\usepackage{centernot}
\usepackage{soul}

\setlist[itemize]{noitemsep, nolistsep}
\setlist[enumerate]{noitemsep, nolistsep}
\theoremstyle{plain}
\newtheorem{lemma}{Lemma}[section]
\newtheorem{theorem}[lemma]{Theorem}
\newtheorem{corollary}[lemma]{Corollary}
\newtheorem{proposition}[lemma]{Proposition}
\newtheorem{assumption}[lemma]{Assumption}
\theoremstyle{definition}
\newtheorem{definition}[lemma]{Definition}
\newtheorem{remark}[lemma]{Remark}
\newtheorem{example}[lemma]{Example}

\makeatletter
\newcommand{\VERT}{\vert\!\vert\!\vert}
\def\onorm#1{| #1|_0}
\def\enorm#1{| #1|_{\eps}}
\def\snorm#1{\Vert #1\Vert_\s}
\def\sabs#1{| #1|_\s}
\def\senorm#1{\Vert #1 \Vert_{\s, \eps}}

\def\wnorm#1{\lfloor \hspace{-0.29em} \rceil #1 \lfloor \hspace{-0.29em} \rceil^{(\eps)}}
\newcommand{\boxvert}{\vrule\@width 2\p@}
\newcommand{\eqdef}{\stackrel{\mathclap{\mbox{\tiny def}}}{=}}
\def\gen #1{{#1}^{\mathrm{gen}}}
\def\poly #1{{#1}_{\mathrm{poly}}}
\makeatother

\renewcommand{\P}{\mathbf{P}}
\newcommand{\R}{\mathbf{R}}

\newcommand{\Z}{\mathbf{Z}}
\newcommand{\N}{\mathbf{N}}
\newcommand{\Torus}{\mathbf{T}}
\newcommand{\projw}{\mathcal{P}}

\newcommand{\eps}{\varepsilon}

\def\CS{\mathcal{S}}
\def\CU{\mathcal{U}}
\def\CK{\mathcal{K}}
\def\CJ{\mathcal{J}}
\def\CI{\mathcal{I}}
\def\CN{\mathcal{N}}
\def\CO{\mathcal{O}}
\def\CQ{\mathcal{Q}}
\def\CR{\mathcal{R}}
\def\CB{\mathcal{B}}
\def\CC{\mathcal{C}}
\def\CD{\mathcal{D}}
\def\CW{\mathcal{W}}

\def\HT{\hat{\CT}}
\def\HF{\hat{\CF}}

\def\CT{\mathcal{T}}
\def\CF{\mathcal{F}}
\def\CA{\mathcal{A}}
\def\CG{\mathcal{G}}

\def\EE{\mathbf{E}}
\def\1{\mathbf{1}}
\def\MM{\mathscr{M}}
\def\ST{\mathscr{T}}
\def\SF{\mathscr{F}}
\def\s{\mathfrak{s}}
\def\CP{\mathcal{P}}
\def\spanning{\mathrm{span}}
\def\PPsi{{\boldsymbol{\Psi}}}
\def\rho{\varrho}
\def\phi{\varphi}

\hypersetup{
	citecolor=blue,
	colorlinks=true, 
	urlcolor= blue, 
	linkcolor= black, 
	bookmarksopen=true, 
	pdftitle={DiscreteSPDEs}, 
}

\begin{document}

\title{Discretisations of rough stochastic PDEs}
\author{M. Hairer$^1$ and K. Matetski$^2$}
\institute{University of Warwick, \email{m.hairer@warwick.ac.uk}
\and University of Toronto, \email{matetski@math.toronto.edu}}
\date{\today}
\maketitle

\begin{abstract}
We develop a general framework for spatial discretisations of parabolic stochastic PDEs whose 
solutions are provided in the framework of the theory of regularity structures and which are 
functions in time. As an application, we show that the dynamical $\Phi^4_3$ model on the 
dyadic grid converges after renormalisation to its continuous counterpart. This result in particular implies that, as expected, the $\Phi^4_3$ measure with a sufficiently small coupling constant is invariant for this equation and that 
the lifetime of its solutions is almost surely infinite for almost 
every initial condition.
\end{abstract}


\tableofcontents


\section{Introduction}

The aim of this article is to develop a general framework for spatial discretisations of the parabolic stochastic PDEs of the form
\begin{equ}
 \partial_t u = A u + F(u, \xi)\;,
\end{equ}
where $A$ is an elliptic differential operator, $\xi$ is a rough noise, and $F$ is a non-linear function in $u$
which is affine in $\xi$. The class of spatial discretisations we work with are of the form
\begin{equ}
 \partial_t u^\eps = A^\eps u^\eps + F^\eps(u^\eps, \xi^\eps)\;,
\end{equ}
with the spatial variable taking values in the dyadic grid with mesh size $\eps > 0$, 
where $A^\eps$, $\xi^\eps$ and $F^\eps$ are discrete approximations of $A$, $\xi$ and $F$ 
respectively.

A particular example prototypical for the class of equations we are interested in 
is the dynamical $\Phi^4$ model in dimension $3$, 
which can be formally described by the equation
\begin{equs}[e:Phi]
\partial_t \Phi = \Delta \Phi + \bigl(\infty - a\bigr) \Phi - \lambda \Phi^3 + \xi\;, \qquad \Phi(0, \cdot) = \Phi_0(\cdot)\;, \tag{$\Phi^4_3$}
\end{equs}
on the torus $\Torus^3 \eqdef (\R / \Z)^3$ and for $t \geq 0$, where $\Delta$ is the Laplace operator on $\Torus^3$, $a \in \R$ is a fixed constant, $\lambda > 0$ is a ``coupling constant'', $\Phi_0$ is some initial data, and $\xi$ is the space-time white noise over $L^2(\R \times \Torus^3)$, see \cite{PZ14}. 

Here, $\infty$ denotes an ``infinite constant'': \eqref{e:Phi} should be interpreted
as the limit of solutions to the equation obtained by mollifying $\xi$ and replacing 
$\infty$
by a constant which diverges in a suitable way as the mollifier tends to the identity.
It was shown in \cite{Hai14} that this limit exists and is independent of the choice
of mollifier. The reason for the appearance of this infinite constant is that
solutions are random Schwartz distributions (this is already the case for the linear
equation, see \cite{PZ14}), so that their third power is undefined.
 The above notation also correctly suggests that 
solutions to \eqref{e:Phi} still depend on one parameter, namely the ``finite part'' of the
infinite constant, but this will not be relevant here and we consider this as being fixed from
now on. 

In two spatial dimensions, a solution theory for \eqref{e:Phi} was given in 
\cite{AR91,DPD03}, see also \cite{MR815192} for earlier work on a closely related model. 
In three dimensions, alternative approaches to \eqref{e:Phi} were recently obtained 
in \cite{CC13} (via paracontrolled 
distributions, see \cite{GIP12} for the development of that approach), 
and in \cite{Antti} (via renormalisation group techniques
\`a la Wilson).

It is natural to consider finite difference approximations to \eqref{e:Phi} for a number of reasons.
Our main motivation goes back to the seminal article \cite{BFS83}, where the authors provide a 
very clean and relatively compact argument showing that lattice approximations $\mu_\eps$ to the $\Phi^4_3$
measure are tight as the mesh size goes to $0$. These measures are given on the dyadic grid $\Torus^3_\eps \subset \Torus^3$ with the mesh size $\eps > 0$ by
\begin{equ}[e:mu_eps]
 \mu_\eps(\Phi^\eps) \eqdef e^{-S_\eps(\Phi^\eps)} \prod_{x \in \Torus^3_\eps} d \Phi^\eps(x) / Z_\eps\;,
\end{equ}
for every function $\Phi^\eps$ on $\Torus^3_\eps$, where $Z_\eps$ is a normalisation factor and
\begin{equ}[e:DAction]
S_\eps(\Phi^\eps) \eqdef \eps \sum_{x \sim y} \bigl(\Phi^\eps(x) - \Phi^\eps(y)\bigr)^2 - \frac{(C_\lambda^{(\eps)} - a) \eps^3}{2} \sum_{x \in \Torus^3_\eps} \Phi^\eps(x)^2 + \frac{\lambda \eps^3}{4} \sum_{x \in \Torus^3_\eps} \Phi^\eps(x)^4\;,
\end{equ}
with $C_\lambda^{(\eps)}$ being a ``renormalisation constant'' and with the first sum running over all the nearest neighbours on the grid $\Torus^3_\eps$, when each pair $x, y$ is counted twice. Then the $\Phi^4_3$ measure $\mu$ can be heuristically written as
\begin{equ}[e:PhiMeasure]
 \mu(\Phi) \sim e^{-S(\Phi)} \prod_{x \in \Torus^3} d \Phi(x)\;, 
\end{equ}
for $\Phi \in \CS'$ and for $S$ begin a limit of its finite difference approximations \eqref{e:DAction}:
\begin{equ}
S(\Phi) = \int_{\Torus^3} \left(\frac{1}{2} \bigl(\nabla \Phi(x)\bigr)^2 - \frac{\infty - a}{2} \Phi(x)^2 + \frac{\lambda}{4} \Phi(x)^4\right) dx\;.
\end{equ}
Since the measures $\mu_\eps$ with a sufficiently small coupling constant are invariant for the natural
finite difference approximations of \eqref{e:Phi}, showing that these converge to \eqref{e:Phi}
straightforwardly implies that any accumulation point of $\mu_\eps$ is invariant for the
solutions of \eqref{e:Phi}. These accumulation points are known to coincide with the $\Phi^4_3$ measure $\mu$
\cite[Thm.~2.1]{Park}, thus showing that $\mu$ is indeed invariant for \eqref{e:Phi}, as one might expect.
Another reason why discretisations of \eqref{e:Phi} are interesting is because they can be related to 
the behaviour of Ising-type models under Glauber dynamics near their critical temperature, see \cite{Simon1,Simon2}. 
See also the related result \cite{MW14} where the dynamical $\Phi^4_2$ model 
is obtained from the Glauber dynamic for a Kac-Ising model
in a more direct way, without going through lattice approximations. 
Similar results are expected to hold in three spatial dimensions, see e.g.\ the
review article \cite{GLP99}.

We will henceforth consider discretisations of \eqref{e:Phi} of the form
\begin{equs}[e:DPhiRenorm]
\frac{d}{dt} \Phi^\eps = \Delta^\eps \Phi^\eps + \bigl( C_\lambda^{(\eps)} - a\bigr) \Phi^\eps - \lambda \bigl(\Phi^\eps\bigr)^3 + \xi^\eps\;, \quad \Phi^\eps(0, \cdot) = \Phi^\eps_0(\cdot)\;, \tag{$\Phi^{4}_{3,\eps}$}
\end{equs}
on the dyadic discretisation $\Torus^3_\eps$ of $\Torus^3$ with mesh size $\eps = 2^{-N}$ for $N \in \N$, where $\Phi^\eps_0 \in \R^{\Torus^3_\eps}$, $\Delta^\eps$ is the nearest-neighbour approximation of the Laplacian $\Delta$, $\xi^\eps$ is a spatial discretisation of $\xi$, and $C_\lambda^{(\eps)}$ is a sequence of diverging, as $\eps \to 0$, renormalization constants depending on $\lambda$. We construct these discretisations on a common probability space by setting
\begin{equ}[e:SimpleDNoise]
\xi^\eps(t,x) \eqdef \eps^{-3} \langle \xi(t, \cdot), \1_{|\cdot - x| \leq \eps / 2} \rangle\;, \qquad (t,x) \in \R \times \Torus^3_\eps\;,
\end{equ}
where $| x|$ denotes the supremum norm of $x \in \R^3$.
Our results are however flexible enough to easily accommodate a variety of different approximations to the noise and the Laplacian. 

Existence and uniqueness of global solutions to \eqref{e:DPhiRenorm} for any fixed $\eps > 0$ follows 
immediately from standard results for SDEs \cite{Khasminkii,IW89}. 
Our main approximation result is the following, where we take the initial conditions
$\Phi^\eps_0$ to be random variables defined on a common probability space,
independent of the noise $\xi$. (We could of course simply take them deterministic, but
this formulation will be how it will then 
be used in our proof of existence of global solutions.)

\begin{theorem}\label{t:Phi}
Let $\xi$ be a space-time white noise over $L^2(\R \times \Torus^3)$ on a probability space $(\Omega, \SF, \mathbf{P})$, let $\Phi_0 \in \CC^\eta(\R^{3})$ almost surely,
for some $\eta > -\frac{2}{3}$, and let $\Phi$ be the unique maximal solution of \eqref{e:Phi} on $[0, T_\star]$ with fixed constants $a \in \R$ and $\lambda > 0$. Let furthermore $\Delta^\eps$ be the nearest-neighbour approximation of $\Delta$, let $\Phi^\eps_0 \in \R^{\Torus^3_\eps}$ be a random variable on the same probability space, let $\xi^\eps$ be given by \eqref{e:SimpleDNoise}, and let $\Phi^\eps$ be the unique global solution of \eqref{e:DPhiRenorm}. If the initial data satisfy almost surely
\begin{equ}
\lim_{\eps \to 0} \Vert \Phi_0; \Phi^\eps_0 \Vert^{(\eps)}_{\CC^\eta} = 0\;,
\end{equ}
then for every $\alpha < -\frac{1}{2}$ there is a sequence of renormalisation constants 
\begin{equ}[e:RenormIntro]
C^{(\eps)}_\lambda \sim \frac{\lambda}{\eps} - \lambda^2 \log \eps
\end{equ}
 in \eqref{e:DPhiRenorm} and a sequence of stopping times $T_\eps$ (which also depend on $\lambda$ and $a$) satisfying $\lim_{\eps \to 0} T_\eps = T_\star$ in probability such that, for every $\bar{\eta} < \eta \wedge \alpha$, and for any $\delta > 0$ small enough, one has the  limit in probability
\begin{equ}[e:PhiConvergence]
\lim_{\eps \to 0} \Vert \Phi; \Phi^\eps \Vert^{(\eps)}_{\CC^{\delta, \alpha}_{\bar{\eta}, T_\eps}} = 0\;.
\end{equ}
\end{theorem}

\begin{remark}
By writing \eqref{e:RenormIntro} we mean that $C^{(\eps)}_\lambda$ is a sum of two terms proportional to $\lambda$ and $- \lambda^2$ respectively, whose asymptotic divergence speeds are $\eps^{-1}$ and $\log \eps$  as $\eps \to 0$. 
\end{remark}

As a corollary of this convergence result and an argument along the lines of
\cite{Bourgain}, we have the following result, where we denote by $\mu$ the $\Phi^4_3$ measure
on the torus with a coupling constant $\lambda > 0$ and mass $m_0 > 0$, see \cite{BFS83} for a definition. 

\begin{corollary}\label{c:Phi}
If $a = m_0^2 > 0$ and if the coupling constant $\lambda > 0$ in \eqref{e:Phi} is sufficiently small, then for $\mu$-almost every initial condition $\Phi_0$ and for every $T > 0$, the solution of \eqref{e:Phi} constructed in \cite{Hai14} belongs to $\CC^{\delta, \alpha}_{\bar \eta}\bigl([0,T], \Torus^3\bigr)$, 
for $\delta$, $\alpha$ and $\bar \eta$ as in \eqref{e:PhiConvergence}. In particular, this
yields a reversible Markov process on $\CC^{\alpha}\bigl(\Torus^3\bigr)$ with an
invariant measure $\mu$.
\end{corollary}

In order to prove this result, we will use regularity structures, as 
introduced in \cite{Hai14}, to obtain uniform bounds (in $\eps$) on solutions to 
\eqref{e:DPhiRenorm} by describing the right hand side via a type of generalised
``Taylor expansion'' in the neighbourhood of any space-time point. The problem of 
obtaining uniform bounds is then split into the problem of on the one hand 
obtaining uniform bounds on the objects playing the role of Taylor monomials
(these require subtle stochastic cancellations, but are given by explicit formulae),
and on the other hand obtaining uniform regularity estimates on the ``Taylor coefficients''
(these are described implicitly as solutions to a fixed point problem but can be
controlled by standard Banach fixed point arguments). 

In order to treat the discretised equation \eqref{e:DPhiRenorm}, we introduce a discrete
analogue to the concept of ``model'' introduced in \cite{Hai14} and we show that the
corresponding ``reconstruction map'' satisfies uniform bounds analogous to the ones
available in the continuous case. 
One technical difficulty we encounter with this approach
is that the set-up is somewhat asymmetric since
time is continuous while space is discrete. Instead of considering a fixed model 
as in \cite{Hai14}, we will consider a family of models indexed by the time parameter
and satisfying a suitable regularity property. This idea requires some modification of the original theory, in particular of the ``abstract integration'' operation 
\cite[Sec.~5]{Hai14} and of the corresponding Schauder-type estimates.

As this article was nearing its completion, Zhu and Zhu \cite{Twins}  
independently obtained the convergence of 
solutions to \eqref{e:DPhiRenorm} to those of \eqref{e:Phi} using different methods.
Additionally, Gubinelli and Perkowski \cite{Reloaded} recently obtained a similar
result for the KPZ equation. 
One advantage of the approach pursued here is that it is quite systematic and that 
many of our intermediate
results do not specifically refer to the $\Phi^4_3$ model.
This lays the foundations of a systematic approximation theory which 
can in principle be applied to many other singular SPDEs, e.g.\ 
stochastic Burgers-type equations \cite{Hai11a,HMW12, HM14}, the KPZ 
equation \cite{PhysRevLett.56.889,MR1462228,Hai13}, 
or the continuous parabolic Anderson model \cite{Hai14, HL15}.

\subsection*{Structure of the article}

In Section~\ref{s:RegStruct} we introduce regularity structures and inhomogeneous models 
(i.e.\ models which are functions in the time variable). Furthermore, we prove here the 
key results of the theory in our present framework, namely the reconstruction theorem 
and the Schauder estimates. In Section~\ref{s:PDEs} we provide a solution theory for 
a general parabolic stochastic PDE, whose solution is a function in time. 
Section~\ref{s:DModels} is devoted to the development of a discrete analogue of 
inhomogeneous models, which we use in Section~\ref{s:DPDEs} to analyse solutions 
of discretised stochastic equations. In Section~\ref{s:GaussModels} we analyse models, 
built from a Gaussian noise. Finally, in Section~\ref{s:DPhi}, we prove Theorem~\ref{t:Phi} 
and Corollary~\ref{c:Phi}.

\subsection*{Notations and conventions}

Throughout this article, we will work in $\R^{d+1}$ where $d$ is the dimension of space and $1$ is the dimension of time. Moreover, we consider the time-space scaling $\s = (\s_0, 1, \ldots, 1)$ of $\R^{d+1}$, where $\s_0 > 0$ is an integer time scaling and $\s_i = 1$, for $i=1, \ldots, d$, is the scaling in each spatial direction. We set $|\s| \eqdef \sum_{i=0}^d \s_i$, denote by $| x|$ the $\ell^\infty$-norm of a point $x \in \R^d$, and define $\snorm{z} \eqdef |t|^{1/\s_0} \vee | x|$ to be the $\s$-scaled $\ell^\infty$-norm of $z=(t,x) \in \R^{d+1}$. For a multiindex $k \in \N^{d+1}$ we define $\sabs{k} \eqdef \sum_{i=0}^{d} \s_i k_i$, and for $k \in \N^d$ with the scaling $(1, \ldots, 1)$ we denote the respective norm by $|k|$. (Our natural numbers $\N$ include $0$.)

For $r > 0$, we denote by $\CC^r(\R^d)$ the usual H\"{o}lder space on $\R^d$, by $\CC^r_0(\R^d)$ we denote the space of compactly supported $\CC^r$-functions and by $\CB^r_0(\R^d)$ we denote the set of $\CC^r$-functions, compactly supported in $B(0,1)$ (the unit ball centered at the origin) and with the $\CC^r$-norm bounded by $1$.

For $\varphi \in \CB^r_0(\R^d)$, $\lambda > 0$ and $x, y \in \R^{d}$ we define $\varphi_x^\lambda(y) \eqdef \lambda^{-d} \varphi(\lambda^{-1}(y-x))$. For $\alpha < 0$, we define the space $\CC^\alpha(\R^d)$ to consist of $\zeta \in \CS'(\R^d)$, belonging to the dual space of the space of $\CC^{r}_0$-functions, with $r > -\lfloor \alpha \rfloor$, and such that
\begin{equ}[e:AlphaNorm]
\Vert \zeta \Vert_{\CC^\alpha} \eqdef \sup_{\varphi \in \CB^r_0} \sup_{x \in \R^d} \sup_{\lambda \in (0,1]} \lambda^{-\alpha} |\langle \zeta, \varphi_x^\lambda \rangle| < \infty\;.
\end{equ}
Furthermore, for a function $\R \ni t \mapsto \zeta_t$ we define the operator $\delta^{s, t}$ by
\begin{equ}[e:deltaTime]
 \delta^{s, t} \zeta \eqdef \zeta_t - \zeta_s\;,
\end{equ}
and for $\delta > 0$, $\eta \leq 0$ and $T > 0$, we define the space $\CC^{\delta, \alpha}_{\eta}\bigl([0,T], \R^d\bigr)$ to consist of the functions $(0, T] \ni t \mapsto \zeta_t \in \CC^{\alpha}(\R^d)$, such that the following norm is finite
\begin{equ}[e:SpaceTimeNorm]
\Vert \zeta \Vert_{\CC^{\delta, \alpha}_{\eta, T}} \eqdef \sup_{t \in (0, T]} \onorm{t}^{-\eta} \Vert \zeta_t \Vert_{\CC^\alpha} + \sup_{s \neq t \in (0, T]} \onorm{t, s}^{-\eta} \frac{\Vert \delta^{s, t} \zeta \Vert_{\CC^{\alpha - \delta}}}{|t-s|^{\delta/\s_0}}\;,
\end{equ}
where $\onorm{t} \eqdef |t|^{1/\s_0} \wedge 1$ and $\onorm{t, s} \eqdef \onorm{t} \wedge \onorm{s}$. The space $\CC^{0, \alpha}_{\eta}\bigl([0,T], \R^d\bigr)$ contains the function $\zeta$ as above which are continuous in time and is equipped with the norm defined by the first term in \eqref{e:SpaceTimeNorm}. 

Sometimes we will need to work with space-time distributions with scaling $\s$. In order to describe their regularities, we define, for a test function $\varphi$ on $\R^{d+1}$, for $\lambda > 0$ and $z, \bar z \in \R^{d+1}$,
\begin{equ}[e:ScaleEta]
\varphi_z^{\lambda, \s}(\bar z) \eqdef \lambda^{-|\s|} \varphi\bigl(\lambda^{-s_0}(\bar z_0- z_0), \lambda^{-1}(\bar z_1- z_1), \ldots, \lambda^{-1}(\bar z_d- z_d)\bigr)\;,
\end{equ}
and we define the space $\CC_{\s}^\alpha(\R^{d+1})$ similarly to $\CC^\alpha(\R^{d})$, but using the scaled functions \eqref{e:ScaleEta} in \eqref{e:AlphaNorm}.

In this article we will also work with discrete functions $\zeta^\eps \in \R^{\Lambda_\eps^d}$ on the dyadic grid $\Lambda_\eps^d \subset \R^d$ with the mesh size $\eps  = 2^{-N}$ for $N \in \N$. In order to compare them with their continuous counterparts $\zeta \in \CC^\alpha(\R^d)$ with $\alpha \leq 0$, we introduce the following ``distance''
\begin{equ}
\Vert \zeta; \zeta^\eps \Vert^{(\eps)}_{\CC^\alpha} \eqdef \sup_{\varphi \in \CB^r_0} \sup_{x \in \Lambda_\eps^d} \sup_{\lambda \in [\eps,1]} \lambda^{-\alpha} |\langle \zeta, \varphi_x^\lambda \rangle - \langle \zeta^\eps, \varphi_x^\lambda \rangle_\eps|\;,
\end{equ}
where $\langle \cdot, \cdot \rangle_\eps$ is the discrete analogue of the duality pairing on the grid, i.e.
\begin{equ}[e:DPairing]
 \langle \zeta^\eps, \varphi_x^\lambda \rangle_\eps \eqdef \int_{\Lambda_\eps^d} \zeta^\eps(y) \varphi_x^\lambda(y)\, dy \eqdef \eps^d \sum_{y \in \Lambda_\eps^d} \zeta^\eps(y) \varphi_x^\lambda(y)\;.
\end{equ}
For space-time distributions / functions $\zeta$ and $\zeta^\eps$, for $\delta > 0$ and $\eta \leq 0$, we define
\begin{equ}[e:DHolderDist]
\Vert \zeta; \zeta^\eps \Vert^{(\eps)}_{\CC^{\delta, \alpha}_{\eta, T}} \eqdef \sup_{t \in (0, T]} \onorm{t}^{-\eta} \Vert \zeta_t; \zeta_t^\eps \Vert^{(\eps)}_{\CC^\alpha} + \sup_{s \neq t \in (0, T]} \onorm{s, t}^{-\eta} \frac{\Vert \delta^{s, t} \zeta; \delta^{s, t}\zeta^\eps \Vert^{(\eps)}_{\CC^{\alpha - \delta}}}{\bigl(|t-s|^{1/\s_0} \vee \eps\bigr)^{\delta}}.
\end{equ}
Furthermore, we define the norm $\Vert \zeta^\eps \Vert^{(\eps)}_{\CC^{\delta, \alpha}_{\eta, T}}$ in the same way as in \eqref{e:AlphaNorm} and \eqref{e:SpaceTimeNorm}, but using the discrete pairing \eqref{e:DPairing}, the quantities $\enorm{t} \eqdef \onorm{t} \vee \eps$ and $\enorm{s, t} \eqdef \enorm{s} \wedge \enorm{t}$ instead of $\onorm{t}$ and $\onorm{s, t}$ respectively, and $|t-s|^{1/\s_0} \vee \eps$ instead of $|t-s|^{1/\s_0}$.

Finally, we denote by $\star$ and $\star_\eps$ the convolutions on $\R^{d+1}$ and $\R \times \Lambda_\eps^d$ respectively, and by $x \lesssim y$ we mean that there exists a constant $C$ independent of the relevant quantities such that $x \leq C y$.

\subsection*{Acknowledgements}

The authors would like to thank Hendrik Weber for valuable discussions of this and related problems, 
Dirk Erhard for pointing out an inaccuracy in the formulation of a previous version 
of Theorem~\ref{t:ModelsConvergence}, as well as Rongchan Zhu and Xiangchan Zhu for pointing out
that the results of \cite{Fel74,Park,BFS83} rely on the smallness of the coupling constant.
MH would like to gratefully acknowledge support by the ERC and the Leverhulme Foundation through
a consolidator grant (number 615897) and a leadership award respectively.


\section{Regularity structures}
\label{s:RegStruct}

In this section we recall the definition of a regularity structure and we introduce the inhomogeneous models
used in this article, 
which are maps from $\R$ (the time coordinate) to the usual space of models as in \cite[Def. 2.17]{Hai14},
endowed with a norm enforcing some amount of time regularity. 
Furthermore, we define inhomogeneous modelled distributions and prove the respective 
reconstruction theorem and Schauder estimates.
Throughout this section, we work with the scaling $\s=(\s_0, 1, \ldots, 1)$ of $\R^{d+1}$,
but all our results can easily be generalised to any non-Euclidean scaling in space, similarly to \cite{Hai14}.

\subsection{Regularity structures and inhomogeneous models}
\label{ss:Models}

The purpose of regularity structures, introduced in \cite{Hai14} and motivated by
\cite{Lyo98,Gub04}, is to generalise Taylor expansions using essentially arbitrary
functions/distributions instead of polynomials. The precise definition is as follows.

\begin{definition}
A {\it regularity structure} $\ST = (\CT, \CG)$ consists of two objects:
\begin{itemize}
\item A {\it model space} $\CT$, which is a graded vector space $\CT = \bigoplus_{\alpha \in \CA} \CT_\alpha$, where each $\CT_\alpha$ is a (finite dimensional in our case) Banach space and $\CA \subset \R$ is a finite set of
``homogeneities''.
\item A {\it structure group} $\CG$ of linear transformations of $\CT$, such that for every $\Gamma~\in~\CG$, every $\alpha \in \CA$ and every $\tau \in \CT_\alpha$ one has
$\Gamma \tau - \tau \in \CT_{< \alpha}$, with $\CT_{< \alpha} \eqdef \bigoplus_{\beta < \alpha} \CT_\beta$. 
\end{itemize}
\end{definition}

In \cite[Def. 2.1]{Hai14}, the set $\CA$ was only assumed to be locally finite and bounded from below. 
Our assumption is more strict, but does not influence anything in the analysis of the equations we consider. 
In addition, our definition rules out the ambiguity of topologies on $\CT$.

\begin{remark}\label{ex:Poly}
One of the simplest non-trivial examples of a regularity structure is given by the ``abstract polynomials'' in $d+1$ 
indeterminates $X_i$, with $i = 0, \ldots, d$. The set $\CA$ in this case consists of the values 
$\alpha \in \N$ such that $\alpha \leq r$, for some $r < \infty$ and, for each $\alpha \in \CA$, the space 
$\CT_\alpha$ contains all monomials in the $X_i$ of scaled degree $\alpha$. 
The structure group $\poly\CG$ is then simply the group of translations in $\R^{d+1}$
acting on $X^k$ by $h \mapsto (X-h)^k$. 

We now fix $r > 0$ to be sufficiently large and denote by $\poly{\CT}$ the 
space of such polynomials of scaled degree $r$ and by $\poly{\CF}$ 
the set $\{X^k\,:\, \sabs{k} \leq r\}$.
We will only ever consider regularity structures containing
$\poly\CT$ as a subspace. In particular, we always assume that there's a natural morphism
$\CG \to \poly\CG$ compatible with the action of $\poly\CG$ on $\poly\CT \hookrightarrow \CT$.  
\end{remark}

\begin{remark}\label{r:QOperators}
For $\tau \in \CT$ we will write $\CQ_\alpha \tau$ for its canonical projection onto $\CT_\alpha$, and define $\Vert \tau \Vert_{\alpha} \eqdef \Vert \CQ_\alpha \tau \Vert$. We also write $\CQ_{<\alpha}$ 
for the projection onto $\CT_{< \alpha}$, etc.
\end{remark}

Another object in the theory of regularity structures is a model. Given an abstract expansion, the model converts it into a concrete distribution describing its local behaviour around every point. We modify the original definition of model in \cite{Hai14}, in order to be able to describe time-dependent distributions.

\begin{definition}\label{d:Model}
Given a regularity structure $\ST = (\CT, \CG)$, an {\it inhomogeneous  model} $(\Pi, \Gamma, \Sigma)$ consists of the following three elements:
\begin{itemize}
\item A collection of maps $\Gamma^t : \R^{d} \times \R^{d} \to \CG$, parametrised by $t \in \R$, such that 
\begin{equ}[e:GammaDef]
\Gamma^t_{x x}=1\;, \qquad \Gamma^t_{x y} \Gamma^t_{y z} = \Gamma^t_{x z}\;,
\end{equ}
for any $x, y, z \in \R^{d}$ and $t \in \R$, and the action of $\Gamma^t_{x y}$ on polynomials is given as in Remark~\ref{ex:Poly} with $h = (0, y-x)$.
\item A collection of maps $\Sigma_x : \R \times \R \to \CG$, parametrized by $x \in \R^d$, such that, for any $x \in \R^{d}$ and $s, r, t \in \R$, one has
\begin{equ}[e:SigmaDef]
\Sigma^{t t}_{x}=1\;, \qquad \Sigma^{s r}_{x} \Sigma^{r t}_{x} = \Sigma^{s t}_{x}\;, \qquad \Sigma^{s t}_{x} \Gamma^{t}_{x y} = \Gamma^{s}_{x y} \Sigma^{s t}_{y}\;,
\end{equ}
and the action of $\Sigma^{s t}_{x}$ on polynomials is given as in Remark~\ref{ex:Poly} with $h = (t-s, 0)$.
\item A collection of linear maps $\Pi^t_x: \CT \to \mathcal{S}'(\R^{d})$, such that
\begin{equ}[e:PiDef]
\Pi^t_{y} = \Pi^t_x \Gamma^t_{x y}\;, \quad \bigl(\Pi_x^t X^{(0, \bar k)}\bigr)(y) = (y-x)^{\bar k}\;, \quad \bigl(\Pi_x^t X^{(k_0, \bar k)}\bigr)(y) = 0\;,
\end{equ}
for all $x, y \in \R^{d}$, $t \in \R$, $\bar{k} \in \N^{d}$, $k_0 \in \N$ such that $k_0 > 0$.
\end{itemize}
Moreover, for any $\gamma > 0$ and every $T > 0$, there is a constant $C$ for which the analytic bounds
\minilab{Model}
\begin{equs}\label{e:PiGammaBound}
| \langle \Pi^t_{x} \tau, \varphi_{x}^\lambda \rangle| \leq C \Vert \tau \Vert \lambda^{l} &\;, \qquad \Vert \Gamma^t_{x y} \tau \Vert_{m} \leq C \Vert \tau \Vert | x-y|^{l - m}\;,\\
\Vert \Sigma^{s t}_{x} \tau \Vert_{m} &\leq C \Vert \tau \Vert |t - s|^{(l - m)/\s_0}\;,\label{e:SigmaBound}
\end{equs}
hold uniformly over all $\tau \in \CT_l$, with $l \in \CA$ and $l < \gamma$, all $m \in \CA$ such that $m < l$, all $\lambda \in (0,1]$, all $\varphi \in \CB^r_0(\R^d)$ with $r > -\lfloor\min \CA\rfloor$, and all $t, s \in [-T, T]$ and $x, y \in \R^d$ such that $|t - s| \leq 1$ and $|x-y| \leq 1$.

In addition, we say that the map $\Pi$ has time regularity $\delta > 0$, if the bound
\begin{equs}
\label{e:PiTimeBound}
| \langle \bigl(\Pi^t_{x} - \Pi^{s}_{x}\bigr) \tau, \varphi_{x}^\lambda \rangle| \leq C \Vert \tau \Vert |t-s|^{\delta/\s_0} \lambda^{l - \delta}\;,
\end{equs}
holds for all $\tau \in \CT_l$ and the other parameters as before.
\end{definition}

\begin{remark}\label{r:ModelNorm}
For a model $Z = (\Pi, \Gamma, \Sigma)$, we denote by $\Vert \Pi \Vert_{\gamma; T}$, $\Vert \Gamma \Vert_{\gamma;T}$ and $\Vert \Sigma \Vert_{\gamma;T}$ the smallest constants $C$ such that the bounds on $\Pi$, $\Gamma$ and $\Sigma$ in \eqref{e:PiGammaBound} and \eqref{e:SigmaBound} hold. Furthermore, we define 
\begin{equ}
\VERT Z \VERT_{\gamma; T} \eqdef \Vert \Pi \Vert_{\gamma; T} + \Vert \Gamma \Vert_{\gamma; T} + \Vert \Sigma \Vert_{\gamma; T}\;.
\end{equ}
If $\bar{Z} = (\bar{\Pi}, \bar{\Gamma}, \bar{\Sigma})$ is another model, then we also define the ``distance'' between two models
\begin{equ}[e:ModelsDist]
\VERT Z; \bar{Z} \VERT_{\gamma; T} \eqdef \Vert \Pi - \bar{\Pi} \Vert_{\gamma; T} + \Vert \Gamma - \bar{\Gamma} \Vert_{\gamma; T} + \Vert \Sigma - \bar{\Sigma} \Vert_{\gamma; T}\;.
\end{equ}
We note that the norms on the right-hand side still make sense with $\Gamma$ and $\Sigma$ viewed
as linear maps on $\CT$. We also set $\Vert \Pi \Vert_{\delta, \gamma; T} \eqdef \Vert \Pi \Vert_{\gamma; T} + C$, where $C$ is the smallest constant such that the bound \eqref{e:PiTimeBound} holds, and we define
\begin{equ}
\VERT Z \VERT_{\delta, \gamma; T} \eqdef \Vert \Pi \Vert_{\delta, \gamma; T} + \Vert \Gamma \Vert_{\gamma; T} + \Vert \Sigma \Vert_{\gamma; T}\;.
\end{equ}
Finally, we define the ``distance'' $\VERT Z; \bar{Z} \VERT_{\delta, \gamma; T}$ as in \eqref{e:ModelsDist}.
\end{remark}

\begin{remark}
In \cite[Def. 2.17]{Hai14} the analytic bounds on a model were assumed to hold locally uniformly. 
In the problems which we aim to consider, the models are periodic in space, which allows us to 
require the bounds to hold globally.
\end{remark}

\begin{remark} \label{r:OriginalModel}
For a given model $(\Pi, \Gamma, \Sigma)$ we can define the following two objects
\begin{equ}[e:ModelTilde]
\bigl(\tilde{\Pi}_{(t,x)} \tau\bigr) (s, y) = \bigl(\Pi_x^s \Sigma_x^{s t} \tau\bigr) (y)\;, \qquad \tilde{\Gamma}_{(t,x), (s,y)} = \Gamma^{t}_{xy} \Sigma_y^{t s} = \Sigma_x^{t s} \Gamma^{s}_{xy} \;,
\end{equ}
for $\tau \in \CT$. Of course, in general we cannot fix the spatial point $y$ in the definition of $\tilde{\Pi}$, and we should really write $\bigl(\bigl(\tilde{\Pi}_{(t,x)} \tau\bigr) (s, \cdot)\bigr)(\varphi) = \bigl(\Pi_x^s \Sigma_x^{s t} \tau\bigr) (\varphi)$ instead, for any test function $\varphi$, but the notation \eqref{e:ModelTilde} is more suggestive. One can then easily verify that the pair $(\tilde{\Pi}, \tilde{\Gamma})$ is a model in the original sense of \cite[Def.~2.17]{Hai14}.
\end{remark}


\subsection{Inhomogeneous modelled distributions}
\label{ss:ModelledDistr}

Modelled distributions represent abstract expansions in the basis of a regularity structure. In order to be able to describe the singularity coming from the behaviour of our
solutions near time $0$, we introduce inhomogeneous modelled distributions which admit 
a certain blow-up as time goes to zero.

Given a regularity structure $\ST = (\CT, \CG)$ with a model $Z=(\Pi, \Gamma, \Sigma)$, values 
$\gamma, \eta \in \R$ and a final time $T > 0$, we consider maps 
$H : (0, T] \times \R^{d} \to \CT_{<\gamma}$ and define
\begin{equs}[e:ModelledDistributionNormAbs]
 \Vert H \Vert_{\gamma, \eta; T} \eqdef \sup_{t \in (0,T]} &\sup_{x \in \R^d} \sup_{l < \gamma} \onorm{t}^{(l - \eta) \vee 0} \Vert H_t(x) \Vert_l\\
 &+ \sup_{t \in (0,T]} \sup_{\substack{x \neq y \in \R^d \\ | x - y | \leq 1}} \sup_{l < \gamma} \frac{\Vert H_t(x) - \Gamma^{t}_{x y} H_t(y) \Vert_l}{\onorm{t}^{\eta - \gamma} | x - y |^{\gamma - l}}\;,
\end{equs}
where $l \in \CA$ in the third supremum. Then the space $\CD^{\gamma, \eta}_T$ consists of all such functions $H$, for which one has
\begin{equ}[e:ModelledDistributionNorm]
\VERT H \VERT_{\gamma, \eta; T} \eqdef \Vert H \Vert_{\gamma, \eta; T} + \sup_{\substack{s \neq t \in (0,T] \\ | t - s | \leq \onorm{t, s}^{\s_0}}} \sup_{x \in \R^d} \sup_{l < \gamma} \frac{\Vert H_t(x) - \Sigma_x^{t s} H_{s}(x) \Vert_l}{\onorm{t, s}^{\eta - \gamma} |t - s|^{(\gamma - l)/\s_0}} < \infty\;.
\end{equ}
The quantities $\onorm{t}$ and $\onorm{t, s}$ used in these definitions were introduced in \eqref{e:SpaceTimeNorm}. Elements of these spaces will be called {\it inhomogeneous modelled distributions}.

\begin{remark}
The norm in \eqref{e:ModelledDistributionNorm} depends on $\Gamma$ and $\Sigma$, but 
does {\it not} depend on $\Pi$; this fact will be crucial in the sequel. 
When we want to stress the dependency on the model, we will also write $\CD^{\gamma, \eta}_T(Z)$.
\end{remark}

\begin{remark}\label{r:ModelledDistrib}
In contrast to the singular modelled distributions from \cite[Def.~6.2]{Hai14}, we do not require the restriction $|x - y| \leq \onorm{t, s}$ in the second term in \eqref{e:ModelledDistributionNormAbs}. This is due to the fact that we consider the space and time variables separately (see the proof of Theorem~\ref{t:Integration}, where this fact is used).
\end{remark}

\begin{remark}\label{r:DistrMult}
Since our spaces $\CD^{\gamma, \eta}_T$ are almost identical to those of \cite[Def.~6.2]{Hai14}, the multiplication and differentiation results from \cite[Sec.~6]{Hai14} hold also for our definition.
\end{remark}

To be able to compare two modelled distributions $H \in \CD^{\gamma, \eta}_T(Z)$ and $\bar{H} \in \CD^{\gamma, \eta}_T(\bar{Z})$, we define the quantities
\minilab{ModelledNorms}
\begin{equs}
\Vert H; \bar{H} \Vert_{\gamma, \eta; T} &\eqdef \sup_{t \in (0,T]} \sup_{x \in \R^d} \sup_{l < \gamma} \onorm{t}^{(l - \eta) \vee 0} \Vert H_t(x) - \bar{H}_t(x) \Vert_l \\
& + \sup_{t \in (0,T]} \sup_{\substack{x \neq y \in \R^d \\ | x - y | \leq 1}} \sup_{l < \gamma} \frac{\Vert H_t(x) - \Gamma^{t}_{x y} H_t(y) - \bar{H}_t(x) + \bar{\Gamma}^{t}_{x y} \bar{H}_t(y) \Vert_l}{\onorm{t}^{\eta - \gamma} | x - y |^{\gamma - l}}\;,\\
\VERT H; \bar{H} \VERT_{\gamma, \eta; T} &\eqdef \Vert H; \bar{H} \Vert_{\gamma, \eta; T}\\
& + \sup_{\substack{s \neq t \in (0,T] \\ | t - s | \leq \onorm{t, s}^{\s_0}}} \sup_{x \in \R^d} \sup_{l < \gamma} \frac{\Vert H_t(x) - \Sigma_x^{t s} H_{s}(x) -  \bar{H}_t(x) + \bar{\Sigma}_x^{t s} \bar{H}_{s}(x) \Vert_l}{\onorm{t, s}^{\eta - \gamma} |t - s|^{(\gamma - l)/\s_0}}\;.
\end{equs}

The ``reconstruction theorem'' is one of the key results of the theory of regularity structures.
Here is its statement in our current framework.

\begin{theorem}\label{t:Reconstruction}
Let $\ST = (\CT, \CG)$ be a regularity structure with $\alpha \eqdef \min \CA < 0$ and $Z = (\Pi, \Gamma, \Sigma)$ be a model. Then, for every $\eta \in \R$, $\gamma > 0$ and $T > 0$, there is a unique family of linear 
operators $\CR_t : \CD_T^{\gamma, \eta}(Z) \to \CC^{\alpha}(\R^d)$, parametrised by $t \in (0,T]$, 
such that the bound
\begin{equ}[e:Reconstruction]
|\langle \CR_t H_t - \Pi^t_x H_t(x), \varphi_x^\lambda \rangle| \lesssim \lambda^\gamma \onorm{t}^{\eta - \gamma} \Vert H \Vert_{\gamma, \eta; T} \Vert \Pi \Vert_{\gamma; T}\;,
\end{equ}
holds uniformly in $H \in \CD^{\gamma, \eta}_T(Z)$, $t \in (0,T]$, $x \in \R^d$, $\lambda \in (0,1]$ and $\varphi \in \CB^r_0(\R^d)$ with $r > -\lfloor \alpha\rfloor$.

If furthermore the map $\Pi$ has time regularity $\delta > 0$, then, for any $\tilde{\delta} \in (0, \delta]$ such that $\tilde{\delta} \leq (m - \zeta)$ for all $\zeta, m \in \left((-\infty, \gamma) \cap \CA\right) \cup \{\gamma\}$ such that $\zeta < m$, the function $t \mapsto \CR_t H_t$ satisfies
\begin{equ}[e:ReconstructBound]
\Vert \CR H \Vert_{\CC^{\tilde{\delta}, \alpha}_{\eta - \gamma, T}} \lesssim \Vert \Pi \Vert_{\delta, \gamma; T} \bigl(1 + \Vert \Sigma \Vert_{\gamma; T} \bigr) \VERT H \VERT_{\gamma, \eta; T}\;.
\end{equ}

Let $\bar{Z} = (\bar{\Pi}, \bar{\Gamma}, \bar{\Sigma})$ be another model for the same regularity structure, and let $\bar{\CR}_t$ be the operator as above, but for the model $\bar Z$. Moreover, let the maps $\Pi$ and $\bar \Pi$ have time regularities $\delta > 0$. Then, for every $H \in \CD^{\gamma, \eta}_T(Z)$ and $\bar{H} \in \CD^{\gamma, \eta}_T(\bar{Z})$, the maps $t \mapsto \CR_t H_t$ and $t \mapsto \bar \CR_t \bar H_t$ satisfy
\begin{equ}[e:ReconstructTime]
\Vert \CR H - \bar \CR \bar H \Vert_{\CC^{\tilde{\delta}, \alpha}_{\eta - \gamma, T}} \lesssim \VERT H; \bar H \VERT_{\gamma, \eta; T} + \VERT Z; \bar{Z} \VERT_{\delta, \gamma; T}\;,
\end{equ}
for any $\tilde{\delta}$ as above, and where the proportionality constant depends on $\VERT H \VERT_{\gamma, \eta; T}$, $\VERT \bar H \VERT_{\gamma, \eta; T}$, $\VERT Z \VERT_{\delta, \gamma; T}$ and $\VERT \bar Z \VERT_{\delta, \gamma; T}$.
\end{theorem}

\begin{proof}
Existence and uniqueness of the maps $\CR_t$, as well as the bound \eqref{e:Reconstruction}, 
follow from \cite[Thm.~3.10]{Hai14}. The uniformity in time in \eqref{e:Reconstruction} follows from the uniformity of the corresponding bounds in \cite[Thm.~3.10]{Hai14}.

To prove that $t \mapsto \CR_t H_t$ belongs to $\CC^{\tilde{\delta}, \alpha}_{\eta - \gamma}\bigl([0,T], \R^d\bigr)$, we will first bound $\langle \CR_t H_t, \varrho_x^\lambda\rangle$, for $\lambda \in (0,1]$, $x \in \R^d$ and $\varrho \in \CB^r_0(\R^d)$. Using \eqref{e:Reconstruction} and properties of $\Pi$ and $H$ we get
\begin{equs}[e:ReconstructMapBound]
|\langle\CR_t H_t, \varrho_x^\lambda\rangle| &\leq |\langle\CR_t H_t - \Pi^t_{x} H_t(x), \varrho_x^\lambda\rangle | + |\langle\Pi^t_{x} H_t(x), \varrho_x^\lambda\rangle|\\
&\lesssim \lambda^{\gamma} \onorm{t}^{\eta - \gamma} + \sum_{\zeta \in [\alpha, \gamma) \cap \CA} \lambda^{\zeta} \onorm{t}^{(\eta - \zeta) \wedge 0} \lesssim \lambda^{\alpha} \onorm{t}^{\eta - \gamma}\;,
\end{equs}
where the proportionality constant is affine in $\Vert H \Vert_{\gamma, \eta; T} \Vert \Pi \Vert_{\gamma; T}$, and $\alpha$ is the minimal homogeneity in $\CA$.

In order to obtain the time regularity of $t \mapsto \CR_t H_t$, we show that the distribution $\zeta_x^{s t} \eqdef \Pi^t_x H_t(x) - \Pi^s_x H_s(x)$ satisfies the bound
\begin{equs}[e:RecosntructTimeProof]
| \langle \zeta_x^{s t} - \zeta_y^{s t}, \varrho_x^{\lambda} \rangle | \lesssim |t-s|^{\tilde{\delta}/\s_0} \onorm{s,t}^{\eta - \gamma} |x-y|^{\gamma - \tilde{\delta} - \alpha} \lambda^\alpha\;,
\end{equs}
uniformly over all $x, y \in \R^d$ such that $\lambda \leq |x-y| \leq 1$, all $s, t \in \R$, and for any value of $\tilde{\delta}$ as in the statement of the theorem. To this end, we consider two regimes: $|x-y| \leq |t-s|^{1/\s_0}$ and $|x-y| > |t-s|^{1/\s_0}$.

In the first case, when $|x-y| \leq |t-s|^{1/\s_0}$, we write, using Definition~\ref{d:Model},
\begin{equ}[e:RecosntructTimeProofFirstCase]
\zeta_x^{s t} - \zeta_y^{s t} = \Pi^t_x \left( H_t(x) - \Gamma^t_{x y} H_t(y)\right) - \Pi^s_x \left( H_s(x) - \Gamma^s_{x y} H_s(y)\right),
\end{equ}
and bound these two terms separately. From the properties \eqref{e:PiGammaBound} and \eqref{e:ModelledDistributionNorm} we get
\begin{equs}
 | \langle \Pi^t_x &\big( H_t(x) - \Gamma^t_{x y} H_t(y)\big), \varrho_x^{\lambda} \rangle | \lesssim \sum_{\zeta \in [\alpha, \gamma) \cap \CA} \lambda^\zeta \Vert H_t(x) - \Gamma^t_{x y} H_t(y) \Vert_\zeta \\
 &\lesssim \sum_{\zeta \in [\alpha, \gamma) \cap \CA} \lambda^\zeta | x - y |^{\gamma - \zeta} \onorm{t}^{\eta - \gamma} \lesssim \lambda^\alpha |x-y|^{\gamma - \alpha} \onorm{t}^{\eta - \gamma}\;,
 \label{e:RecosntructTimeProofFirstCaseBound}
\end{equs}
where we have exploited the condition $|x-y| \geq \lambda$. Recalling now the case we consider, we can bound the last expression by the right-hand side of \eqref{e:RecosntructTimeProof}. The same estimate holds for the second term in \eqref{e:RecosntructTimeProofFirstCase}.

Now, we will consider the case $|x-y| > |t-s|^{1/\s_0}$. In this regime we use the definition of model and write
\begin{equs}
\zeta_x^{s t} - \zeta_y^{s t} &= \big(\Pi^t_x - \Pi^s_x\big) \big(H_t(x) - \Gamma^t_{x y} H_t(y)\big) + \Pi^s_x \big(1 - \Sigma_x^{s t}\big) \big(H_t(x) - \Gamma^t_{x y} H_t(y)\big)\\
&\qquad - \Pi^s_x \big(H_s(x) - \Sigma^{s t}_{x} H_t(x)\big) + \Pi^s_y \big(H_s(y) - \Sigma^{s t}_{y} H_t(y)\big)\;.
\label{e:RecosntructTimeProofSecondCase}
\end{equs}
The first term can be bounded exactly as \eqref{e:RecosntructTimeProofFirstCaseBound}, but using this time \eqref{e:PiTimeBound}, i.e.
\begin{equs}
 | \langle \big(\Pi^t_x - \Pi^s_x\big) \big( H_t(x) - \Gamma^t_{x y} H_t(y)\big), \varrho_x^{\lambda} \rangle | \lesssim \lambda^{\alpha - \delta} |x-y|^{\gamma - \alpha} \onorm{t}^{\eta - \gamma} |t-s|^{\delta/\s_0}\;.
\end{equs}

In order to estimate the second term in \eqref{e:RecosntructTimeProofSecondCase}, we first notice that from \eqref{e:SigmaBound} and \eqref{e:ModelledDistributionNorm} we get
\begin{equs}[e:RecosntructTimeProofSecondCaseSigma]
\Vert \big(1 &- \Sigma_x^{s t}\big) \big(H_t(x) - \Gamma^t_{x y} H_t(y)\big) \Vert_{\zeta} \lesssim \sum_{\zeta < m < \gamma} |t - s|^{(m - \zeta) / \s_0} \Vert H_t(x) - \Gamma^t_{x y} H_t(y) \Vert_{m}\\
&\lesssim \sum_{\zeta < m < \gamma} |t - s|^{(m - \zeta) / \s_0} | x - y |^{\gamma - m} \onorm{t}^{\eta - \gamma} \lesssim |t - s|^{\tilde{\delta}/ \s_0} | x - y |^{\gamma - \tilde{\delta} - \zeta} \onorm{t}^{\eta - \gamma}\;,
\end{equs}
for any $\tilde{\delta} \leq \min_{m > \zeta \in \CA} (m - \zeta)$, where we have used the assumption on the time variables. Hence, for the second term in \eqref{e:RecosntructTimeProofSecondCase} we have
\begin{equs}
| \langle \Pi^s_x \big(1 - \Sigma_x^{s t}\big) \big(H_t(x) &- \Gamma^t_{x y} H_t(y)\big), \varrho_x^{\lambda} \rangle | \\
&\lesssim |t - s|^{\tilde{\delta}/\s_0} \onorm{t}^{\eta - \gamma} \sum_{\zeta < \gamma} \lambda^\zeta | x - y |^{\gamma - \tilde{\delta} - \zeta}\;.
\end{equs}
Since $|x-y| \geq \lambda$ and $\zeta \geq \alpha$, the estimate \eqref{e:RecosntructTimeProof} holds for this expression.

The third term in \eqref{e:RecosntructTimeProofSecondCase} we bound using the properties \eqref{e:PiGammaBound} and \eqref{e:ModelledDistributionNorm} by
\begin{equs}[e:RecosntructTimeProofSecondCaseBound]
| \langle \Pi^s_x (H_s(x) - \Sigma^{s t}_{x} H_t(x)), \varrho_x^{\lambda} \rangle | &\lesssim \sum_{\zeta < \gamma} \lambda^\zeta \Vert H_s(x) - \Sigma^{s t}_{x} H_t(x) \Vert_\zeta \\
&\lesssim \sum_{\zeta < \gamma} \lambda^\zeta |t - s|^{(\gamma - \zeta)/\s_0} \onorm{t, s}^{\eta - \gamma}\;.
\end{equs}
It follows from $|x-y| \geq \lambda$, $|x-y| > |t-s|^{1/\s_0}$ and $\zeta \geq \alpha$, that the latter can be estimated as in \eqref{e:RecosntructTimeProof}, when $\tilde{\delta} \leq \min\{\gamma - \zeta : \zeta \in \CA,\, \zeta < \gamma\}$. The same bound holds for the last term in \eqref{e:RecosntructTimeProofSecondCase}, and this finishes the proof of \eqref{e:RecosntructTimeProof}.

In view of the bound \eqref{e:RecosntructTimeProof} and \cite[Prop.~3.25]{Hai14}, we conclude that 
\begin{equs}[e:ReconstructFullTime]
|\langle \CR_t H_t - \CR_s H_s - \zeta_x^{s t}, \varrho_x^\lambda \rangle| \lesssim |t-s|^{\tilde{\delta}/\s_0} \lambda^{\gamma - \tilde{\delta}} \onorm{s,t}^{\eta - \gamma}\;,
\end{equs}
uniformly over $s, t \in \R$ and the other parameters as in \eqref{e:Reconstruction}. Thus, we can write 
\begin{equs}
\langle \CR_t H_t - \CR_s H_s, \varrho_x^\lambda \rangle = \langle \CR_t H_t - \CR_s H_s - \zeta_x^{s t}, \varrho_x^\lambda \rangle + \langle \zeta_x^{s t}, \varrho_x^\lambda \rangle\;,
\end{equs}
where the first term is bounded in \eqref{e:ReconstructFullTime}. The second term we can write as 
\begin{equs}
\langle \zeta_x^{s t}, \varrho_x^\lambda \rangle = \langle \big(\Pi^t_x - \Pi^s_x\big) H_t(x), \varrho_x^\lambda \rangle &+ \langle \Pi^s_x \big(H_t(x) - \Sigma_x^{t s} H_s(x)\big), \varrho_x^\lambda \rangle\\
&+ \langle \Pi^s_x \big(\Sigma_x^{t s}-1\big) H_s(x), \varrho_x^\lambda \rangle\;,
\end{equs}
which can be bounded by $|t-s|^{\tilde{\delta}/\s_0} \lambda^{\alpha - \tilde{\delta}} \onorm{s,t}^{\eta - \gamma}$, using \eqref{e:PiTimeBound}, \eqref{e:RecosntructTimeProofSecondCaseBound} and \eqref{e:SigmaBound}. Here, in order to estimate the last term, we act similarly to \eqref{e:RecosntructTimeProofSecondCaseSigma}. Combining all these bounds together, we conclude that
\begin{equ}[e:ReconstructTimeReg]
|\langle \CR_t H_t - \CR_s H_s, \varrho_x^\lambda \rangle| \lesssim |t-s|^{\tilde{\delta}/\s_0} \lambda^{\alpha - \tilde{\delta}} \onorm{s,t}^{\eta - \gamma}\;,
\end{equ}
which finishes the proof of the claim.

The bound \eqref{e:ReconstructTime} can be shown in a similar way. More precisely, similarly to \eqref{e:ReconstructMapBound} and using \cite[Eq.~3.4]{Hai14}, we can show that
\begin{equs}
|\langle\CR_t H_t - \bar{\CR}_t \bar{H}_t, \varrho_x^\lambda\rangle| \lesssim \lambda^{\alpha} \onorm{t}^{\eta - \gamma}\bigl(\Vert \Pi \Vert_{\gamma; T} \VERT H; \bar H \VERT_{\gamma, \eta; T} + \Vert \Pi - \bar \Pi \Vert_{\gamma; T} \VERT \bar H \VERT_{\gamma, \eta; T}\bigr).
\end{equs}
Denoting $\bar{\zeta}_x^{s t} \eqdef \bar \Pi^t_x \bar H_t(x) - \bar \Pi^s_x \bar H_s(x)$ and acting as above, we can prove an analogue of \eqref{e:ReconstructFullTime}:
\begin{equs}
| \langle \CR_t H_t - \bar{\CR}_t \bar{H}_t &- \CR_s H_s + \bar{\CR}_s \bar{H}_s - \zeta_x^{s t} + \bar{\zeta}_x^{s t}, \varrho_x^\lambda \rangle| \\
&\lesssim |t-s|^{\tilde{\delta}/\s_0} \lambda^{\gamma - \tilde{\delta}} \onorm{s,t}^{\eta - \gamma} \bigl(\VERT H; \bar H \VERT_{\gamma, \eta; T} + \VERT Z; \bar{Z} \VERT_{\delta, \gamma; T}\bigr)\;,
\end{equs}
with the values of $\tilde{\delta}$ as before. Finally, similarly to \eqref{e:ReconstructTimeReg} we get
\begin{equs}
| \langle \CR_t H_t - \bar{\CR}_t \bar{H}_t - \CR_s H_s + \bar{\CR}_s \bar{H}_s&, \varrho_x^\lambda \rangle| \lesssim |t-s|^{\tilde{\delta}/\s_0} \lambda^{\alpha - \tilde{\delta}} \onorm{s,t}^{\eta - \gamma} \\
&\times \bigl(\VERT H; \bar H \VERT_{\gamma, \eta; T} + \VERT Z; \bar{Z} \VERT_{\delta, \gamma; T}\bigr)\;,
\end{equs}
which finishes the proof.
\end{proof}

\begin{definition}\label{d:Reconstruct}
We will call the map $\CR$, introduced in Theorem~\ref{t:Reconstruction}, the {\it reconstruction operator}, and we will always postulate in what follows that $\CR_t = 0$, for $t \leq 0$.
\end{definition}

\begin{remark}\label{r:OriginalReconstruct}
One can see that the map $\tilde{\CR}(t,\cdot) \eqdef \CR_t(\cdot)$ is the reconstruction operator for the model \eqref{e:ModelTilde} in the sense of \cite[Thm.~3.10]{Hai14}.
\end{remark}


\subsection{Convolutions with singular kernels}

In the definition of a mild solution to a parabolic stochastic PDE, convolutions with singular kernels are involved. In particular Schauder estimates plays a key role. To describe this on the abstract level, we introduce the abstract integration map.

\begin{definition}\label{d:AbstractIntegration}
Given a regularity structure $\ST=(\CT, \CG)$, a linear map $\CI : \CT \to \CT$ is said to be 
an {\it abstract integration map} of order $\beta > 0$ if it satisfies the following properties:
\begin{itemize}
 \item One has $\CI : \CT_m \to \CT_{m + \beta}$, for every $m \in \CA$ such that $m + \beta \in \CA$.
 \item For every $\tau \in \poly{\CT}$, one has $\CI \tau = 0$, where $\poly{\CT} \subset \CT$ contains the polynomial part of $\CT$ and was introduced in Remark~\ref{ex:Poly}.
 \item One has $\CI \Gamma \tau - \Gamma \CI \tau \in \poly{\CT}$, for every $\tau \in \CT$ and $\Gamma \in \CG$.
\end{itemize}
\end{definition}

\begin{remark}
The second and third properties are dictated by the special role played by polynomials in the Taylor expansion. 
One can find a more detailed motivation for this definition in \cite[Sec. 5]{Hai14}.
In general, we also allow for the situation where $\CI$ has a domain which isn't all of $\CT$.
\end{remark}

Now, we will define the singular kernels, convolutions with which we are going to describe.

\begin{definition}\label{d:Kernel}
A function $K : \R^{d + 1} \setminus \{0\} \to \R$ is regularising of order $\beta > 0$, if there is a constant $r > 0$ such that we can decompose
\begin{equ}[e:KernelExpansion]
 K = \sum_{n \geq 0} K^{(n)}\;,
\end{equ}
in such a way that each term $K^{(n)}$ 
is supported in $\{z \in \R^{d+1} : \snorm{z} \leq c 2^{-n} \}$ for some $c > 0$,
satisfies
\begin{equs}[e:KernelBound]
|D^k K^{(n)}(z) | \lesssim 2^{(|\s| - \beta + \sabs{k})n}\;,
\end{equs}
for every multiindex $k$ with $\sabs{k} \leq r$, and annihilates every polynomial
of scaled degree $r$, i.e. for every $k \in \N^{d+1}$ such that $\sabs{k} \leq r$ it satisfies 
\begin{equ}[e:PolyKill]
 \int_{\R^{d+1}} z^k K^{(n)}(z)\, dz = 0\;.
\end{equ}
\end{definition}

Now, we will describe the action of a model on the abstract integration map. When it is convenient for us, we will write $K_t(x) = K(z)$, for $z=(t,x)$.

\begin{definition}
Let $\CI$ be an abstract integration map of order $\beta$ for a regularity structure $\ST=(\CT, \CG)$, let $Z = (\Pi, \Gamma, \Sigma)$ be a model and let $K$ be regularising of 
order $\beta$ with $r > -\lfloor\min \CA\rfloor$. We say that $Z$ realises $K$ for $\CI$, if 
for every $\alpha \in \CA$ and every $\tau \in \CT_\alpha$ one has the identity
\begin{equ}[e:PiIntegral]
\Pi^t_{x} \left(\CI \tau + \CJ_{t, x} \tau \right)(y) = \int_{\R} \langle \Pi^{s}_{x} \Sigma_x^{s t} \tau, K_{t-s}(y - \cdot)\rangle\, ds\;,
\end{equ}
where the polynomial $\CJ_{t, x} \tau$ is defined by
\begin{equs}
\label{e:JDef}
\CJ_{t, x} \tau \eqdef \sum_{\sabs{k} < \alpha + \beta} \frac{X^k}{k!} \int_{\R} \langle\Pi^{s}_{x} \Sigma_x^{s t} \tau, D^k K_{t-s}(x - \cdot)\rangle\, ds\;,
\end{equs}
with $k \in \N^{d+1}$ and the derivative $D^k$ in time-space. Moreover, we require that
\begin{equs}[e:GammaSigmaIntegral]
\Gamma_{x y}^t \big(\CI + \CJ_{t, y}\big) &= \big(\CI + \CJ_{t, x}\big) \Gamma_{x y}^t\;,\\
\Sigma_x^{st} \big(\CI + \CJ_{t, x}\big) &= \big(\CI + \CJ_{s, x}\big) \Sigma_x^{st}\;,
\end{equs}
for all $s, t \in \R$ and $x, y \in \R^d$.
\end{definition}

\begin{remark}
We define the integrals in \eqref{e:PiIntegral} and \eqref{e:JDef} as sums of the same integrals, but using the functions $K^{(n)}$ from the expansion \eqref{e:KernelExpansion}. Since these integrals coincide with those from \cite{Hai14} for the model \eqref{e:ModelTilde}, it follows from \cite[Lem.~5.19]{Hai14} that these sums converge absolutely, and hence the expressions in \eqref{e:PiIntegral} and \eqref{e:JDef} are well defined.
\end{remark}

\begin{remark}
The identities \eqref{e:GammaSigmaIntegral} should be viewed as defining 
$\Gamma_{xy}^t \CI\tau$ and $\Sigma_x^{st} \CI \tau$ in terms of $\Gamma_{xy}^t \tau$,
$\Sigma_x^{st} \tau$, and \eqref{e:JDef}.
\end{remark}

With all these notations at hand we introduce the following operator acting on modelled distribution 
$H \in \CD^{\gamma, \eta}_T(Z)$ with $\gamma + \beta > 0$:
\begin{equ}[e:KDef]
\bigl(\CK_{\gamma} H\bigr)_t(x) \eqdef \CI H_t(x) + \CJ_{t, x} H_t(x) + \bigl(\CN_{\gamma} H\bigr)_t(x)\;.
\end{equ}
Here, the last term is $\poly{\CT}$-valued and is given by
\begin{equ}[e:NDef]
\bigl(\CN_{\gamma} H\bigr)_t(x) \eqdef \sum_{\sabs{k} < \gamma + \beta} \frac{X^k}{k!} \int_{\R} \langle\CR_s H_s - \Pi^{s}_{x} \Sigma_x^{s t} H_t(x), D^k K_{t-s}(x - \cdot)\rangle\, ds\;,
\end{equ} 
where as before $k \in \N^{d+1}$ and the derivative $D^k$ is in time-space, see Definition~\ref{d:Reconstruct} for consistency of notation.

\begin{remark}
It follows from Remark~\ref{r:OriginalReconstruct} and the proof of \cite[Thm.~5.12]{Hai14}, that the integral in \eqref{e:NDef} is well-defined, if we express it as a sum of the respective integrals with the functions $K^{(n)}$ in place of $K$. (See also the definition of the operator $\bf{R}^+$ in \cite[Sec. 7.1]{Hai14}.)
\end{remark}

The modelled distribution $\CK_{\gamma} H$ represents the space-time convolution of $H$ with $K$, and the following result shows that this action ``improves'' regularity by $\beta$.

\begin{theorem}\label{t:Integration}
Let $\ST=(\CT, \CG)$ be a regularity structure with the minimal homogeneity $\alpha$, let $\CI$ be an abstract integration map of an integer order $\beta > 0$, let $K$ be a singular function regularising by $\beta$, and let $Z = (\Pi, \Gamma, \Sigma)$ be a model, which realises $K$ for $\CI$. Furthermore, let $\gamma > 0$, $\eta < \gamma$, $\eta > -\s_0$, $\gamma < \eta + \s_0$, $\gamma + \beta \notin \N$, $\alpha + \beta > 0$ and $r > -\lfloor \alpha \rfloor$, $r > \gamma + \beta$ in Definition~\ref{d:Kernel}.

Then $\CK_\gamma$ maps $\CD^{\gamma, \eta}_T(Z)$ into $\CD^{\bar{\gamma}, \bar{\eta}}_T(Z)$, where $\bar{\gamma} = \gamma + \beta$, $\bar{\eta} = \eta \wedge \alpha + \beta$, and for any $H \in \CD^{\gamma, \eta}_T(Z)$ the following bound holds
\begin{equ}[e:Integration]
\VERT \CK_{\gamma} H \VERT_{\bar{\gamma}, \bar{\eta}; T} \lesssim \VERT H \VERT_{\gamma, \eta; T} \Vert \Pi \Vert_{\gamma; T} \Vert \Sigma \Vert_{\gamma; T} \bigl(1 + \Vert \Gamma \Vert_{\bar{\gamma}; T} + \Vert \Sigma \Vert_{\bar{\gamma}; T}\bigr)\;.
\end{equ}
Furthermore, for every $t \in (0, T]$, one has the identity
\begin{equ}[e:IntegralIdentity]
\CR_t \bigl(\CK_{\gamma} H\bigr)_t(x) = \int_{0}^t \langle\CR_s H_s, K_{t-s}(x - \cdot)\rangle\, ds\;.
\end{equ}

Let $\bar{Z} = (\bar{\Pi}, \bar{\Gamma}, \bar{\Sigma})$ be another model realising $K$ for $\CI$, which satisfies the same assumptions, and let $\bar{\CK}_{\gamma}$ be defined by \eqref{e:KDef} for this model. Then one has
\begin{equ}[e:IntegrationDistance]
\VERT \CK_{\gamma} H; \bar{\CK}_{\gamma} \bar{H} \VERT_{\bar{\gamma}, \bar{\eta}; T} \lesssim \VERT H; \bar{H} \VERT_{\gamma, \eta; T} + \VERT Z; \bar{Z} \VERT_{\bar{\gamma}; T}\;,
\end{equ}
for all $H \in \CD^{\gamma, \eta}_T(Z)$ and $\bar{H} \in \CD^{\gamma, \eta}_T(\bar{Z})$. Here, the proportionality constant depends on $\VERT H \VERT_{\gamma, \eta; T}$, $\VERT \bar{H} \VERT_{\gamma, \eta; T}$ and the norms on the models $Z$ and $\bar{Z}$ involved in the estimate \eqref{e:Integration}.
\end{theorem}

\begin{proof}
In view of Remarks \ref{r:OriginalModel} and \ref{r:OriginalReconstruct}, the required bounds on the components of $(\CK_\gamma H)_t(x)$ and $(\CK_\gamma H)_{t}(x) - \Sigma^{t s}_x (\CK_\gamma H)_{s}(x)$, as well as on the components of $(\CK_{\gamma} H)_t(y) - \Gamma_{y x}^t (\CK_{\gamma} H)_t(x)$ with non-integer homogeneities, can be obtained in exactly the same way as in \cite[Prop.~6.16]{Hai14}. (See the definition of the operator $\bf{R}^+$ in \cite[Sec.~7.1]{Hai14}.)

In order to get the required bounds on the elements of $(\CK_\gamma H)_{t}(x) - \Gamma^{t}_{x y} (\CK_\gamma H)_{t}(y)$ with integer homogeneities, we need to modify the proof of \cite[Prop.~6.16]{Hai14}. The problem is that our definition of modelled distributions is slightly different than the one in \cite[Def.~6.2]{Hai14} (see Remark~\ref{r:ModelledDistrib}). That's why we have to consider only two regimes, $c 2^{-n + 1} \leq |x - y|$ and $c 2^{-n + 1} > |x - y|$, in the proof of \cite[Prop.~6.16]{Hai14}, where $c$ is from Definition \ref{d:Kernel}. The only place in the proof, which requires a special treatment, is the derivation of the estimate
\begin{equ}
\Big|\int_{\R} \langle\CR_s H_s - \Pi^{s}_{x} H_s(x), D^k K^{(n)}_{t-s}(x - \cdot)\rangle\, ds\Big| \lesssim 2^{(\sabs{k} - \gamma - \beta)n} \onorm{t}^{\eta - \gamma}\;,
\end{equ}
which in our case follows trivially from Theorem~\ref{t:Reconstruction} and Definition~\ref{d:Kernel}. Here is the place where we need $\gamma - \eta < \s_0$, in order to have an integrable singularity. Here, we use the same argument as in the proof of \cite[Thm.~7.1]{Hai14} to make sure that the time interval does not increase.

With respective modifications of the proof of \cite[Prop.~6.16]{Hai14} we can also show that \eqref{e:IntegralIdentity} and \eqref{e:IntegrationDistance} hold.
\end{proof}


\section{Solutions to parabolic stochastic PDEs}
\label{s:PDEs}

We consider a general parabolic stochastic PDE of the form
\begin{equ}[e:SPDE]
 \partial_t u = A u + F(u, \xi)\;, \qquad u(0, \cdot) = u_0(\cdot)\;,
\end{equ}
on $\R_+ \times \R^d$, where $u_0$ is the initial data, $\xi$ is a rough noise, $F$ is a function in $u$ and $\xi$, which depends in general on the space-time point $z$ and which is affine in $\xi$, and $A$ is a differential operator such that $\partial_t - A$ has a Green's function $G$, i.e. $G$ is the distributional solution of $(\partial_t - A) G = \delta_0$. Then we require the following assumption to be satisfied.

\begin{assumption}\label{a:Operator}
The operator $A$ is given by $Q(\nabla)$, for $Q$ a homogeneous polynomial on $\R^d$ of some even degree $\beta > 0$. Its Green's function $G : \R^{d+1} \setminus \{0\} \mapsto \R$ is smooth, non-anticipative, i.e. $G_{t} = 0$ for $t \leq 0$, and for $\lambda > 0$ satisfies the scaling relation
\begin{equ}
\lambda^{d} G_{\lambda^{\beta} t}(\lambda x) = G_{t}(x)\;.
\end{equ}
\end{assumption}

\begin{remark}
One can find in \cite{Hor55} precise conditions on $Q$ such that $G$ satisfies Assumption~\ref{a:Operator}.
\end{remark}

In order to apply the abstract integration developed in the previous section, we would like the localised singular part of $G$ to have the properties from Definition~\ref{d:Kernel}. The following result, following from \cite[Lem.~7.7]{Hai14}, shows that this is indeed the case.

\begin{lemma}\label{l:GreenDecomposition}
Let us consider functions $u$ supported in $\R_+ \times \R^d$ and periodic in the spatial variable with some fixed period. If Assumption~\ref{a:Operator} is satisfied with some $\beta > 0$, then we can write $G = K + R$, in such a way that the identity
\begin{equ}
\bigl(G \star u\bigr)(z) = \bigl(K \star u\bigr)(z) + \bigl(R \star u\bigr)(z)\;,
\end{equ}
holds for every such function $u$ and every $z \in (-\infty, 1] \times \R^{d}$, where $\star$ is the space-time convolution. Furthermore, $K$ has the properties from Definition~\ref{d:Kernel} with the parameters $\beta$ and some arbitrary (but fixed) value $r$, and the scaling $\s = (\beta, 1, \ldots, 1)$. The function $R$ is smooth, non-anticipative and compactly supported.
\end{lemma}

In particular, it follows from Lemma~\ref{l:GreenDecomposition} that for any $\gamma > 0$ and any periodic $\zeta_t \in \CC^{\alpha}\big(\R^d\big)$, with $t \in \R$, which is allowed to have an integrable singularity at $t = 0$, we can define
\begin{equ}[e:RDef]
\left(R_{\gamma} \zeta\right)_t(x) \eqdef \sum_{\sabs{k} < \gamma} \frac{X^k}{k!} \int_\R \langle \zeta_s,  D^k R_{t-s}(x - \cdot) \rangle\, ds\;,
\end{equ}
where $k \in \N^{d+1}$ and $D^k$ is taken in time-space.


\subsection{Regularity structures for locally subcritical stochastic PDEs}
\label{ss:RegStruct}

In this section we provide conditions on the equation \eqref{e:SPDE}, under which one can build a regularity structure for it. More precisely, we consider the mild form of equation \eqref{e:SPDE}:
\begin{equ}[e:MildGeneral]
 u = G \star F(u, \xi) + S u_0\;,
\end{equ}
where $\star$ is the space-time convolution, $S$ is the semigroup generated by $A$ and
$G$ is its fundamental solution. We will always assume that we are in a subcritical 
setting, as defined in \cite[Sec.~8]{Hai14}.

It was shown in \cite[Sec.~8.1]{Hai14} that it is possible to build a regularity structure $\ST = (\CT, \CG)$ for a locally subcritical equation and to reformulate it as a fixed point problem in an associated space of modelled distributions. We do not want to give a precise description of this regularity structure, see for example
\cite{Hai14,CDM} for details in the case of $\Phi^4_3$. 
Let us just mention that we can recursively build two sets of symbols, $\CF$ and $\CU$. 
The set $\CF$ contains $\Xi$, $\1$, $X_i$, as well as some of the symbols that can
be built recursively from these basic building blocks by the operations 
\begin{equ}[e:basicOps]
\tau \mapsto \CI(\tau)\;,\qquad (\tau,\bar \tau) \mapsto \tau\bar \tau\;,
\end{equ}
subject to the equivalences $\tau \bar \tau = \bar \tau \tau$, $\1 \tau = \tau$,
and $\CI(X^k) = 0$.
These symbols are involved in the  description of the right hand side of \eqref{e:SPDE}. The set $\CU \subset \CF$ on the other hand contains only those symbols which 
are used in the description of the solution itself, which are either of the form
$X^k$ or of the form $\CI(\tau)$ with $\tau \in \CF$. The 
model space $\CT$ is then defined as $\spanning\{\tau \in \CF\,:\, |\tau| \le r\}$ for a sufficiently large $r > 0$, the set of all (real) linear combinations of  symbols in $\CF$ of homogeneity $|\tau| \le r$, where
$\tau \mapsto |\tau|$ is given by
\begin{equ}[e:defDegree]
|\1| = 0\;,\quad |X_i| = \s_i\;,\quad 
|\Xi| = \alpha, \quad|\CI (\tau)| = |\tau| + \beta\;,\quad 
|\tau \bar \tau| = |\tau| + |\bar \tau|\;.
\end{equ} 
In the situation of interest, namely the $\Phi^4_3$ model, one chooses $\beta = 2$
and $\alpha = -{5\over 2}-\kappa$ for some $\kappa > 0$ sufficiently small. 
Subcriticality then guarantees that $\CT$ is finite-dimensional.
We will also write $\CT_\CU$ for the linear span of $\CU$ in $\CT$.

One can also build a structure group $\CG$ acting on $\CT$ in such a way that
the operation $\CI$ satisfies the assumptions of Definition~\ref{d:AbstractIntegration} 
(corresponding to the convolution operation with the kernel $K$), and such that 
it acts on $\poly{\CT}$ by translations as required.

Let now $Z$ be a model realising $K$ for $\CI$, we denote by $\CR$, $\CK_{\bar{\gamma}}$ and $R_\gamma$ the reconstruction operator, and the corresponding operators \eqref{e:KDef} and \eqref{e:RDef}. 
We also use the notation $\CP \eqdef \CK_{\bar \gamma} + R_\gamma \CR$
for the operator representing convolution with the heat kernel.
With these notations at hand, it was shown in \cite{Hai14} that  
one can associate to \eqref{e:MildGeneral} the fixed point problem
in $\CD^{\gamma, \eta}_T(Z)$ given by
\begin{equ}[e:AbstractEquationBasic]
U = \CP F(U) + Su_0\;,
\end{equ}
for a suitable function (which we call again $F$) 
which ``represents'' the nonlinearity of the SPDE 
in the sense of \cite[Sec.~8]{Hai14} and which is such that
$\CI F(\tau) \in \CT$ for every $\tau \in \CT_\CU$.
In our running example, we would take
\begin{equ}[e:ourF]
F(\tau) = - \CQ_{\le 0} \bigl(a \tau + \lambda \tau^3\bigr) + \Xi\;,
\end{equ}
where $\CQ_{\le 0}$ denotes the canonical projection onto $\CT_{\le 0}$ defined in Remark~\ref{r:QOperators}%
\footnote{The reason for adding this projection is to guarantee that $\CI F$
maps $\CT_\CU$ into $\CT$, since we truncated $\CT$ at homogeneity $r$. Note also that the presence
of this projection does not affect the outcome of the reconstruction operator when applied to $F(U)$.} and $a$ and $\lambda$ are the constants from \eqref{e:Phi}.
The problem we encounter is that since we impose that our models are
functions of time, there exists no model for which $\Pi^t_x \Xi = \xi$ with
$\xi$ a typical realisation of space-time white noise. 
We would like to replace \eqref{e:AbstractEquationBasic} by an equivalent fixed 
point problem that circumvents this problem, and this is the content of the
next two subsections. 


\subsection{Truncation of regularity structures}
\label{sec:trunc}

In general, as just discussed, we cannot always define a suitable inhomogeneous model for the 
regularity structure $\ST = (\CT, \CG)$, so we introduce the 
following truncation procedure, which amounts to simply removing the problematic 
symbols.

\begin{definition}\label{d:TruncSets}
Consider a set of {\it generators} $\gen{\CF} \subset \CF$ such that $\poly{\CF} \subset \gen{\CF}$
and such that $\gen{\CT} \eqdef \spanning\{\tau \in \gen{\CF}\,:\, |\tau| \le r\} \subset \CT$ is closed under the action of $\CG$. 
We then define the corresponding 
{\it generating regularity structure} $\gen{\ST} = (\gen{\CT}, \CG)$.

Moreover, we define $\HF$ as the subset of $\CF$ generated by $\gen{\CF}$ via  
the two operations \eqref{e:basicOps}, and we assume that $\gen{\CF}$ was chosen
in such a way that $\CU \subset \HF$, with $\CU$ as in the previous section. 
Finally, we define the {\it truncated regularity structure} $\hat{\ST} = (\HT, \CG)$
with $\HT \eqdef \spanning\{\tau \in \hat{\CF}\,:\, |\tau| \le r\} \subset \CT$.
\end{definition}

\begin{remark}
Note that $\hat{\ST}$ is indeed a regularity structure since
$\HT$ is automatically closed under $\CG$. This can easily be verified by induction
using the definition of $\CG$ given in \cite{Hai14}.

A set $\gen{\CF}$ with these properties always exists, because one can take either
$\gen{\CF} = \CF$ or $\gen\CF = \{\Xi\} \cup \poly\CF$. In both of these examples,
one simply has 
$\HF = \CF$, but in the case of \eqref{e:Phi}, it turns out to be convenient 
to make a choice for which this is not the case (see Section~\ref{s:DPhi} below).
\end{remark}


\subsection{A general fixed point map}
\label{ss:FixedPointMap}

We now reformulate \eqref{e:SPDE}, with the operator $A$ such that Assumption~\ref{a:Operator} is satisfied, using the regularity structure from the previous section, and show that the
corresponding fixed point problem admits local solutions.
For an initial condition $u_0$ in \eqref{e:SPDE} with ``sufficiently nice'' behavior at infinity, we can define the function $S_t u_0 : \R^d \to \R$, which has a singularity at $t=0$, where as before $S_t$ is the semigroup generated by $A$. In particular, we have a precise description of its singularity, the proof of which is provided in \cite[Lem.~7.5]{Hai14}:

\begin{lemma}\label{l:InitialData}
For some $\eta < 0$, let $u_0 \in \CC^\eta(\R^d)$  be periodic. 
Then, for every $\gamma > 0$
and every $T > 0$, the map $(t,x) \mapsto S_t u_0(x)$ can be lifted to 
$\CD^{\gamma, \eta}_T$ via its Taylor expansion. Furthermore, one has the bound
\begin{equ}[e:InitBound]
\VERT S u_0 \VERT_{\gamma, \eta; T} \lesssim \Vert u_0 \Vert_{\CC^\eta}\;.
\end{equ}
\end{lemma}

Before reformulating \eqref{e:SPDE}, we make some assumptions on its nonlinear term $F$. 
For a regularity structure $\ST = (\CT, \CG)$, let $\hat{\ST} = (\HT, \CG)$ be as in Definition~\ref{d:TruncSets} for a suitable set $\gen{\CF}$. In what follows, we consider models on $\hat{\ST}$ and denote by $\CD^{\gamma, \eta}_{T}$ the respective spaces of modelled distributions. 
We also assume that we are given a function $F\colon \CT_\CU \to \CT$ as above (for example
\eqref{e:ourF}),
and we make the following assumption on $F$.

For some fixed $\bar \gamma > 0$, $\eta \in \R$ we choose, for 
any model $Z$ on $\hat \ST$,
elements $F_0(Z), I_0(Z) \in \CD^{\bar \gamma, \eta}_T(Z)$ such that, for
every $z$, $I_0(z) \in \HT$,
$I_0(z) - \CI F_0(z) \in \poly{\CT}$ and such that, setting
\begin{equ}[e:NonlinAs]
\hat{F}(z, \tau) \eqdef F(z, \tau) - F_0(z) \;,
\end{equ}
$\hat F(z,\cdot)$ maps $\{I_0(z) + \tau : \tau \in \HT \cap \CT_\CU\}$ into $\HT$.
Here we suppressed the argument $Z$ for conciseness by writing 
for example $I_0(z)$ instead of $I_0(Z)(z)$.

\begin{remark}
Since it is the \textit{same} structure group $\CG$ acting on both $\CT$
and $\HT$, the condition $F_0 \in \CD^{\bar \gamma, \eta}_T$ makes sense for
a given model on $\hat\ST$, even
though $F_0(z)$ takes values in all of $\CT$ rather than just $\HT$.
\end{remark}

Given such a choice of $I_0$ and $F_0$ and given $H : \R^{d+1} \to \HT \cap \CT_\CU$, we denote by 
$\hat F(H)$ the function 
\begin{equ}[e:NonlinearTerm]
\bigl(\hat F(H)\bigr)_t(x) \eqdef \hat F\bigl((t, x), H_t(x)\bigr)\;. 
\end{equ}
With this notation, we replace the problem \eqref{e:AbstractEquationBasic} by the problem
\begin{equ}[e:AbstractEquation]
U = \CP \hat F(U) + Su_0 + I_0\;.
\end{equ}
This shows that one should really think of $I_0$ as being given by 
$I_0 = \CP F_0$ since, at least formally, this would then turn \eqref{e:AbstractEquation}
into \eqref{e:AbstractEquationBasic}. 
The advantage of \eqref{e:AbstractEquation} is that it makes sense for any model on
$\hat\ST$ and does not require a model on all of $\ST$. 

We then assume that $\hat F$, $I_0$ and $F_0$ satisfy the following conditions.

\begin{assumption}\label{a:Nonlin}
In the above context, we assume that there exists $\gamma \ge \bar \gamma$ such
that, for every $B > 0$ there exists a constant $C > 0$ such that the bounds
\begin{equs}\label{e:Lipschitz}
 \VERT \hat F(H); \hat F(\bar H) \VERT_{\bar{\gamma}, \bar{\eta}; T} &\leq C \left( \VERT H; \bar H \VERT_{\gamma, \eta; T} + \VERT Z; \bar{Z} \VERT_{\gamma; T} \right),\\
 \VERT I_0(Z); I_0(\bar Z) \VERT_{\bar{\gamma}, \bar{\eta}; T} \leq C \VERT Z; \bar{Z}& \VERT_{\gamma; T}\;,\quad 
 \VERT F_0(Z); F_0(\bar Z) \VERT_{\bar{\gamma}, \bar{\eta}; T} \leq C \VERT Z; \bar{Z} \VERT_{\gamma; T}\;,
\end{equs}
hold for any two models $Z$, $\bar{Z}$ with $\VERT Z \VERT_{\gamma; T} + \VERT \bar{Z} \VERT_{\gamma; T} \leq B$, and for $H \in \CD^{\gamma, \eta}_T(Z)$, $\bar H \in \CD^{\gamma, \eta}_T(\bar{Z})$ such that $\VERT H \VERT_{\gamma, \eta; T} + \VERT \bar H \VERT_{\gamma, \eta; T} \leq B$.
\end{assumption}

\begin{remark}
The bounds in Assumption~\eqref{a:Nonlin} can usually be easily checked for a polynomial nonlinearity $F$ in \eqref{e:MildGeneral}. See Lemma~\ref{l:PhiLipschiz} below for a respective prove in the case when $F$ is give by \eqref{e:ourF}.
\end{remark}

The following theorem provides the existence and uniqueness results of a local solution to this equation.

\begin{theorem}\label{t:FixedMap}
In the described context, let $\alpha \eqdef \min \hat{\CA}$, and an abstract integration map $\CI$ be of order $\beta > -\alpha$. Furthermore, let the values $\gamma \geq \bar{\gamma} > 0$ and $\eta, \bar{\eta} \in \R$ from Assumption~\ref{a:Nonlin} satisfy $\eta < \bar{\eta} \wedge \alpha + \beta$, $\gamma < \bar{\gamma} + \beta$ and $\bar{\eta} > -\beta$.
 
Then, for every model $Z$ as above, and for every periodic $u_0 \in \CC^\eta(\R^d)$, there exists a time $T_\star \in (0, +\infty]$ such that, for every $T < T_\star$ the equation \eqref{e:AbstractEquation} admits a unique solution $U \in \CD^{\gamma, \eta}_T(Z)$.
Furthermore, if $T_\star < \infty$, then
\begin{equ}
\lim_{T \to T_\star} \Vert \CR_T \CS_{T} (u_0, Z)_T \Vert_{\CC^\eta} = \infty\;,
\end{equ}
where $\mathcal{S}_T : (u_0, Z) \mapsto U$ is the solution map. Finally, for every $T < T_\star$, the solution map $\mathcal{S}_T$ is jointly Lipschitz continuous in a neighbourhood around $(u_0, Z)$ in the sense that, for any $B > 0$ there is $C > 0$ such that, if $\bar{U} = \mathcal{S}_T(\bar{u}_0, \bar{Z})$ for some initial data $(\bar{u}_0, \bar{Z})$, then one has the bound $\VERT U; \bar{U} \VERT_{\gamma, \eta; T} \leq C \delta$, provided $\Vert u_0 - \bar{u}_0 \Vert_{\CC^{\eta}} + \VERT Z; \bar{Z} \VERT_{\gamma; T} \leq \delta$, for any $\delta \in (0,B]$.
\end{theorem}

\begin{proof}
See \cite[Thm.~7.8]{Hai14}, combined with \cite[Prop.~7.11]{Hai14}.
Note that since we consider inhomogeneous models, we have no problems
in evaluating $\CR_t U_t$.
\end{proof}

\begin{definition}\label{d:SPDESolution}
In the setting of Theorem~\ref{t:FixedMap}, let $U$ be the unique solution to the equation~\eqref{e:AbstractEquation} on $[0, T_\star)$. Then for $t < T_\star$ we define the solution to \eqref{e:SPDE} by
\begin{equ}[e:SPDESolution]
u_t(x) \eqdef \bigl(\CR_t U_t\bigr)(x)\;.
\end{equ}
\end{definition}

\begin{remark}\label{r:ClassicalSolution}
If the noise $\xi$ in \eqref{e:SPDE} is smooth, so that this equation can be solved in the classical sense, one can see that the reconstruction operator satisfies 
\begin{equ}
\bigl(\CR_t U_t\bigr)(x) = \bigl(\Pi_x^t U_t(x)\bigr)(x)\;, 
\end{equ}
and the solution \eqref{e:SPDESolution} coincides with the classical solution.
\end{remark}


\section{Discrete models and modelled distributions}
\label{s:DModels}

In order to be able to consider discretisations of the equations whose solutions were provided in Section~\ref{s:PDEs}, we introduce the discrete counterparts of inhomogeneous models and modelled distributions. In this section we use the following notation: for $N \in \N$, we denote by $\eps \eqdef 2^{-N}$ the mesh size of the grid $\Lambda_\eps^d \eqdef \bigl(\eps \Z\bigr)^d$, and we fix some scaling $\s=(\s_0, 1, \ldots, 1)$ of $\R^{d+1}$ with an integer $\s_0 > 0$.


\subsection{Definitions and the reconstruction theorem}

Now we define discrete analogues of the objects from Sections~\ref{ss:Models} and \ref{ss:ModelledDistr}. 

\begin{definition}\label{d:DModel}
Given a regularity structure $\ST$ and $\eps>0$, a {\it discrete model} $(\Pi^\eps, \Gamma^\eps, \Sigma^\eps)$ consists of the collections of maps 
\begin{equ}
\Pi_x^{\eps, t}: \CT \to \R^{\Lambda_\eps^d}\;, \qquad \Gamma^{\eps, t} : \Lambda^d_\eps \times \Lambda^d_\eps \to \CG\;, \qquad \Sigma^{\eps}_x : \R \times \R \to \CG\;,
\end{equ}
parametrised by $t \in \R$ and $x \in \Lambda_\eps^d$, which have all the algebraic properties of their continuous counterparts in Definition~\ref{d:Model}, with the spatial variables restricted to the grid. Additionally, we require $\bigl(\Pi^{\eps, t}_x \tau\bigr) (x) = 0$, for all $\tau \in \CT_l$ with $l > 0$, and all $x \in \Lambda^d_\eps$ and $t \in \R$.
\end{definition}

We define the quantities $\Vert \Pi^\eps \Vert^{(\eps)}_{\gamma; T}$ and $\Vert \Gamma^\eps \Vert_{\gamma;T}^{(\eps)}$ to be the smallest constants $C$ such that the bounds \eqref{e:PiGammaBound} hold uniformly in $x, y \in \Lambda^d_\eps$, $t \in \R$, $\lambda \in [\eps,1]$ and with the discrete pairing \eqref{e:DPairing} in place of the standard one. The quantity $\Vert \Sigma^\eps \Vert_{\gamma;T}^{(\eps)}$ is defined as the smallest constant $C$ such that the bounds
\begin{equ}[e:DSigmaBound]
 \Vert \Sigma^{\eps, s t}_{x} \tau \Vert_{m} \leq C \Vert \tau \Vert \bigl(|t - s|^{1/\s_0} \vee \eps\bigr)^{l - m}\;,
\end{equ}
hold uniformly in $x \in \Lambda_\eps^d$ and the other parameters as in \eqref{e:SigmaBound}.

We measure the time regularity of $\Pi^\eps$ as in \eqref{e:PiTimeBound}, by substituting the continuous objects by their discrete analogues, and by using $|t - s|^{1/\s_0} \vee \eps$ instead of $|t - s|^{1/\s_0}$ on the right-hand side. All the other quantities $\Vert \cdot \Vert^{(\eps)}$, $\VERT \cdot \VERT^{(\eps)}$, etc.\ are defined by analogy with Remark~\ref{r:ModelNorm}.

\begin{remark}
The fact that $\bigl(\Pi^{\eps, t}_x \tau\bigr) (x) = 0$ if $|\tau| > 0$ does not follow  
automatically from the discrete analogue of \eqref{e:PiGammaBound} since these
are only assumed to hold for test functions at scale $\lambda \ge \eps$. We use this 
property in the proof of \eqref{e:DIntegralIdentity}.
\end{remark}

\begin{remark}
The weakening of the continuity property of $\Sigma^{\eps,st}_x$ given by \eqref{e:DSigmaBound} 
will be used in the analysis of the ``discrete abstract integration'' in Section~\ref{ss:DConvols}. 
It allows us to deal with the fact that the discrete heat kernel is discontinuous at $t=0$,
so we simply use uniform bounds on very small time scales 
(see \cite[Lem.~6.7]{HMW12} for a simple explanation in a related context).
\end{remark}

For $\gamma, \eta \in \R$ and $T > 0$, for a discrete model $Z^\eps=(\Pi^\eps, \Gamma^\eps, \Sigma^\eps)$ on a regularity structure $\ST = (\CT, \CG)$, and for a function $H^\eps : (0, T] \times \Lambda_\eps^d \to \CT_{<\gamma}$, we define
\begin{equs}[e:DModelledDistributionNormAbs]
 \Vert H^\eps \Vert^{(\eps)}_{\gamma, \eta; T} \eqdef \sup_{t \in (0,T]} &\sup_{x \in \Lambda_\eps^d} \sup_{l < \gamma} \enorm{t}^{(l - \eta) \vee 0} \Vert H^\eps_t(x) \Vert_l\\
 &+ \sup_{t \in (0,T]} \sup_{\substack{x \neq y \in \Lambda^d_\eps \\ | x - y | \leq 1}} \sup_{l < \gamma} \frac{\Vert H^\eps_t(x) - \Gamma^{\eps, t}_{x y} H^\eps_t(y) \Vert_l}{\enorm{t}^{\eta - \gamma} | x - y |^{\gamma - l}}\;,
\end{equs}
where $l \in \CA$. Furthermore, we define the norm
\begin{equ}[e:DModelledDistributionNorm]
\VERT H^\eps \VERT^{(\eps)}_{\gamma, \eta; T} \eqdef \Vert H^\eps \Vert^{(\eps)}_{\gamma, \eta; T} + \sup_{\substack{s \neq t \in (0,T] \\ | t - s | \leq \onorm{t, s}^{\s_0}}} \sup_{x \in \Lambda_\eps^d} \sup_{l < \gamma} \frac{\Vert H^\eps_t(x) - \Sigma_x^{\eps, t s} H^\eps_{s}(x) \Vert_l}{\enorm{t, s}^{\eta - \gamma} \bigl(|t - s|^{1/s_0} \vee \eps\bigr)^{\gamma - l}}\;,
\end{equ}
where the quantities $\enorm{t}$ and $\enorm{t, s}$ are defined below \eqref{e:DHolderDist}. We will call such functions $H^\eps$ {\it discrete modelled distributions}.

\begin{remark}\label{r:DDistrMult}
It is easy to see that the properties of multiplication of modeled distributions from \cite[Sec.~6.2]{Hai14} can be translated mutatis mutandis to the discrete case.
\end{remark}

In contrast to the continuous case, a reconstruction operator of discrete modeled distributions can be defined in a simple way.

\begin{definition}\label{d:DReconstruct}
Given a discrete model $Z^\eps = (\Pi^\eps, \Gamma^\eps, \Sigma^\eps)$ and a discrete modelled distribution $H^\eps$ we define the {\it discrete reconstruction map} $\CR^{\eps}$ by $\CR^{\eps}_t = 0$ for $t \leq 0$, and
\begin{equ}[e:DReconstructDef]
\big(\CR^{\eps}_t H^\eps_t\big)(x) \eqdef \big(\Pi_x^{\eps, t} H^\eps_t(x) \big)(x)\;, \qquad (t, x) \in (0, T] \times \Lambda_\eps^d\;.
\end{equ}
\end{definition}

Recalling the definition of the norms from \eqref{e:DHolderDist}, the
following result is a discrete analogue of Theorem~\ref{t:Reconstruction}.

\begin{theorem}
\label{t:DReconstruct}
Let $\ST$ be a regularity structure with $\alpha \eqdef \min \CA < 0$ and $Z^\eps = (\Pi^\eps, \Gamma^\eps, \Sigma^\eps)$ be a discrete model. Then the bound
\begin{equ}
|\langle \CR^\eps_t H^\eps_t - \Pi^{\eps, t}_x H^\eps_t(x), \rho_x^\lambda \rangle_\eps| \lesssim \lambda^\gamma \enorm{t}^{\eta - \gamma} \Vert H^\eps \Vert^{(\eps)}_{\gamma, \eta; T} \Vert \Pi^\eps \Vert^{(\eps)}_{\gamma; T}\;,
\end{equ}
holds uniformly in $\eps$ (see Remark~\ref{r:Uniformity} below) for all discrete modelled distributions $H^\eps$, all $t \in (0,T]$, $x \in \Lambda^d_\eps$, $\rho \in \CB^r_0(\R^d)$ with $r > -\lfloor \alpha \rfloor$, all $\lambda \in [\eps, 1]$.

Let furthermore $\bar{Z}^\eps = (\bar{\Pi}^\eps, \bar{\Gamma}^\eps, \bar{\Sigma}^\eps)$ be another model for $\ST$ with the reconstruction operator $\bar{\CR}^\eps_t$, and let the maps $\Pi^\eps$ and $\bar \Pi^\eps$ have time regularities $\delta > 0$. Then, for any two discrete modelled distributions $H^\eps$ and $\bar{H}^\eps$, the maps $t \mapsto \CR^\eps_t H^\eps_t$ and $t \mapsto \bar \CR^\eps_t \bar H^\eps_t$ satisfy
\minilab{e:DReconstructBounds}
\begin{equs}
\Vert \CR^\eps H^\eps \Vert^{(\eps)}_{\CC^{\tilde{\delta}, \alpha}_{\eta - \gamma, T}} &\lesssim \Vert \Pi^\eps \Vert^{(\eps)}_{\delta, \gamma; T} \bigl(1 + \Vert \Sigma^\eps \Vert^{(\eps)}_{\gamma; T} \bigr) \VERT H^\eps \VERT^{(\eps)}_{\gamma, \eta; T}\;,\label{e:DReconstructSpace}\\
\Vert \CR^\eps H^\eps - \bar \CR^\eps \bar H^\eps \Vert^{(\eps)}_{\CC^{\tilde{\delta}, \alpha}_{\eta - \gamma, T}} &\lesssim \VERT H^\eps; \bar H^{\eps} \VERT^{(\eps)}_{\gamma, \eta; T} + \VERT Z^\eps; \bar Z^\eps \VERT^{(\eps)}_{\delta, \gamma; T}\;,\label{e:DReconstructTime}
\end{equs}
for any $\tilde{\delta}$ as in Theorem~\ref{t:Reconstruction}. Here, the norms of $H^\eps$ and $\bar H^\eps$ are defined via the models $Z^\eps$ and $\bar Z^\eps$ respectively, and the proportionality constants depend on $\eps$ only via $\VERT H^\eps \VERT^{(\eps)}_{\gamma, \eta; T}$, $\VERT \bar H^\eps \VERT^{(\eps)}_{\gamma, \eta; T}$, $\VERT Z^\eps \VERT^{(\eps)}_{\delta, \gamma; T}$ and $\VERT \bar Z^\eps \VERT^{(\eps)}_{\delta, \gamma; T}$.
\end{theorem}

\begin{remark}\label{r:Uniformity}
 In the statement of Theorem~\ref{t:DReconstruct} and the following results we actually consider a sequence of discrete models and modeled distributions parametrised by $\eps = 2^{-N}$ with $N \in \N$. By ``uniformity in $\eps$'' we then mean that the estimates hold for all values of $\eps$ with a proportionality constant independent of $\eps$.
\end{remark}

\begin{remark}
 To compare a discrete model $Z^\eps = (\Pi^\eps, \Gamma^\eps, \Sigma^\eps)$ to a continuous model $Z = (\Pi, \Gamma, \Sigma)$, we can define
 \begin{equs}
\Vert \Pi&; \Pi^\eps \Vert^{(\eps)}_{\delta, \gamma; T} \eqdef \sup_{\varphi, x, \lambda, l, \tau} \sup_{t \in [-T, T]} \lambda^{-l} | \langle \Pi^t_{x} \tau, \varphi_{x}^\lambda \rangle - \langle \Pi^{\eps, t}_{x} \tau, \varphi_{x}^\lambda \rangle_\eps|\\
&+ \sup_{\varphi, x, \lambda, l, \tau} \sup_{\substack{s \neq t \in [-T, T] \\ |t-s| \leq 1}} \lambda^{-l + \delta} \frac{| \langle \bigl(\Pi^t_{x} - \Pi^s_{x}\bigr) \tau, \varphi_{x}^\lambda \rangle - \langle \bigl(\Pi^{\eps, t}_{x} - \Pi^{\eps, s}_{x}\bigr) \tau, \varphi_{x}^\lambda \rangle_\eps|}{\bigl(|t-s|^{1/\s_0} \vee \eps\bigr)^{\delta}}\;,
\end{equs}
where the supremum is taken over $\varphi \in \CB^r_0$, $x \in \Lambda_\eps^d$, $\lambda \in [\eps, 1]$, $l < \gamma$ and $\tau \in \CT_l$ with $\Vert \tau \Vert = 1$. In order to compare discrete and continuous modelled distributions, we use the quantities as in \eqref{ModelledNorms}, but with the respective modifications as in \eqref{e:DModelledDistributionNorm}.
 
 Then one can show similarly to \eqref{e:ReconstructTime} that for $H \in \CD^{\gamma, \eta}_T(Z)$ and a discrete modeled distribution $H^\eps$ the maps $t \mapsto \CR_t H_t$ and $t \mapsto \CR^\eps_t H^\eps_t$ satisfy the estimate
\begin{equ}
\Vert \CR H; \CR^\eps H^\eps \Vert^{(\eps)}_{\CC^{\tilde{\delta}, \alpha}_{\eta - \gamma, T}} \lesssim \VERT H; H^\eps \VERT^{(\eps)}_{\gamma, \eta; T} + \VERT Z; Z^\eps \VERT^{(\eps)}_{\delta, \gamma; T} + \eps^\theta\;,
\end{equ}
for $\tilde \delta > 0$ and $\theta > 0$ small enough. We will however not make use of this
in the present article.
\end{remark}

In order to prove Theorem~\ref{t:DReconstruct}, we need to introduce a multiresolution analysis and its discrete analogue.


\subsubsection{Elements of multiresolution analysis}
\label{ss:MultiresolutionAnalysis}

In this section we provide only the very basics of the multiresolution analysis, which are used in the sequel. For a more detailed introduction and for the proofs of the provided results we refer to \cite{Dau92} and \cite{Mey92}.

One of the remarkable results of \cite{Dau88} is that for every $r > 0$ there exists a compactly supported function $\varphi \in \CC^r(\R)$ (called {\it scaling function}) such that
\begin{equ}[e:ScaleFunctionsOrthonormality]
\int_{\R} \varphi(x)\, dx = 1\;, \qquad \int_{\R} \varphi(x) \varphi(x + k)\,dx = \delta_{0, k}, \quad k \in \Z\;,
\end{equ}
where $\delta_{\cdot, \cdot}$ is the Kronecker's delta on $\Z$. Furthermore, if for $n \in \N$ we define the grid $\Lambda_n \eqdef \{2^{-n} k : k \in \Z\}$ and the family of functions
\begin{equ}[e:FatherScale]
\varphi_x^n(\cdot) \eqdef 2^{n / 2} \varphi\bigl(2^{n} (\cdot - x)\bigr)\;, \qquad x \in \Lambda_n\;,
\end{equ}
then there is a finite collection of vectors $\CK \subset \Lambda_1$ and a collection of structure constants $\{ a_k : k \in \CK \}$ such that the {\it refinement equation}
\begin{equ}[e:FatherRelation]
\varphi_x^n = \sum_{k \in \CK} a_k \varphi_{x + 2^{-n}k}^{n+1}
\end{equ}
holds. Note that the multiplier in \eqref{e:FatherScale} preserves the $L^2$-norm of the scaled functions
rather than their $L^1$-norm. It follows immediately from \eqref{e:ScaleFunctionsOrthonormality} and \eqref{e:FatherRelation} that  one has the identities
\begin{equ}[e:StructRelations]
\sum_{k \in \CK} a_k = 2^{d/2}\;, \qquad \sum_{k \in \CK} a_k a_{k + m} = \delta_{0, m}\;, \quad m \in \Z^d\;.
\end{equ}

For a fixed scaling function $\varphi$, we denote by $V_n \subset L^2(\R)$ the subspace spanned by $\{ \varphi_x^n : x \in \Lambda_n \}$. Then the relation \eqref{e:FatherRelation} ensures the inclusion $V_n \subset V_{n+1}$ for every $n$. It turns out that there is a compactly supported function $\psi \in \CC^r(\R)$ (called {\it wavelet function}) such that the space $V_n^\perp$, which is the orthogonal complement of $V_n$ in $V_{n+1}$, is given by
\begin{equ}
V_n^\perp = \mathrm{span} \{ \psi_x^n : x \in \Lambda_n \}\;,
\end{equ}
where $\psi_x^n$ is as in \eqref{e:FatherRelation}. Moreover, there are constants $\{b_k : k \in \CK\}$, such that the {\it wavelet equation} holds:
\begin{equ}[e:MotherRelation]
\psi_x^n = \sum_{k \in \CK} b_k \varphi_{x + 2^{-n}k}^{n+1}\;.
\end{equ}
One more useful property of the wavelet function is that it has vanishing moments, in the sense that the identity
\begin{equ}[e:MotherKiller]
 \int_{\R} \psi(x)\, x^m dx = 0
\end{equ}
holds for all $m \in \N$ such that $m \leq r$.

There is a standard generalization of scaling and wavelet functions to $\R^d$, namely for $n \geq 0$ and $x = (x_1, \ldots, x_d) \in \Lambda^d_n$ we define
\begin{equ}
\varphi_x^n(y) \eqdef \varphi^n_{x_1}(y_1) \cdots \varphi^n_{x_d}(y_d)\;, \qquad y = (y_1, \ldots, y_d) \in \R^d\;.
\end{equ}
For these scaling functions we also define $V_n$ as the closed subspace in $L^2$ spanned 
by $\{ \varphi_x^n : x \in \Lambda^d_n \}$. Then there is a finite set $\PPsi$ of functions on $\R^d$ such that the space $V_n^\perp \eqdef V_{n+1} \setminus V_n$ is a span of $\{ \psi_x^n : \psi \in \PPsi,\, x \in \Lambda^d_n \}$, where we define the scaled function $\psi_x^n$ by
\begin{equ}
\psi_x^n(y) \eqdef 2^{n d/2} \psi\bigl(2^n (y_1 - x_1), \ldots, 2^n (y_d - x_d)\bigr)\;.
\end{equ}
All the results mentioned above can be literally translated from $\R$ to $\R^d$, but of course with $\CK \subset \Lambda_1^d$ and with different structure constants $\{a_k : k \in \CK\}$ and $\{b_k : k \in \CK\}$.


\subsubsection{An analogue of the multiresolution analysis on the grid}
\label{ss:DMultiresolutionAnalysis}

In this section we will develop an analogue of the multiresolution analysis which will be useful for working with functions defined on a dyadic grid. Our construction agrees with
the standard discrete wavelets on gridpoints, but also extends off the grid.
To this end, we use the notation of Section~\ref{ss:MultiresolutionAnalysis}. We recall furthermore that we use $\eps = 2^{-N}$ for a fixed $N \in \N$.

Let us fix a scaling function $\varphi \in \CC^r_0(\R)$, for some integer $r > 0$, as in Section~\ref{ss:MultiresolutionAnalysis}. For integers $0 \leq n \leq N$ we define the functions
\begin{equ}[e:DFather]
\varphi^{N,n}_x(\cdot) \eqdef 2^{N d / 2} \langle \varphi^N_{\cdot}, \varphi^{n}_x \rangle\;, \qquad x \in \Lambda^d_n\;.
\end{equ}
One has that $\varphi^{N,n}_x \in \CC^r(\R^d)$, it is supported
in a ball of radius $\CO(2^{-n})$ centered at $x$, it has the same scaling properties as $\varphi^{n}_x$, and it satisfies
\begin{equ}[e:DFatherProperty]
\varphi_x^{N,N}(y) = 2^{N d / 2} \delta_{x,y}\;, \qquad \quad x, y \in \Lambda^d_N\;,
\end{equ}
where $\delta_{\cdot, \cdot}$ is the Kronecker's delta on $\Lambda^d_N$. The last property follows from \eqref{e:ScaleFunctionsOrthonormality}. Furthermore, it follows from \eqref{e:FatherRelation} that for $n < N$ these functions satisfy the refinement identity
\begin{equ}[e:DFatherRelation]
\varphi^{N,n}_x = \sum_{k \in \CK} a_k\, \varphi^{N,n+1}_{x + 2^{-n}k}\;,
\end{equ}
with the same structure constants $\{ a_k : k \in \CK \}$ as for the functions $\varphi^{n}_x$. One more consequence of \eqref{e:ScaleFunctionsOrthonormality} is
\begin{equ}
2^{-Nd} \sum_{y \in \Lambda^d_N} \varphi_x^{N,n}(y) = 2^{- n d / 2}\;,
\end{equ}
which obviously holds for $n = N$, and for $n < N$ it can be proved by induction, using \eqref{e:DFatherRelation} and \eqref{e:StructRelations}.

The functions $\varphi_x^{N,n}$ inherit many of the crucial properties of the functions $\varphi^{n}_x$, which allows us to use them in the multiresolution analysis. In particular, for $n < N$ and $\psi \in \PPsi$ (the set of wavelet functions, introduced in Section~\ref{ss:MultiresolutionAnalysis}), we can define the functions
\begin{equ}
\psi^{N,n}_x(\cdot) \eqdef 2^{N d / 2} \langle \varphi^N_{\cdot}, \psi^n_x \rangle\;, \qquad x \in \Lambda^d_n\;,
\end{equ}
whose properties are similar to those of $\psi^n_x$. For example, $\psi^{N,n}_x \in \CC^r(\R)$, and it has the same scaling and support properties as $\psi^{n}_x$. Furthermore, it follows from \eqref{e:MotherRelation} that for $n < N$ the following identity holds
\begin{equ}[e:DMotherRelation]
\psi_x^{N,n} = \sum_{k \in \CK} b_k \varphi_{x + 2^{-n}k}^{N, n+1}\;,
\end{equ}
with the same constants $\{b_k : k \in \CK\}$. It is easy to see that the functions just introduced are not $L^2$-orthogonal, but still, using \eqref{e:StructRelations}, one can go by induction from $N$ to any $n < N$ and prove the following result:

\begin{proposition}
In the context just described, for every integer $n \in [0, N)$, the set
\begin{equ}
\{ \varphi_x^{N, n} : x \in \Lambda_n \} \cup \{ \psi_x^{N, m} : m \in [n, N), \,x \in \Lambda_m \}\;,
\end{equ}
forms an orthonormal basis of $\ell^2(\Lambda_\eps)$ equipped with the inner product $\langle \cdot, \cdot \rangle_\eps$.
\end{proposition}

A generalisation of this discrete analogue of the wavelet analysis to higher dimensions can be done by analogy with the continuous case in Section~\ref{ss:MultiresolutionAnalysis}.


\subsubsection{Proof of the discrete reconstruction theorem}

With the help of the discrete analogue of the multiresolution analysis introduced in the previous section we are ready to prove Theorem~\ref{t:DReconstruct}.

\begin{proof}[Proof of Theorem~\ref{t:DReconstruct}]
We take a compactly supported scaling function $\varphi \in \CC^r(\R^d)$ of regularity $r > -\lfloor \alpha\rfloor$, where $\alpha$ is as in the statement of the theorem, and build the functions $\varphi^{N,n}_x$ as in \eqref{e:DFather}. Furthermore, we define the discrete functions $\zeta_{x}^{\eps, t} \eqdef \Pi^{\eps, t}_{x} H^\eps_t(x)$ and $\zeta_{x y}^{\eps, t} \eqdef \zeta_{y}^{\eps, t} - \zeta_{x}^{\eps, t}$. Then from Definition~\ref{d:DModel} we obtain
\begin{equs}
\bigl|\langle \zeta_{x y}^{\eps, t}, \varphi_y^{N,n} \rangle_{\eps}\bigr| &\lesssim \Vert \Pi^\eps \Vert^{(\eps)}_{\gamma; T} \sum_{l \in [\alpha, \gamma) \cap \CA} 2^{-n d / 2 - l n} \Vert H^\eps_t(y) - \Gamma_{yx}^{\eps, t} H^\eps_t(x) \Vert_{l} \\
&\lesssim \Vert \Pi^\eps \Vert^{(\eps)}_{\gamma; T} \Vert H^\eps \Vert^{(\eps)}_{\gamma, \eta; T} \enorm{t}^{\eta - \gamma} \sum_{l \in [\alpha, \gamma) \cap \CA} 2^{-n d / 2 - l n} |y - x|^{\gamma - l}\\
&\lesssim \Vert \Pi^\eps \Vert^{(\eps)}_{\gamma; T} \Vert H^\eps \Vert^{(\eps)}_{\gamma, \eta; T} \enorm{t}^{\eta - \gamma} 2^{- n d / 2 - \alpha n} |y - x|^{\gamma - \alpha}\;,\label{e:ReconstructIntermBound}
\end{equs}
which holds as soon as $|x -y| \geq 2^{-n}$. Moreover, we define
\begin{equ}
\CR_t^{\eps,n} H^\eps_t \eqdef \sum_{y \in \Lambda^d_n} \langle \zeta_{y}^{\eps, t}, \varphi_y^{N,n} \rangle_\eps\, \varphi_y^{N,n}\;.
\end{equ}
It follows from the property \eqref{e:DFatherProperty} that $\CR_t^{\eps} H^\eps_t = \CR_t^{\eps,N} H^\eps_t$ and $\Pi^{\eps, t}_{x} H^\eps_t(x) = \projw_{\eps,N} (\zeta_{x}^{\eps,t})$ (recall that $\eps = 2^{-N}$), where the operator $\projw_{\eps,n}$ is defined by
\begin{equ}
\projw_{\eps,n} (\zeta) \eqdef \sum_{y \in \Lambda^d_n} \langle \zeta, \varphi_y^{N,n} \rangle_\eps\, \varphi_y^{N,n}\;.
\end{equ}
This allows us to choose $n_0 \geq 0$ to be the smallest integer such that $2^{-n_0} \leq \lambda$ and rewrite
\begin{equs}
\CR_t^{\eps} H^\eps_t &- \Pi^{\eps, t}_{x} H^\eps_t(x) = \left( \CR^{\eps,n_0}_t H^\eps_t - \projw_{\eps,n_0} (\zeta_{x}^{\eps, t}) \right) \label{e:DReconstructExpansion}\\
&+ \sum_{n = n_0}^{N-1} \left( \CR_t^{\eps,n+1} H^\eps_t - \projw_{\eps,n+1} (\zeta_{x}^{\eps,t}) - \CR_t^{\eps,n} H^\eps_t + \projw_{\eps,n} (\zeta_{x}^{\eps, t}) \right).
\end{equs}
The first term on the right hand side yields
\begin{equ}[e:DReconstructExpansionFirst]
\langle \CR^{\eps,n_0}_t H^\eps_t - \projw_{\eps,n_0} (\zeta_{x}^{\eps, t}), \rho_x^\lambda \rangle_\eps = \sum_{y \in \Lambda^d_{n_0}} \langle \zeta_{x y}^{\eps, t}, \varphi_y^{N,n_0}\rangle_\eps\, \langle \varphi_y^{N,n_0}, \rho_x^\lambda \rangle_\eps\;.
\end{equ}
Using \eqref{e:ReconstructIntermBound} and the bound $|\langle \varphi_y^{N,n_0}, \rho_x^\lambda \rangle_\eps| \lesssim 2^{n_0 d/2}$, we obtain
\begin{equs}
\bigl| \langle \CR^{\eps,n_0}_t H^\eps_t - \projw_{\eps,n_0} (\zeta_{x}^{\eps, t}), \rho_x^\lambda\rangle_\eps\bigr| \lesssim \Vert \Pi^\eps \Vert^{(\eps)}_{\gamma; T} \Vert H^\eps \Vert^{(\eps)}_{\gamma, \eta; T} \enorm{t}^{\eta - \gamma} 2^{ - \gamma n_0}\;.
\end{equs}
Here, we have also used $|x-y| \lesssim 2^{-n_0}$ in the sum in \eqref{e:DReconstructExpansionFirst}, and the fact that only a finite number of points $y \in \Lambda^d_{n_0}$ contribute to this sum.

Now we will bound each term in the sum in \eqref{e:DReconstructExpansion}. Using \eqref{e:DFatherRelation} and \eqref{e:DMotherRelation}, we can write 
\begin{equ}
\CR_t^{\eps,n+1} H^\eps_t - \projw_{\eps,n+1} (\zeta_{x}^{\eps,t}) - \CR_t^{\eps,n} H^\eps_t + \projw_{\eps,n} (\zeta_{x}^{\eps, t}) = g^\eps_{t, n} + h^\eps_{t,n}\;, 
\end{equ}
where $g^\eps_{t, n}$ is defined by
\begin{equ}
g^\eps_{t, n} = \sum_{y \in \Lambda_n^d} \sum_{k \in \CK} a_k \langle \zeta^{\eps, t}_{y, y+2^{-n} k}, \varphi_{y+2^{-n} k}^{N, n+1}\rangle_\eps\, \varphi_y^{N, n}
\end{equ}
and the constants $\{a_k : k \in \CK\}$ are from \eqref{e:DFatherRelation}. For $h^\eps_{t,n}$ we have the identity
\begin{equs}[e:deltaFExpansion]
h^\eps_{t,n} = \sum_{y \in \Lambda_{n+1}^d} \sum_{k \in \CK} \sum_{\psi \in \PPsi} b_k \langle \zeta^{\eps, t}_{x y} , \varphi_{y}^{N, n+1}\rangle_\eps\, \psi_{y - 2^{-n} k}^{N,n}\;.
\end{equs}
Moreover, the following bounds, for $n \in [n_0, N]$, follow from the properties of the functions $\varphi^n_x$ and $\psi^n_x$:
\begin{equ}
|\langle \varphi_y^{N,n}, \varrho_x^\lambda \rangle_\eps| \lesssim 2^{n_0 d/2} 2^{-(n-n_0) d/2}\;, \quad |\langle \psi_y^{N,n}, \varrho_x^\lambda \rangle_\eps| \lesssim 2^{n_0 d/2} 2^{-(n-n_0) (r + d/2)}\;.
\end{equ}
Using them and \eqref{e:ReconstructIntermBound}, we obtain a bound on $g^\eps_{t, n}$:
\begin{equs}
|\langle g^\eps_{t, n}, \rho_x^\lambda \rangle_\eps| &\lesssim \sum_{y \in \Lambda_n^d} \sum_{k \in \CK} |\langle \zeta^{\eps, t}_{y, y+2^{-n} k} , \varphi_{y+2^{-n} k}^{N, n+1}\rangle_\eps|\, |\langle \varphi_y^{N,n}, \rho_x^\lambda \rangle_\eps|\\
&\lesssim \Vert \Pi^\eps \Vert^{(\eps)}_{\gamma; T} \Vert H^\eps \Vert^{(\eps)}_{\gamma, \eta; T} \enorm{t}^{\eta - \gamma} 2^{ - \gamma n}\;,
\end{equs}
where we have used $|x-y| \lesssim 2^{-n}$ in the sum. Summing these bounds over $n \in [n_0, N]$, we obtain a bound of the required order. Similarly, we obtain the following bound on \eqref{e:deltaFExpansion}:
\begin{equs}
|\langle h^\eps_{t, n}, \rho_x^\lambda \rangle_\eps| \lesssim \Vert \Pi^\eps \Vert^{(\eps)}_{\gamma; T} \Vert H^\eps \Vert^{(\eps)}_{\gamma, \eta; T} \enorm{t}^{\eta - \gamma} 2^{ - \gamma n_0} 2^{-(n-n_0)(r + \alpha)}\;,
\end{equs}
which gives the required bound after summing over $n \in [n_0, N]$. In this estimate we have used the fact that $|y-x| \lesssim 2^{-n_0}$ in the sum in \eqref{e:deltaFExpansion}.

The bounds \eqref{e:DReconstructBounds} can be shown similarly to \eqref{e:ReconstructBound} and \eqref{e:ReconstructTime}. 
\end{proof}


\subsection{Convolutions with discrete kernels}
\label{ss:DConvols}

In this section we describe on the abstract level convolutions with discrete kernels. We start with a definition of the kernels we will work with.

\begin{definition}\label{d:DKernel}
We say that a function $K^{\eps} : \R \times \Lambda^d_\eps \to \R$ is regularising of order $\beta > 0$, 
if one can find functions $K^{(\eps, n)} : \R^{d+1} \to \R$ and $\mathring{K}^{\eps} : \R \times \Lambda^d_\eps \to \R$ such that
\begin{equ}[e:DEpsExpansion]
 K^{\eps} = \sum_{n = 0}^{N-1} K^{(\eps, n)} + \mathring{K}^{\eps} \eqdef \bar K^\eps + \mathring{K}^{\eps}\;,
\end{equ}
where the function $K^{(\eps, n)}$ has the same support and bounds as the function $K^{(n)}$ in Definition~\ref{d:Kernel}, for some $c, r > 0$, and furthermore, for $k \in \N^{d+1}$ such that $\sabs{k} \leq r$, it satisfies 
\begin{equ}[e:DPolyKill]
 \int_{\R \times \Lambda_\eps^d} z^k K^{(\eps, n)}(z)\, dz = 0\;.
\end{equ}
The function $\mathring{K}^{\eps}$ is supported in $\{z \in \R \times \Lambda_\eps^d : \snorm{z} \leq c \eps \}$ and satisfies \eqref{e:DPolyKill} with $k = 0$ and
\begin{equs}[e:KZeroBounds]
\sup_{z \in \R \times \Lambda^d_\eps}|\mathring{K}^{\eps}(z) | \leq C \eps^{-|\s| + \beta}\;.
\end{equs}
\end{definition}

Now, we will define how a discrete model acts on an abstract integration map.

\begin{definition}\label{d:DIntegralModel}
Let $\CI$ be an abstract integration map of order $\beta$ as in Definition~\ref{d:AbstractIntegration} for a regularity structure $\ST=(\CT, \CG)$, let $Z^\eps=(\Pi^\eps, \Gamma^\eps, \Sigma^\eps)$ be a discrete model, and let $K^\eps$ be regularising of order $\beta$ with $r > -\lfloor \min \CA \rfloor$. Let furthermore $\bar{K}^\eps$ and $\mathring{K}^\eps$ be as in~\eqref{e:DEpsExpansion}. We define $\bar{\CJ}^\eps$ on the grid in the same way as its continuous analogue in~\eqref{e:JDef}, but using $\bar{K}^\eps$ instead of $K$ and using the discrete objects instead of their continuous counterparts. Moreover, we define
\begin{equs}
\mathring{\CJ}^\eps_{t,x} \tau \eqdef \1 \int_{\R} \langle \Pi^{\eps, s}_{x} \Sigma_x^{\eps, s t} \tau, \mathring{K}^\eps_{t-s}(x - \cdot) \rangle_\eps\, ds\;,
\end{equs}
and $\CJ^\eps_{t,x} \eqdef \bar{\CJ}^\eps_{t,x} + \mathring{\CJ}^\eps_{t,x}$. We say that $Z^\eps$ realises $K^\eps$ for $\CI$ if the identities \eqref{e:PiIntegral} and \eqref{e:GammaSigmaIntegral}
hold for the corresponding discrete objects.
As before, these two identities should be thought of as providing the definitions
of $\Gamma_{xy}^{\eps,t} \CI \tau$ and $\Sigma_x^{\eps,st} \CI \tau$ via $\Gamma_{xy}^{\eps,t} \tau$ and $\Sigma_x^{\eps,st} \tau$.
\end{definition}

For a discrete modelled distribution $H^\eps$, we define $\bar \CN^\eps_{\gamma} H^\eps$ as in \eqref{e:NDef}, but using the discrete objects instead of the continuous ones, and using the kernel $\bar{K}^\eps$ instead of $K$. Furthermore, we define the term containing $\mathring{K}^\eps$ by
\begin{equ}[e:NZeroDef]
\bigl(\mathring{\CN}^\eps_{\gamma} H^\eps\bigr)_t(x) \eqdef \1 \int_{\R} \langle \CR^\eps_s H^\eps_s - \Pi^{\eps, s}_{x} \Sigma_x^{\eps, s t} H^\eps_t(x), \mathring{K}^\eps_{t-s}(x - \cdot) \rangle_\eps\, ds\;,
\end{equ}
and we set $\CN^\eps_{\gamma} H^\eps \eqdef \bar{\CN}^\eps_{\gamma} H^\eps + \mathring{\CN}^\eps_{\gamma} H^\eps$. Finally, we define the discrete analogue of \eqref{e:KDef} by
\begin{equ}[e:KEpsDef]
\bigl(\CK^\eps_{\gamma} H^\eps\bigr)_t(x) \eqdef \CI H^\eps_t(x) + \CJ^\eps_{t, x} H^\eps_t(x) + \bigl(\CN^\eps_{\gamma} H^\eps\bigr)_t(x)\;.
\end{equ}
Our definition is consistent thanks to the following two lemmas.

\begin{lemma}\label{l:DPiIntegralBound}
 In the setting of Definition~\ref{d:DIntegralModel}, let $\min\CA + \beta > 0$. Then all the algebraic relations of Definition~\ref{d:DModel} hold for the symbol $\CI \tau$. Moreover, for $\delta > 0$ sufficiently small and for any $\l \in \CA$ and $\tau \in \CT_l$ such that $l + \beta \notin \N$ and $\Vert \tau \Vert = 1$, one has the bounds
\begin{equs}
| \langle \Pi^{\eps, t}_{x} \CI \tau, \varphi_{x}^\lambda \rangle_\eps| &\lesssim \lambda^{l + \beta} \Vert \Pi^\eps \Vert_{\l; T}^{(\eps)} \Vert \Sigma^\eps \Vert^{(\eps)}_{l; T} \bigl( 1 + \Vert \Gamma^\eps \Vert^{(\eps)}_{l; T} \bigr)\;,\label{e:PiIntegralBound}\\
\frac{| \langle \bigl(\Pi^{\eps, t}_{x} - \Pi^{\eps, s}_{x}\bigr) \CI \tau, \varphi_{x}^\lambda \rangle_\eps|}{\bigl(|t-s|^{1/\s_0} \vee \eps\bigr)^{\delta}} &\lesssim \lambda^{l + \beta - \delta} \Vert \Pi^\eps \Vert_{\delta, l; T}^{(\eps)} \Vert \Sigma^\eps \Vert^{(\eps)}_{l; T} \bigl( 1 + \Vert \Gamma^\eps \Vert^{(\eps)}_{l; T} \bigr)\;,\label{e:PiIntegralTimeBound}
\end{equs}
uniformly over $\eps$ (see Remark~\ref{r:Uniformity}), $x \in \Lambda_\eps^d$, $s, t \in [-T, T]$, $\lambda \in [\eps, 1]$ and $\varphi \in \CB_0^r(\R^d)$.
\end{lemma}

\begin{proof}
The algebraic properties of the models for the symbol $\CI \tau$ follow easily from Definition~\ref{d:DIntegralModel}. In order to prove \eqref{e:PiIntegralBound}, we will consider the terms in \eqref{e:PiIntegral} containing $\mathring{K}^\eps$ separately from the others. To this end, we define
\begin{equs}
\bigl(\mathring{\Pi}^{\eps,t}_{x} \CI \tau \bigr)(y) &\eqdef \int_{\R} \langle \Pi^{\eps, s}_{x} \Sigma_x^{\eps, s t} \tau, \mathring{K}^{\eps}_{t-s}(y - \cdot) - \mathring{K}^{\eps}_{t-s}(x - \cdot)\rangle_\eps\, ds\;,\label{e:PiRingDef}\\
\bigl(\bar{\Pi}^{\eps,t}_{x} \CI \tau \bigr)(y) &\eqdef \bigl(\Pi^{\eps,t}_{x} - \mathring{\Pi}^{\eps,t}_{x} \bigr) \bigl(\CI \tau \bigr)(y)\;.
\end{equs}
Furthermore, for $x, y \in \Lambda_\eps^d$ we use the assumption $0^0 \eqdef 1$ and set
\begin{equ}
 T^l_{x y} K^{(\eps, n)}_{t}(\cdot) \eqdef K^{(\eps, n)}_{t}(y - \cdot) - \sum_{\sabs{k} < l + \beta} \frac{(0, y-x)^k}{k!} D^k K^{(\eps, n)}_{t}(x - \cdot)\;.
\end{equ}

Using Definitions~\ref{d:DModel} and \ref{d:DKernel} and acting as in the proof of \cite[Lem.~5.19]{Hai14}, we can obtain the following analogues of the bounds \cite[Eq.~5.33]{Hai14}:
\begin{equs}[e:DPiInterBounds]
| \langle \Pi^{\eps, r}_{x} \Sigma^{\eps, r t}_x \tau, T^l_{x y} K^{(\eps, n)}_{t-r} \rangle_\eps| &\lesssim \sum_{\zeta > 0} |y-x|^{l + \beta + \zeta} 2^{(\s_0 + \zeta) n} \1_{|t-r| \lesssim 2^{-\s_0 n}}\;,\\
\Big|\int_{\Lambda_\eps^d} \langle \Pi^{\eps, r}_{x} \Sigma^{\eps, r t}_x \tau, T^l_{x y} K^{(\eps, n)}_{t-r} \rangle_\eps&\, \varphi^\lambda_{x}(y) \, dy\Big| \lesssim \sum_{\zeta > 0} \lambda^{l + \beta - \zeta} 2^{(\s_0 - \zeta) n} \1_{|t-r| \lesssim 2^{-\s_0 n}}\;,
\end{equs}
for $\eps \leq |y-x| \leq 1$, $\lambda \in [\eps, 1]$, with $\zeta$ taking a finite number of values and with the proportionality constants as in \eqref{e:PiIntegralBound}. Integrating these bounds in the time variable $r$ and using the first bound in \eqref{e:DPiInterBounds} in the case $|y-x| \leq 2^{-n}$ and the second bound in the case $2^{-n} \leq \lambda$, we obtain the required estimate on $\langle \bar \Pi^{\eps, t}_{x} \CI \tau, \varphi_{x}^\lambda \rangle_\eps$.

In order to bound $\bigl(\bar \Pi^{\eps, t}_{x} - \bar \Pi^{\eps, s}_{x}\bigr) \CI \tau$, we consider two cases $|t-s| \geq 2^{-\s_0 n}$ and $|t-s| < 2^{-\s_0 n}$. In the first case we estimate $\bar \Pi^{\eps, t}_{x} \CI \tau$ and $\bar \Pi^{\eps, s}_{x} \CI \tau$ separately using \eqref{e:DPiInterBounds}, and obtain the required bound, if $\delta > 0$ is sufficiently small. In the case $|t-s| < 2^{-\s_0 n}$ we write
\begin{equs}
\langle& \Pi^{\eps, r}_{x} \Sigma^{\eps, r t}_x \tau, T^l_{x y} K^{(\eps, n)}_{t-r} \rangle_\eps - \langle \Pi^{\eps, r}_{x} \Sigma^{\eps, r s}_x \tau, T^l_{x y} K^{(\eps, n)}_{s-r} \rangle_\eps\\
&= \langle \Pi^{\eps, r}_{x} \Sigma^{\eps, r s}_x \bigl(\Sigma^{\eps, s t}_x - 1\bigr) \tau, T^l_{x y} K^{(\eps, n)}_{t-r} \rangle_\eps + \langle \Pi^{\eps, r}_{x} \Sigma^{\eps, r s}_x \tau, T^l_{x y} \bigl(K^{(\eps, n)}_{t-r} - K^{(\eps, n)}_{s-r}\bigr) \rangle_\eps\;,
\end{equs}
and estimate each of these terms similarly to \eqref{e:DPiInterBounds}, which gives the required bound for sufficiently small $\delta > 0$.

It is only left to prove the required bounds for $\mathring{\Pi}^{\eps,t}_{x} \bigl(\CI \tau \bigr)$. It follows immediately from Definition~\ref{d:DModel} that $|\bigl(\Pi^{\eps, t}_x a\bigr)(x)| \lesssim \Vert a \Vert \eps^{\zeta}$, for $a \in \CT_\zeta$. Hence, using the properties \eqref{e:SigmaDef} and \eqref{e:PiDef} we obtain
\begin{equs}
\int_{\R} \bigl|\langle \Pi^{\eps, s}_{x} \Sigma_x^{\eps, s t} \tau, \mathring{K}^{\eps}_{t-s}(y - \cdot)\rangle_\eps\bigr|\, ds &= \int_{\R} \bigl|\langle \Pi^{\eps, s}_{y} \Sigma_y^{\eps, s t} \Gamma_{y x}^{\eps, t} \tau, \mathring{K}^{\eps}_{t-s}(y - \cdot)\rangle_\eps\bigr|\, ds\\
&\lesssim \sum_{\zeta \leq l} \eps^{\zeta + \beta} |y-x|^{l - \zeta}\;,\label{e:KRingBound}
\end{equs}
where $\zeta \in \CA$. Similarly, the second term in \eqref{e:PiRingDef}
is bounded by $\eps^{l +\beta}$, implying that if $\lambda \geq \eps$ and $\min\CA + \beta > 0$, then one has
\begin{equs}[e:PiRingIntegralBound]
| \langle \mathring{\Pi}^{\eps, t}_{x} \CI \tau, \varphi_{x}^\lambda \rangle_\eps| \lesssim \sum_{\zeta \leq l} \eps^{\zeta + \beta} \lambda^{l - \zeta} \lesssim \lambda^{l + \beta}\;,
\end{equs}
which finishes the proof of \eqref{e:PiIntegralBound}. In order to complete the proof of \eqref{e:PiIntegralTimeBound}, we use \eqref{e:KRingBound} and brutally bound
\begin{equs}
| \langle \bigl(\mathring{\Pi}^{\eps, t}_{x} &- \mathring{\Pi}^{\eps, s}_{x}\bigr) \CI \tau, \varphi_{x}^\lambda \rangle_\eps| \leq | \langle \mathring{\Pi}^{\eps, t}_{x} \CI \tau, \varphi_{x}^\lambda \rangle_\eps| + | \langle \mathring{\Pi}^{\eps, s}_{x} \CI \tau, \varphi_{x}^\lambda \rangle_\eps| \\
&\lesssim \sum_{\zeta \leq l} \eps^{\zeta + \beta} |y-x|^{l - \zeta} \lesssim \bigl(|t-s|^{1/\s_0} \vee \eps\bigr)^{\tilde{\delta}} \sum_{\zeta \leq l} \eps^{\zeta + \beta - \tilde{\delta}} |y-x|^{l - \zeta}\;,
\end{equs}
from which we obtain the required bound in the same way as before, as soon as $\delta \in (0, \min \CA + \beta)$.
\end{proof}

The following lemma provides a relation between $\CJ^\eps$ and the operators $\Gamma^\eps$, $\Sigma^\eps$.

\begin{lemma}\label{l:DGammaIntegralBound}
In the setting of Lemma~\ref{l:DPiIntegralBound}, the operators 
 \begin{equs}[e:DSigmaDiff]
\CJ^{\eps, t}_{x y} \eqdef \CJ^\eps_{t, x} \Gamma^{\eps, t}_{x y} - \Gamma^{\eps, t}_{x y} \CJ^\eps_{t, y}\;,\qquad \CJ^{\eps, s t}_{x} \eqdef \CJ^\eps_{s, x} \Sigma^{\eps, s t}_{x} - \Sigma^{\eps, s t}_{x} \CJ^\eps_{t, x}\;,
\end{equs}
with $s, t \in \R$ and $x, y \in \Lambda_\eps^d$, satisfy the following bounds:
\begin{equs}
\bigl| \bigl(\CJ^{\eps, t}_{x y} \tau\bigr)_k \bigr| &\lesssim \Vert \Pi^\eps \Vert_{\l; T}^{(\eps)} \Vert \Sigma^\eps \Vert^{(\eps)}_{l; T} \bigl( 1 + \Vert \Gamma^\eps \Vert^{(\eps)}_{l; T} \bigr) |x - y|^{l + \beta - |k|_{\s}}\;,\\
\bigl| \bigl(\CJ^{\eps, s t}_{x} \tau\bigr)_k \bigr| &\lesssim \Vert \Pi^\eps \Vert_{\l; T}^{(\eps)} \Vert \Sigma^\eps \Vert^{(\eps)}_{l; T} \bigl( 1 + \Vert \Gamma^\eps \Vert^{(\eps)}_{l; T} \bigr) \bigl(|t - s|^{1/\s_0} \vee \eps\bigr)^{l + \beta - |k|_{\s}}\;,\label{e:DSigmaDiffBound}
\end{equs}
uniformly in $\eps$ (see Remark~\ref{r:Uniformity}), for $\tau$ as in Lemma~\ref{l:DPiIntegralBound}, for any $k \in \N^{d+1}$ such that $|k|_{\s} < l + \beta$, and for $(\cdot)_k$ being the multiplier of $X^k$. In particular, the required bounds on $\Gamma^\eps \CI\tau$ and $\Sigma^\eps \CI\tau$ from Definition~\ref{d:DModel} hold.
\end{lemma}

\begin{proof}
The bounds on the parts of $\CJ^{\eps, t}_{x y} \tau$ and $\CJ^{\eps, s t}_{x} \tau$ not containing $\mathring{K}^\eps$ can be obtained as in \cite[Lem.~5.21]{Hai14}, where the bound on the right-hand side of \eqref{e:DSigmaDiffBound} comes from the fact that the scaling of the kernels $K^{(\eps, n)}$ in \eqref{e:DEpsExpansion} does not go below $\eps$. The contributions to \eqref{e:DSigmaDiff} from the kernel $\mathring{K}^\eps$ come via the terms $\mathring{\CJ}^\eps_{t, x} \Gamma^{\eps, t}_{x y}$, $\mathring{\CJ}^\eps_{t, y}$, $\mathring{\CJ}^\eps_{s, x} \Sigma^{\eps, s t}_{x}$ and $\mathring{\CJ}^\eps_{t, x}$. We can bound all of them separately, similarly to \eqref{e:KRingBound}, and use $|x-y| \geq \eps$ and $|t - s|^{1/\s_0} \vee \eps \geq \eps$ to estimate the powers of $\eps$. Since all of these powers are positive by assumption, this yields the required bounds.

Now, we will prove the bound on $\Gamma^\eps \CI\tau$ required by Definition~\ref{d:DModel}. For $m < l$ such that $m \notin \N$, \eqref{e:GammaSigmaIntegral} yields
\begin{equ}
\Vert \Gamma_{x y}^{\eps, t} \CI \tau \Vert_m = \Vert \CI \bigl(\Gamma_{x y}^{\eps, t} \tau\bigr) \Vert_m \leq \Vert \Gamma_{x y}^{\eps, t} \tau \Vert_{m - \beta} \lesssim |y-x|^{l + \beta - m}\;,
\end{equ}
where we have used the properties of $\CI$. Similarly, we can bound $\Vert \Sigma_{x}^{\eps, s t} \CI \tau \Vert_m$. Furthermore, since the map $\CI$ does not produce elements of integer homogeneity, we have for $m \in \N$,
\begin{equs}
\Vert \Gamma_{x y}^{\eps, t} \CI \tau \Vert_m = \Vert \CJ^{\eps, t}_{x y} \Vert_m \lesssim |y-x|^{l + \beta - m}\;,
\end{equs}
where the last bound we have proved above. In the same way we can obtain the required bound on $\Vert \Sigma_{x}^{\eps, s t} \CI \tau \Vert_m$.
\end{proof}

\begin{remark}\label{r:ModelLift}
If $(\Pi^{\eps}, \Gamma^{\eps}, \Sigma^{\eps})$ is a discrete model on $\gen{\ST}$, which is introduced in Definition~\ref{d:TruncSets}, then there is a canonical way to extend it to a discrete model on $\hat{\ST}$. Since the symbols from $\HF$ are ``generated'' by $\gen{\CF}$, we only have to define the actions of $\Pi^{\eps}$, $\Gamma^{\eps}$ and $\Sigma^{\eps}$ on the symbols $\tau \bar{\tau}$ and $\CI\tau \in \HF \setminus \gen{\CF}$ with $\tau, \bar{\tau} \in \HF$, so that the extension of the model to $\hat{\ST}$ will follow by induction. For the product $\tau \bar{\tau}$, we set
\minilab{CanonicalProduct}
\begin{equs}
\big(\Pi^{\eps,t}_{x} \tau \bar{\tau}\big) (y) = \bigl(\Pi^{\eps,t}_{x} \tau\bigr) (y)\,& \big(\Pi^{\eps,t}_{x} \bar{\tau}\big) (y)\;,\label{e:CanonicalPiProduct}\\
\Sigma_x^{\eps, s t} \tau \bar{\tau} = \big(\Sigma_x^{\eps, s t} \tau\big)\, \big(\Sigma_x^{\eps, s t} \bar{\tau}\big)\;,\qquad \Gamma_{x y}^{\eps, t}& \tau \bar{\tau} = \big(\Gamma_{x y}^{\eps, t} \tau\big)\, \big(\Gamma_{x y}^{\eps, t} \bar{\tau}\big)\;.\label{e:CanonicalGammaSigmaProduct}
\end{equs}
For the symbol $\CI \tau$ we define the actions of the maps $(\Pi^{\eps}, \Gamma^{\eps}, \Sigma^{\eps})$ by the identities \eqref{e:PiIntegral} and \eqref{e:GammaSigmaIntegral}. 
However, even if the family of models satisfy analytic bounds uniformly in $\eps$ on $\gen{\ST}$,
this is not necessarily true for its extension to $\hat{\ST}$.
\end{remark}

The structure of the canonical extension of a discrete model will be important for us. That is why we make the following definition.

\begin{definition}
We call a discrete model $Z^\eps = (\Pi^{\eps}, \Gamma^{\eps}, \Sigma^{\eps})$ defined on $\hat{\ST}$ {\it admissible}, if it satisfies the identities \eqref{e:CanonicalGammaSigmaProduct} and furthermore
realises $K^\eps$ for $\CI$.
\end{definition}

\begin{remark}\label{r:RenormModel}
If $M \in \mathfrak{R}$ is a renormalisation map as mentioned in Section~\ref{ss:RegStruct}, such that $M \HT \subset \HT$, where $\HT$ is introduced in Definition~\ref{d:TruncSets}, and if $Z^\eps = (\Pi^{\eps}, \Gamma^{\eps}, \Sigma^{\eps})$ is an admissible model, then we can define a renormalised discrete model $\hat Z^\eps$ as in \cite[Sec.~8.3]{Hai14}, which is also admissible.
\end{remark}

The following result is a discrete analogue of Theorem~\ref{t:Integration}.

\begin{theorem}
For a regularity structure $\ST = (\CT, \CG)$ with the minimal homogeneity $\alpha$, let $\beta$, $\gamma$, $\eta$, $\bar \gamma$, $\bar \eta$ and $r$ be as in Theorem~\ref{t:Integration} and let $Z^\eps = (\Pi^\eps, \Gamma^\eps, \Sigma^\eps)$ be a discrete model which realises $K^\eps$ for $\CI$. Then for any discrete modelled distribution $H^\eps$ the following bound holds
\begin{equ}[e:DIntegralBound]
\VERT \CK^\eps_{\gamma} H^\eps \VERT^{(\eps)}_{\bar{\gamma}, \bar{\eta}; T} \lesssim \VERT H^\eps \VERT^{(\eps)}_{\gamma, \eta; T} \Vert \Pi^\eps \Vert_{\gamma; T}^{(\eps)} \Vert \Sigma^\eps \Vert_{\gamma; T}^{(\eps)} \bigl(1 + \Vert \Gamma^\eps \Vert^{(\eps)}_{\bar{\gamma}; T} + \Vert \Sigma^\eps \Vert^{(\eps)}_{\bar{\gamma}; T}\bigr)\;,
\end{equ}
and one has the identity
\begin{equ}[e:DIntegralIdentity]
\CR^\eps_t \bigl(\CK^\eps_{\gamma} H^\eps\bigr)_t(x) = \int_{0}^t \langle \CR^\eps_s H^\eps_s, K^\eps_{t-s}(x - \cdot)\rangle_\eps\, ds\;.
\end{equ}

Moreover, if $\bar{Z}^\eps = (\bar{\Pi}^\eps, \bar{\Gamma}^\eps, \bar{\Sigma}^\eps)$ is another discrete model realising $K^\eps$ for $\CI$, and if $\bar{\CK}^\eps_{\gamma}$ is defined as in \eqref{e:KEpsDef} for this model, then one has the bound
\begin{equ}[e:DIntegrationDistance]
\VERT \CK^\eps_{\gamma} H^\eps; \bar{\CK}^\eps_{\gamma} \bar{H}^\eps \VERT^{(\eps)}_{\bar{\gamma}, \bar{\eta}; T} \lesssim \VERT H^\eps; \bar{H}^\eps \VERT^{(\eps)}_{\gamma, \eta; T} + \VERT Z^\eps; \bar{Z}^\eps \VERT^{(\eps)}_{\bar{\gamma}; T}\;,
\end{equ}
for all discrete modelled distributions $H^\eps$ and $\bar{H}^\eps$, where the norms on $H^\eps$ and $\bar H^\eps$ are defined via the models $Z^\eps$ and $\bar Z^\eps$ respectively, and the proportionality constant depends on $\eps$ only via the same norms of the discrete objects as in \eqref{e:IntegrationDistance}.
\end{theorem}

\begin{proof}
The proof of the bound \eqref{e:DIntegralBound} for the components of $\CK^\eps_{\gamma} H^\eps$ not containing $\mathring{K}^\eps$ is almost identical to that of \eqref{e:Integration}, and we only need to bound the terms $\mathring{\CJ}^\eps H^\eps$ and $\mathring{\CN}^\eps_{\gamma} H^\eps$. The estimates on $\mathring{\CJ}^\eps H^\eps$ were obtained in the proof of Lemma~\ref{l:DGammaIntegralBound}. To bound $\mathring{\CN}^\eps_{\gamma} H^\eps$, for $x, y \in \Lambda_\eps^d$, we write
\begin{equs}
\bigl(\CR^\eps_s H^\eps_s - \Pi^{\eps, s}_{x} \Sigma_x^{\eps, s t} H^\eps_t(x)\bigr)(y) &= \Pi^{\eps, s}_{y} \bigl(H^\eps_s(y) - \Gamma_{y x}^{\eps, s} H^\eps_s(x)\bigr)(y) \\
&\qquad + \Pi^{\eps, s}_{y} \Gamma_{y x}^{\eps, s} \bigl(H^\eps_s(x) - \Sigma_x^{\eps, s t} H^\eps_t(x)\bigr)(y)\;,
\end{equs}
where we made use of Definitions~\ref{d:DReconstruct} and \ref{d:DModel}. Estimating this expression similarly to~\eqref{e:KRingBound}, but using~\eqref{e:DModelledDistributionNorm} this time, we obtain
\begin{equs}[e:NZeroBound]
\Vert \bigl(\mathring{\CN}^{\eps}_{\gamma} H^\eps\bigr)_t(x) \Vert_0 \lesssim \enorm{t}^{\eta - \gamma} \eps^{\gamma + \beta} \lesssim \enorm{t}^{\eta + \beta}\;,
\end{equs}
where we have used $\gamma + \beta > 0$. 

Furthermore, the operator $\Gamma^{\eps, t}_{y x}$ leaves $\1$ invariant, and we have 
\begin{equ}
\Gamma^{\eps, t}_{y x} \bigl(\mathring{\CN}^{\eps}_{\gamma} H^\eps\bigr)_t(x) = \bigl(\mathring{\CN}^{\eps}_{\gamma} H^\eps\bigr)_t(x)\;. 
\end{equ}
Thus, estimating $\bigl(\mathring{\CN}^{\eps}_{\gamma} H^\eps\bigr)_t(y)$ and $\bigl(\mathring{\CN}^{\eps}_{\gamma} H^\eps\bigr)_t(x)$ separately by the intermediate bound in \eqref{e:NZeroBound} and using $|x-y| \geq \eps$, yields the required bound. In the same way we obtain the required estimate on $\Sigma^{\eps, s t}_{x} \bigl(\mathring{\CN}^{\eps}_{\gamma} H^\eps\bigr)_t(x) - \bigl(\mathring{\CN}^{\eps}_{\gamma} H^\eps\bigr)_s(x)$.

The bound \eqref{e:DIntegrationDistance} can be show similarly to \eqref{e:IntegrationDistance}, using the above approach. In order to show that the identity \eqref{e:DIntegralIdentity} holds, we notice that 
\begin{equ}
\bigl(\CK^\eps_{\gamma} H^\eps\bigr)_t(x) \in \poly{\CT} + \CT_{\geq \alpha + \beta}\;, 
\end{equ}
where $\poly{\CT}$ contains only the abstract polynomials and $\alpha + \beta > 0$ by assumption. It hence follows from Definitions \ref{d:DModel} and \ref{d:DReconstruct} that
\begin{equ}
\CR^\eps_t \bigl(\CK^\eps_{\gamma} H^\eps\bigr)_t(x) = \langle \1, \bigl(\CK^\eps_{\gamma} H^\eps\bigr)_t(x) \rangle\;,
\end{equ}
which is equal to the right-hand side of \eqref{e:DIntegralIdentity}.
\end{proof}


\section{Analysis of discrete stochastic PDEs}
\label{s:DPDEs}

We consider the following spatial discretisation of equation \eqref{e:SPDE} on $\R_+ \times \Lambda_\eps^d$:
\begin{equ}[e:DSPDE]
 \partial_t u^{\eps} = A^{\eps} u^{\eps} + F^{\eps}(u^{\eps}, \xi^\eps)\;, \qquad u^{\eps}(0, \cdot) = u^{\eps}_0(\cdot)\;,
\end{equ}
where $u^{\eps}_0 \in \R^{\Lambda_\eps^d}$, $\xi^\eps$ is a spatial discretisation of $\xi$, $F^{\eps}$ is a discrete approximation of $F$, and $A^\eps : \ell^\infty(\Lambda_\eps^d) \to \ell^\infty(\Lambda_\eps^d)$ is a bounded linear operator satisfying the following assumption.

\begin{assumption}\label{a:DOperator}
There exists an operator $A$ given by a Fourier multiplier $a : \R^d \to \R$ satisfying Assumption~\ref{a:Operator} with an even integer parameter $\beta > 0$ and a measure $\mu$ on $\Z^d$
with finite support such that
\begin{equ}[e:DOperatorExtension]
\bigl(A^{\eps} \varphi\bigr) (x) = \eps^{-\beta} \int_{\R^d} \varphi(x - \eps y)\, \mu(dy)\;, \qquad x \in \Lambda_\eps^d\;,
\end{equ}
for every $\varphi \in \CC(\R^d)$, and such that the identity
\begin{equ}[e:DOperatorPoly]
 \int_{\R^d} P(x - y)\, \mu(dy) = (A P)(x)\;, \qquad x \in \R^d\;,
\end{equ}
holds for every polynomial $P$ on $\R^d$ with $\deg P \leq \beta$. 
Furthermore, the Fourier transform of $\mu$ only vanishes on $\Z^d$.
\end{assumption}

\begin{example}\label{ex:Laplacian}
A common example of the operator $A$ is the Laplacian $\Delta$, with its nearest neighbor discrete approximation $\Delta^{\eps}$, defined by \eqref{e:DOperatorExtension} with the measure $\mu$ given by 
\begin{equ}[e:DLaplacian]
\mu ( \varphi) = \sum_{x \in \Z^d : \Vert x \Vert = 1} \bigl(\varphi(x) - \varphi(0)\bigr)\;,
\end{equ}
for every $\varphi \in \ell^\infty(\Z^d)$, and where $\Vert x \Vert$ is the Euclidean norm. In this case, the Fourier multiplier of $\Delta$ is $a(\zeta) = - 4 \pi^2 \Vert \zeta \Vert^2$ and
\begin{equ}
\bigl(\mathscr{F} \mu\bigr)(\zeta) = - 4 \sum_{i=1}^d \sin^2 \bigl(\pi \zeta_i\bigr)\;, \qquad \zeta \in \R^d\;.
\end{equ}
where $\SF$ is the Fourier transform. One can see that Assumption~\ref{a:DOperator} is satisfied with $\beta = 2$.
\end{example}

The following section is devoted to the analysis of discrete operators.


\subsection{Analysis of discrete operators}

We assume that the operator $A^\eps : \ell^\infty(\Lambda_\eps^d) \to \ell^\infty(\Lambda_\eps^d)$ satisfies Assumption~\ref{a:DOperator} and we define the Green's function of $\partial_t - A^\eps$ by
\begin{equ}[e:DGreenDef]
G^\eps_t(x) \eqdef \eps^{-d} \1_{t \geq 0} \bigl(e^{t A^\eps} \delta_{0, \cdot}\bigr)(x)\;, \qquad (t, x) \in \R \times \Lambda_\eps^d\;,
\end{equ}
where $\delta_{\cdot, \cdot}$ is the Kronecker's delta. 

In order to build an extension of $G^\eps$ off the grid, we first choose a function $\varphi \in \CS(\R^d)$ whose values coincide with $\delta_{0, \cdot}$ on $\Z^d$, and such that $\bigl(\SF \varphi\bigr)(\zeta) = 0$ for $|\zeta|_\infty \geq 3/4$, say, where $\SF$ is the Fourier transform. 
To build such a function, write $\tilde \varphi \in \CC^\infty(\R^d)$ for the Dirichlet kernel
$\tilde \varphi(x) = \prod_{i=1}^d \frac{\sin(\pi x_i)}{\pi x_i}$,
whose values coincide with $\delta_{0, x}$ for $x \in \Z^d$, and whose Fourier transform is supported in $\{\zeta : |\zeta|_\infty \leq \frac{1}{2}\}$. Choosing any function $\psi \in \CC^\infty(\R^d)$ 
supported in the ball of radius $1/4$ around the origin and integrating to $1$, it then suffices to set $\SF \varphi = \bigl(\SF \tilde \varphi\bigr) * \psi$.

Furthermore, we define the bounded operator $\tilde{A}^\eps : \CC_b(\R^d) \to \CC_b(\R^d)$ by the right-hand side of \eqref{e:DOperatorExtension}, where $\CC_b(\R^d)$ is the space of bounded continuous functions on $\R^d$ equipped with the supremum norm. Then, denoting as usual by $\varphi^{\eps}$ the rescaled version of $\varphi$, we have for $G^\eps$ the representation
\begin{equ}[e:DGreenExt]
G^\eps_t(x) = \1_{t \geq 0} \bigl(e^{t \tilde{A}^\eps} \varphi^{\eps}\bigr)(x)\;, \qquad (t, x) \in \R \times \Lambda_\eps^d\;.
\end{equ}
By setting $x \in \R^{d}$ in \eqref{e:DGreenExt}, we obtain an extension of $G^\eps$ to $\R^{d+1}$, which we again denote by $G^\eps$.

Unfortunately, the function $G^\eps_t(x)$ is discontinuous at $t = 0$, and our next aim is to modify it in such a way that it becomes differentiable at least for sufficiently large values of $|x|$. Since $\tilde{A}^\eps$ generates a strongly continuous semigroup, for every $m \in \N$ we have the uniform limit
\begin{equ}[e:DGreenTimeDiff]
 \lim_{t \downarrow 0} \partial_t^m G^\eps_t = \bigl(\tilde{A}^\eps \bigr)^m \varphi^{\eps}\;.
\end{equ}
This gives us the terms which we have to subtract from $G^\eps$ to make it continuously differentiable at $t = 0$. For this, we take a function $\varrho : \R \to \R$ such that $\varrho(t) = 1$ for $t \in \bigl[0, \frac{1}{2}\bigr]$, $\varrho(t) = 0$ for $t \in (-\infty, 0) \cup [1, +\infty)$, and $\varrho(t)$ is smooth on $t > 0$. Then, for $r > 0$, we define
\begin{equ}[e:TimeOperator]
T^{\eps, r}(t, x) \eqdef \varrho \bigl(t / \eps^\beta\bigr) \sum_{m \leq r / \beta} \frac{t^m}{m!} \bigl(\tilde A^\eps\bigr)^m \varphi^\eps (x)\;, \qquad (t, x) \in \R^{d+1}\;.
\end{equ}
The role of the function $\varrho$ is to have $T^{\eps, r}$ compactly supported in $t$. Then we have the following result.

\begin{lemma}\label{l:DGreensBound}
In the described context, let Assumption~\ref{a:DOperator} be satisfied. Then for every fixed value $r > 0$ there exists a constant $c > 0$ such that the bound
\begin{equ}[e:GHatBound]
\bigl| D^k \bigl(G^\eps - T^{\eps, r}\bigr)(z) \bigr| \leq C\snorm{z}^{-d - \sabs{k}}\;,
\end{equ}
holds uniformly over $z \in \R^{d+1}$ with $\snorm{z} \geq c \eps$, for all $k \in \N^{d+1}$ with $|k|_{\s} \leq r$, for $D^k$ begin a space-time derivative and for the space-time scaling $\s = (\beta, 1, \ldots, 1)$. 

Moreover, for $|t|_\eps \eqdef |t|^{1/\beta} \vee \eps$, the function $\bar G^\eps_t(x) \eqdef |t|_{\eps}^d G^\eps_{t} \bigl(|t|_{\eps} x\bigr)$ is Schwartz in $x$, i.e. for every $m \in \N$ and $\bar k \in \N^d$ there is a constant $\bar C$ such that the bound
\begin{equ}[e:GHatSchwartz]
\bigl| D_x^{\bar k} \bar G^\eps_t (x) \bigr| \leq \bar C \bigl(1 + |x|\bigr)^{-m}\;,
\end{equ}
holds uniformly over $(t,x) \in \R^{d+1}$.
\end{lemma}

\begin{proof}
The function $G^\eps - T^{\eps, r}$ is of class $\CC^r_\s$ on $\R^{d+1}$. Indeed, spatial regularity follows immediately from the regularity of $\varphi$ and commutation of $\tilde A^\eps$ with the differential operator. Continuous differentiability at $t = 0$ follows from \eqref{e:DGreenTimeDiff}. Furthermore, since $G^\eps$ vanishes on $t \leq 0$, we only need to consider $t > 0$.

Next, we notice that the bound \eqref{e:GHatBound} follows from \eqref{e:GHatSchwartz}. 
Let $\hat r>0$ be such that the measure $\mu$ in Assumption~\ref{a:DOperator} is supported in the ball
of radius $\hat r$. Then, for $k = (k_0, \bar k) \in \N^{d+1}$ with $k_0 \in \N$ and $|k|_{\s} \leq r$ we use \eqref{e:DGreenExt} and the identities \eqref{e:DOperatorPoly}, combined with the Taylor's formula, to get
\begin{equs}[e:DGreenInterBound]
\bigl| D^k G^\eps_t (x) \bigr| = \bigl| \bigl(\tilde A^\eps\bigr)^{k_0} D^{\bar k}_x G^\eps_t(x) \bigr| \lesssim \sup_{y : |y-x| \leq k_0 \hat r \eps} \sup_{l : |l| = \beta k_0} \bigl| D^{\bar k + l}_y G^\eps_t (y) \bigr|\;,
\end{equs}
where $y \in \R^d$, $l \in \N^d$. For $\snorm{t, x} \geq c \eps$, in the case $|t|^{1/\beta} \geq |x|$, we bound the right-hand side of \eqref{e:DGreenInterBound} using \eqref{e:GHatSchwartz} with $m = 0$, what gives an estimate of order $|t|^{-(d + |k|_{\s})/\beta}$. In the case $|t|^{1/\beta} < |x|$, we use \eqref{e:GHatSchwartz} with $m = d + |k|_s$, and we get a bound of order $|x|^{-d - |k|_{\s}}$, if we take $c \geq 2 r \hat r / \beta$. Furthermore, the required bound on $T^{\eps, r}$ follows easily from the properties of the functions $\varphi$ and $\varrho$. Hence, we only need to prove the bound \eqref{e:GHatSchwartz}.

Denoting by $\mathscr{F}$ the Fourier transform, we get from \eqref{e:DGreenExt} and Assumption~\ref{a:DOperator}:
\begin{equs}[e:DGreenBound]
\bigl( \mathscr{F} \bar G^{\eps}_t\bigr)(\zeta) = \bigl( \mathscr{F} \varphi\bigr) \bigl(\eps |t|_{\eps}^{-1} \zeta\bigr)\, e^{t |t|_{\eps}^{-1} a(\zeta) f( \eps |t|_{\eps}^{-1} \zeta)}\;,
\end{equs}
where we have used the scaling property $\lambda^\beta a(\zeta) = a(\lambda \zeta)$, and where $f \eqdef (\mathscr{F} \mu) / a$.

We start with considering the case $t \geq \eps^\beta$. It follows from the last part of
Assumption~\ref{a:DOperator} that there exists $\bar c>0$ such that $f(\zeta) \ge \bar c$ for
$|\zeta|_\infty \le 3/4$. Since
$\eps |t|_{\eps}^{-1} \leq 1$, we conclude that
\begin{equ}
\bigl|D^{\bar k}_\zeta e^{a(\zeta) f( \eps |t|_{\eps}^{-1} \zeta)} \bigr| \lesssim |\zeta|^{\beta |\bar k|} e^{a(\zeta) \bar c} \lesssim \bigl(1 + |\zeta| \bigr)^{-m}\;,
\end{equ}
for $|\zeta|_\infty < 3 / \bigl(4\eps |t|_{\eps}^{-1}\bigr)$, for every $m \geq 0$ and for a proportionality constant dependent on $m$ and $\bar k$. Here, we have used $a(\zeta) < 0$ and polynomial growth of $|a(\zeta)|$. Since $\bigl( \mathscr{F} \varphi\bigr) \bigl(\eps |t|_{\eps}^{-1} \zeta\bigr)$ vanishes for $|\zeta|_\infty \geq 3 / \bigl(4\eps |t|_{\eps}^{-1}\bigr)$, we conclude that 
\begin{equ}
\bigl|D^{\bar k}_\zeta \bigl( \mathscr{F} \bar G^{\eps}_t\bigr)(\zeta)\bigr| \lesssim \bigl(1 + |\zeta| \bigr)^{-m}\;,
\end{equ}
uniformly in $t$ and $\eps$ (provided that $t \geq \eps^\beta$), 
and for every $m \in \N$ and $\bar k \in \N^d$.

In the case $t < \eps^\beta$, we can bound the exponent in \eqref{e:DGreenBound} by $1$, and the polynomial decay comes from the factor $\bigl( \mathscr{F} \varphi\bigr) \bigl(\zeta\bigr)$, because $\varphi \in \CS(\R^d)$. Since the Fourier transform is continuous on Schwartz space, this implies that $\bar G^{\eps}_t$ is a Schwartz function, with bounds uniform in $\eps $ and $t$, which is exactly the claim.
\end{proof}

The following result is an analogue of Lemma~\ref{l:GreenDecomposition} for $G^\eps$.

\begin{lemma} \label{l:DGreenDecomposition}
Let Assumption~\ref{a:DOperator} be satisfied. Then, the function $G^\eps$ defined in \eqref{e:DGreenExt} can be written as $G^\eps = K^\eps + R^\eps$ in such a way that the identity
\begin{equ}[e:sumGreen]
\bigl(G^\eps \star_\eps u\bigr)(z) = \bigl(K^\eps \star_\eps u\bigr)(z) + \bigl(R^\eps \star_\eps u\bigr)(z)\;,
\end{equ}
holds for all $z \in (-\infty, 1] \times \Lambda_\eps^{d}$ and all functions $u$ on $\R_+ \times \Lambda_\eps^d$, periodic in the spatial variable with some fixed period. Furthermore, $K^\eps$ is regularising of order $\beta$ in the sense of Definition~\ref{d:DKernel}, for arbitrary (but fixed) $r$ and with the scaling $\s = (\beta, 1, \ldots, 1)$. The function $R^{\eps}$ is compactly supported, non-anticipative and the norm $\Vert R^\eps \Vert_{\CC^r}$ is bounded uniformly in $\eps$.
\end{lemma}

\begin{proof}
Let $M : \R^{d+1} \to \R_+$ be a smooth norm for the scaling $\s$ (see for example \cite[Rem.~2.13]{Hai14}). Furthermore, let $\bar \rho : \R_+ \to [0,1]$ be a smooth ``cutoff function'' such that $\bar \rho(s) = 0$ if $s \notin [1/2, 2]$, and such that $\sum_{n \in \Z} \bar \rho(2^n s) = 1$ for all $s > 0$ (see the construction of the partition of unity in \cite{BCD11}). For integers $n \in [0, N)$ we set the functions 
\begin{equ}
\bar \rho_n(z) \eqdef \bar \rho(2^{n} M(z))\;, \qquad \bar \rho_{< 0} \eqdef \sum_{n < 0} \bar \rho_n\;, \qquad \bar \rho_{\ge N} \eqdef \sum_{n \ge N} \bar \rho_n\;, 
\end{equ}
as well as
\begin{equs}
\bar{K}^{(\eps, n)}(z) = \bar \varrho_n(z) &\bigl(G^\eps - T^{\eps, r}\bigr)(z)\;, \qquad \bar{R}^\eps(z) = \bar \varrho_{< 0}(z) \bigl(G^\eps - T^{\eps, r}\bigr)(z)\;,\\
\tilde{K}^{\eps}(z) &= \bar \varrho_{\ge N}(z) \bigl(G^\eps - T^{\eps, r}\bigr)(z) +  T^{\eps, r}(z)\;.\label{e:TildeKDef}
\end{equs}
Then it follows immediately from the properties of $\bar \varrho$ that 
\begin{equ}
G^\eps = \sum_{n = 0}^{N-1} \bar K^{(\eps, n)} + \tilde{K}^{\eps} + \bar R^\eps\;.
\end{equ}
Since $\bar \varrho_{< 0}$ is supported away from the origin, we use \eqref{e:GHatBound} and Assumption~\ref{a:DOperator} to conclude that $\Vert \bar{R}^\eps \Vert_{\CC^r}$ is bounded uniformly in $\eps$. 
(Actually, its value and derivatives even decay faster than any power.)

Furthermore, the function $\bar{K}^{(\eps, n)}$ is supported in the ball of radius $c 2^{-n}$, for $c$ as 
in Lemma~\ref{l:DGreensBound}, provided that the norm $M$ was chosen such that $M(z) \geq 2 c \snorm{z}$. 
By the same reason, the first term in \eqref{e:TildeKDef} is supported in the ball of radius $c \eps$. 
Moreover, the support property of the measure $\mu$ and the properties of the functions $\varrho$ 
and $\varphi^\eps$ in \eqref{e:TimeOperator} yield that the restriction of $T^{\eps, r}$ to 
the grid $\Lambda_\eps^d$ in space is supported in the ball of radius $c \eps$, as soon as $c \geq 2 r \hat r / \beta$, where $\hat r$ is the support radius of the measure $\mu$ from Assumption~\ref{a:DOperator}.

As a consequence of \eqref{e:DOperatorExtension}, \eqref{e:DGreenExt} and \eqref{e:TimeOperator}, 
we get for $0 \leq n < N$ the exact scaling properties
\begin{equs}
\bar{K}^{(\eps, n)}(z) = 2^{n d} \bar{K}^{(\eps 2^n, 0)}\bigl(2^{\s n} z\bigr)\;,\qquad \tilde{K}^{\eps}(z) = \eps^{-d} \tilde{K}^{1}\bigl(\eps^{-\s n} z\bigr)\;,
\end{equs}
and \eqref{e:KernelBound} and \eqref{e:KZeroBounds} follow immediately from \eqref{e:GHatBound} and \eqref{e:TimeOperator}.

It remains to modify these functions in such a way that they ``kill'' polynomials in the sense of \eqref{e:DPolyKill}. To this end, we take a smooth function $P^{(N)}$ on $\R^{d+1}$, whose support coincides with the support of $\tilde{K}^\eps$, which satisfies $|P^{(N)}(z)| \lesssim \eps^{-d}$, for every $z \in \R^{d+1}$, and such that one has
\begin{equs}[e:PKillerDef]
\int_{\R \times \Lambda_\eps^d} \bigl(\tilde{K}^\eps - P^{(N)} \bigr)(z)\, dz = 0\;.
\end{equs}
Then we define $\mathring{K}^{\eps}$ to be the restriction of $\tilde{K}^{\eps} - P^{(N)}$ to the grid $\Lambda_\eps^d$ in space. Clearly, the function $\mathring{K}^{\eps}$ has the same scaling and support properties as $\tilde{K}^{\eps}$, and it follows from \eqref{e:PKillerDef} that it satisfies \eqref{e:DPolyKill} with $k=0$.

Moreover, we can recursively build a sequence of smooth functions $P^{(n)}$, for integers $n \in [0, N)$, such that $P^{(n)}$ in supported in the ball of radius $c 2^{-n}$, the function $P^{(n)}$ satisfies the bounds in \eqref{e:KernelBound}, and for every $k \in \N^{d+1}$ with $\sabs{k} \leq r$ one has
\begin{equ}[e:PKillerTwoDef]
\int_{\R \times \Lambda_\eps^d} z^{k} \left(\bar K^{(\eps, n)} - P^{(n)} + P^{(n + 1)} \right)(z)\, dz = 0\;. 
\end{equ}
Then, for such values of $n$, we define
\begin{equs}
K^{(\eps, n)} = \bar{K}^{(\eps, n)} - P^{(n)} + P^{(n+1)}\;, \qquad R^\eps \eqdef \bar R^\eps + P^{(0)}\;.
\end{equs}
It follows from the properties of the functions $P^{(n)}$ that $K^{(\eps, n)}$ has all the required
properties. The function $R^\eps$ also has the required properties, and
the decompositions \eqref{e:DEpsExpansion} and \eqref{e:sumGreen} hold by construction.
Finally, using \eqref{e:GHatSchwartz}, we can make the function $R^\eps$ compactly supported in the same way as in \cite[Lem.~7.7]{Hai14}.
\end{proof}

\begin{remark}\label{r:KernelDiff}
 One can see from the proof of Lemma~\ref{l:DGreenDecomposition} that the function $\mathring{K}^\eps$ is $(r / \s_0)$-times continuously differentiable in the time variable for $t \neq 0$ and has a discontinuity at $t = 0$.
\end{remark}

By analogy with \eqref{e:RDef}, we use the function $R^\eps$ from Lemma~\ref{l:DGreenDecomposition} to define for periodic $\zeta_t \in \R^{\Lambda_\eps^d}$, $t \in \R$, the abstract polynomial
\begin{equ}[e:DRDef]
\bigl(R^\eps_{\gamma} \zeta\bigr)_t(x) \eqdef \sum_{\sabs{k} < \gamma} \frac{X^k}{k!} \int_{\R} \langle \zeta_s, D^k R^\eps_{t-s}(x - \cdot) \rangle_\eps\, ds\;,
\end{equ}
where as before $k \in \N^{d+1}$ and the mixed derivative $D^k$ is in space-time.


\subsection{Properties of the discrete equations}

In this section we show that a discrete analogue of Theorem~\ref{t:FixedMap} holds for the solution map of the equation \eqref{e:DSPDE} with an operator $A^\eps$ satisfying Assumption~\ref{a:DOperator}.

Similarly to \cite[Lem.~7.5]{Hai14}, but using the properties of $G^\eps$ proved in the previous section, we can show that for every periodic $u^\eps_0 \in \R^{\Lambda^d_\eps}$, we have a discrete analogue of Lemma~\ref{l:InitialData} for the map $(t,x) \mapsto S^\eps_t u^\eps_0(x)$, where $S^\eps$ is the semigroup generated by $A^\eps$.

For the regularity structure $\ST$ from Section~\ref{ss:RegStruct}, we take a truncated regularity structure $\hat{\ST}=(\HT, \CG)$ and make the following assumption on the nonlinearity $F^\eps$.

\begin{assumption}\label{a:DNonlin}
For some $0 < \bar\gamma \leq \gamma$, $\eta \in \R$, every $\eps > 0$ and every discrete model $Z^\eps$ on $\hat{\ST}$, there exist discrete modeled distributions $F_0^\eps(Z^\eps)$ and $I_0^\eps(Z^\eps)$, with exactly the same properties as of $F_0$ and $I_0$ in Assumption~\ref{a:Nonlin} on the grid. Furthermore, we define $\hat{F}^\eps$ as in \eqref{e:NonlinAs}, but via $F^\eps$ and $F_0^\eps$, and we define $\hat{F}^{\eps}(H)$ for $H : \R_+ \times \Lambda^d_\eps \to \CT_{< \gamma}$ as in \eqref{e:NonlinearTerm}. Finally, we assume that the discrete analogue of the Lipschitz condition \eqref{e:Lipschitz} holds for $\hat{F}^{\eps}$, with the constant $C$ independent of $\eps$.
\end{assumption}

Similarly to \eqref{e:AbstractEquation}, but using the discrete operators \eqref{e:DReconstructDef}, \eqref{e:DRDef} and \eqref{e:KEpsDef}, we reformulate the equation \eqref{e:DSPDE} as
\begin{equ}[e:DAbstractEquation]
U^\eps = \CP^\eps \hat{F}^\eps(U^\eps) + S^\eps u^\eps_0 + I_0^\eps\;,
\end{equ}
where $\CP^\eps \eqdef \CK^\eps_{\bar{\gamma}} + R^\eps_\gamma \CR^\eps$ and $U^\eps$ is a discrete modeled distribution.

\begin{remark}\label{r:DClassicalSolution}
If $Z^\eps$ is a canonical discrete model, then it follows from \eqref{e:DIntegralIdentity}, \eqref{e:DRDef}, \eqref{e:DReconstructDef}, Definition~\ref{d:DModel} and Assumption~\ref{a:DNonlin} that
\begin{equ}[e:DSPDESolution]
u^\eps_t(x) = \bigl(\CR^\eps_t U^\eps_t\bigr)(x)\;, \qquad (t, x) \in \R_+ \times \Lambda_\eps^d\;.
\end{equ}
is a solution of the equation \eqref{e:DSPDE}.
\end{remark}

The following result can be proven in the same way as Theorem~\ref{t:FixedMap}.

\begin{theorem}\label{t:DSolutions}
Let $Z^\eps$ be a sequence of models and let $u^\eps_0$ be a sequence of periodic functions on $\Lambda_\eps^d$. Let furthermore the assumptions of Theorem~\ref{t:FixedMap} and Assumption~\ref{a:DNonlin} be satisfied. Then there exists $T_{\star} \in (0, +\infty]$ such that for every $T < T_{\star}$ the sequence of solution maps $\mathcal{S}^\eps_T : (u^\eps_0, Z^\eps) \mapsto U^\eps$ of the equation \eqref{e:DAbstractEquation} is jointly Lipschitz continuous (uniformly in $\eps$!) in the sense of Theorem~\ref{t:FixedMap}, but for the discrete objects.
\end{theorem}

\begin{remark}
Since we require uniformity in $\eps$ in Theorem~\ref{t:DSolutions}, the solution of equation \eqref{e:DAbstractEquation} is considered only up to some time point $T_{\star}$.
\end{remark}


\section{Inhomogeneous Gaussian models}
\label{s:GaussModels}

In this section we analyse discrete and continuous models which are built from Gaussian noises. In the discrete case, we will work as usual on the grid $\Lambda_\eps^d$, with $\eps = 2^{-N}$ and $N \in \N$, and with the time-space scaling $\s = (\s_0, 1, \ldots, 1)$.

We assume that we are given a probability space $(\Omega, \CF, \P)$, together with a white noise $\xi$ over the Hilbert space $H \eqdef L^2(D)$ (see \cite{Nua06}), where $D \eqdef \R \times \Torus^d$ and $\Torus \eqdef \R / \Z$ is the unit circle. In the sequel, we will always identify $\xi$ with its periodic extension to $\R^{d+1}$.

In order to build a spatial discretisation of $\xi$, we take a compactly supported function $\varrho : \R^d \to \R$, such that for every $y \in \Z^d$ one has 
\begin{equ}
\int_{\R^d} \varrho(x) \varrho(x - y)\, dx = \delta_{0, y}\;, 
\end{equ}
where $\delta_{\cdot, \cdot}$ is the Kronecker's function. Then, for $x \in \Lambda_\eps^d$, we define the scaled function $\varrho^\eps_x(y) \eqdef \eps^{-d} \varrho((y-x)/\eps)$ and
\begin{equ}[e:DNoise]
 \xi^\eps(t,x) \eqdef \xi(t, \varrho^\eps_x)\;, \qquad (t,x) \in \R \times \Lambda_\eps^d\;.
\end{equ}
One can see that $\xi^\eps$ is a white noise on the Hilbert space $H_\eps \eqdef L^2(\R) \otimes \ell^2(\Torus_\eps^d)$, where $\Torus_\eps \eqdef (\eps \Z) / \Z$ and $\ell^2(\Torus_\eps^d)$ is equipped with the inner product $\langle \cdot, \cdot \rangle_\eps$.

In the setting of Section~\ref{sec:trunc}, we assume that
$Z^\eps = (\Pi^\eps, \Gamma^\eps, \Sigma^\eps)$ is a discrete model on $\HT$ such that, for each $\tau \in \hat{\CF}$ and each test function $\phi$, the maps $\langle\Pi^{\eps, t}_x \tau, \varphi\rangle_\eps$, $\Gamma^{\eps, t}_{xy} \tau$ and $\Sigma^{\eps, s t}_{x} \tau$ belong to the inhomogeneous 
Wiener chaos of order $\VERT \tau \VERT$ (the number of occurrences of $\Xi$ in $\tau$) with respect to $\xi^\eps$. Moreover, we assume that the distributions of the functions $(t,x) \mapsto \langle \Pi^{\eps, t}_x \tau, \varphi_x\rangle_\eps$, $(t,x) \mapsto \Gamma^{\eps, t}_{x, x + h_1} \tau$ and $(t,x) \mapsto \Sigma^{\eps, t, t+h_2}_{x} \tau$ are stationary, for all $h_1 \in \Lambda_\eps^d$ and $h_2 \in \R$. In what follows, we will call the discrete models with these properties {\it stationary Gaussian discrete models}.

The following result provides a criterion for such a model to be bounded uniformly in $\eps$. In its statement we use the following set:
\begin{equ}[e:TMinus]
\HF^{-} \eqdef \left(\{\tau \in \HF: |\tau| < 0\} \cup \gen{\CF}\right) \setminus \poly{\CF}\;.
\end{equ}

\begin{theorem}\label{t:ModelsConvergence}
In the described context, let $\hat{\ST} = (\HT, \CG)$ be a truncated regularity structure and let $Z^\eps = (\Pi^\eps, \Gamma^\eps, \Sigma^\eps)$ be an admissible stationary Gaussian discrete model on it. Let furthermore the bounds
\begin{equ}[e:GaussModelBoundSigmaGamma]
\EE \Big[\Vert \Gamma^\eps \Vert_{\gamma; T}^{(\eps)}\Big]^p \lesssim 1\;, \qquad \EE\Big[\Vert \Sigma^\eps \Vert_{\gamma; T}^{(\eps)}\Big]^p \lesssim 1\;.
\end{equ}
hold uniformly in $\eps$ (see Remark~\ref{r:Uniformity}) on the respective generating regularity structure $\gen{\ST} = (\gen{\CT}, \CG)$, for every $p \geq 1$, for every $\gamma > 0$ and for some $T \geq c$, where $c > 0$ is from Definition~\ref{d:DKernel} and where the proportionality constants can depend on $p$. Let finally $\Pi^\eps$ be such that for some $\delta > 0$ and for each $\tau \in \HF^{-}$ the bounds
\begin{equs}[e:GaussModelBound]
\EE\Big[  |\langle \Pi^{\eps, t}_x \tau, \varphi_x^\lambda\rangle_\eps|^2\Big] &\lesssim \lambda^{2 |\tau| + \kappa}\;, \\
\EE\Big[ |\langle \bigl(\Pi^{\eps, t}_x - \Pi^{\eps, s}_x\bigr) \tau, \varphi_x^\lambda\rangle_\eps|^2 \Big] &\lesssim \lambda^{2 (|\tau| - \delta) + \kappa} |t-s|^{2 \delta /\s_0}\;,
\end{equs}
hold uniformly in $\eps$, all $\lambda \in [\eps,1]$, all $x \in \Lambda_\eps^d$, all $s \neq t \in [-T, T]$ and all $\varphi \in \CB^r_0(\R^d)$ with $r > - \lfloor\min \hat{\CA}\rfloor$. Then, for every $\gamma > 0$, $p \geq 1$ and $\bar \delta \in [0, \delta)$, one has the following bound on $\hat{\ST}$ uniformly in $\eps$:
\begin{equ}[e:GaussianModelBound]
\EE\left[\VERT Z^\eps \VERT_{\bar \delta, \gamma; T}^{(\eps)}\right]^p \lesssim 1\;.
\end{equ}

Finally, let $\bar Z^\eps = (\bar \Pi^\eps, \bar \Gamma^\eps, \bar\Sigma^\eps)$ be another admissible stationary Gaussian discrete model on $\HT$, such that for some $\theta > 0$ and some $\bar{\eps} > 0$ the maps $\Gamma^\eps - \bar{\Gamma}^\eps$, $\Sigma^\eps - \bar{\Sigma}^\eps$ and $\Pi^{\eps} - \bar{\Pi}^\eps$ satisfy the bounds \eqref{e:GaussModelBoundSigmaGamma} and \eqref{e:GaussModelBound} respectively with proportionality constants of order $\bar{\eps}^{2\theta}$. Then, for every $\gamma > 0$, $p \geq 1$ and $\bar \delta \in [0, \delta)$, the models $Z^\eps$ and $\bar Z^\eps$ satisfy on $\hat{\ST}$ the bound
\begin{equ}[e:GaussianModelsBound]
\EE\left[\VERT Z^\eps; \bar{Z}^\eps \VERT_{\bar \delta, \gamma; T}^{(\eps)}\right]^p \lesssim \bar{\eps}^{\theta p}\;,
\end{equ}
uniformly in $\eps \in (0,1]$.
\end{theorem}

\begin{proof}
Since by assumption $\langle\Pi^{\eps, t}_x \tau, \varphi\rangle_\eps$ belongs to a fixed inhomogeneous Wiener chaos, the equivalence of moments \cite{Nel73} and the bounds \eqref{e:GaussModelBound} yield the respective bounds on the $p$-th moments, for any $p \geq 1$. In particular, the Kolmogorov continuity criterion implies that for such $p$ the bounds 
\begin{equs}[e:GaussModelBoundKolmogorov]
\EE\bigg[ \sup_{t \in [-T, T]} |\langle \Pi^{\eps, t}_x \tau, \varphi_x^\lambda\rangle_\eps|\bigg]^p &\lesssim \lambda^{p |\tau| + \bar \kappa}\;, \\
\EE\bigg[ \sup_{s \neq t \in [-T, T]} \frac{|\langle \bigl(\Pi^{\eps, t}_x - \Pi^{\eps, s}_x\bigr) \tau, \varphi_x^\lambda \rangle_\eps|}{|t-s|^{\bar \delta /\s_0}} \bigg]^p &\lesssim \lambda^{p (|\tau| - \delta) + \bar \kappa}\;,
\end{equs}
hold uniformly over $x$, $\varphi$ and $\lambda$ as in \eqref{e:GaussModelBound} and for some $\bar \kappa > 0$ depending on $p$. Going now by induction from the elements of $\gen{\CT}$ to the elements of $\HT$, using Lemmas~\ref{l:DPiIntegralBound} and~\ref{l:DGammaIntegralBound} and the discrete multiresolution analysis defined in Section~\ref{ss:DMultiresolutionAnalysis}, we can obtain~\eqref{e:GaussianModelBound} in the same way as in the proof of \cite[Thm.~10.7]{Hai14}. The bound \eqref{e:GaussianModelsBound} can be proved similarly.
\end{proof}

The conditions \eqref{e:GaussModelBound} can be checked quite easily if the maps $\Pi^\eps \tau$ have certain Wiener chaos expansions. More precisely, we assume that there exist kernels $\CW^{(\eps; k)}\tau$ such that $\bigl(\CW^{(\eps; k)} \tau\bigr)(z) \in H^{\otimes k}_\eps$, for $z \in \R \times \Lambda_\eps^d$, and
\begin{equ}[e:PiWiener]
\langle\Pi^{\eps, t}_0 \tau, \varphi\rangle_\eps = \sum_{k \leq \VERT \tau \VERT} I^\eps_k\Big(\int_{\Lambda_\eps^{d}} \varphi(y)\, \bigl(\CW^{(\eps; k)}\tau\bigr)(t,y)\, dy\Big)\;,
\end{equ}
where $I^\eps_k$ is the $k$-th order Wiener integral with respect to $\xi^\eps$ and the space $H_\eps$ is introduced above. Then we define the function
\begin{equ}[e:CovDef]
\bigl(\CK^{(\eps; k)}\tau\bigr)(z_1, z_2) \eqdef \langle \bigl(\CW^{(\eps; k)}\tau\bigr)(z_1), \bigl(\CW^{(\eps; k)}\tau\bigr)(z_2) \rangle_{H_\eps^{\otimes k}}\;,
\end{equ}
for $z_1 \neq z_2 \in \R \times \Lambda_\eps^d$, assuming that the expression on the right-hand side is well-defined. 

In the same way, we assume that the maps $\bar{\Pi}^\eps \tau$ are given by \eqref{e:PiWiener} via the respective kernels $\bar{\CW}^{(\eps; k)}\tau$. Moreover, we define the functions $\delta \CK^{(\eps; k)}\tau$ as in \eqref{e:CovDef}, but via the kernels $\bar \CW^{(\eps; k)}\tau - \CW^{(\eps; k)}\tau$, and we assume that the functions $\CK^{(\eps; k)}\tau$ and $\delta \CK^{(\eps; k)}\tau$ depend on the time variables $t_1$ and $t_2$ only via $t_1 - t_2$, i.e.
\begin{equ}[e:KernelTime]
\bigl(\CK^{(\eps; k)}\tau\bigr)_{t_1 - t_2}(x_1, x_2) \eqdef \bigl(\CK^{(\eps; k)}\tau\bigr)(z_1, z_2)\;,
\end{equ}
where $z_i = (t_i, x_i)$, and similarly for $\delta \CK^{(\eps; k)}\tau$.

The following result shows that the bounds \eqref{e:GaussModelBound} follow from 
corresponding bounds on these functions.

\begin{proposition}
\label{p:CovarianceConvergence}
In the described context, we assume that for some $\tau \in \HF^{-}$ there are values $\alpha > |\tau| \vee (-d/2)$ and $\delta \in (0, \alpha + d / 2)$ such that the bounds
\begin{equs}[e:CovarianceBounds]
|\big(\CK^{(\eps; k)}\tau\big)_0(x_1, x_2)| &\lesssim \sum_{\zeta \geq 0} \big( \senorm{0, x_1} + \senorm{0, x_2}\big)^\zeta \senorm{0, x_1 - x_2}^{2 \alpha - \zeta}\;,\\
\frac{|\delta^{0, t} \big(\CK^{(\eps; k)}\tau\big)(x_1, x_2)|}{|t|^{2 \delta/ \s_0}} &\lesssim \sum_{\zeta \geq 0} \big( \senorm{t, x_1} + \senorm{t, x_2}\big)^\zeta \senorm{0, x_1- x_2}^{2 \alpha - 2 \delta - \zeta}\;,
\end{equs}
hold uniformly in $\eps$ for $t \in \R$, $x_1, x_2 \in \Lambda_\eps^d$ and $k \leq \VERT \tau \VERT$, where the operator $\delta^{0, t}$ is defined in \eqref{e:deltaTime}, where $\senorm{z} \eqdef \snorm{z} \vee \eps$, and where the sums run over finitely many values of $\zeta \in [0, 2 \alpha - 2 \delta + d)$. Then the bounds \eqref{e:GaussModelBound} hold for $\tau$ with a sufficiently small value of $\kappa > 0$.

Let furthermore \eqref{e:CovarianceBounds} hold for the function $\delta \CK^{(\eps; k)}\tau$ with the proportionality constant of order $\bar{\eps}^{2\theta}$, for some $\theta > 0$. Then the required bounds on $\bigl(\Pi^\eps - \bar \Pi^\eps\bigr)\tau$ in Theorem~\ref{t:ModelsConvergence} hold.
\end{proposition}

\begin{proof}
We note that due to our assumptions on stationarity of the models, it is sufficient to check the conditions \eqref{e:GaussModelBound} only for $\langle \Pi^{\eps, t}_0 \tau, \varphi_0^\lambda\rangle_\eps$ and $\langle \bigl(\Pi^{\eps, t}_0 - \Pi^{\eps, 0}_0\bigr) \tau, \varphi_0^\lambda \rangle_\eps$, and respectively for the map $\bar \Pi^\eps$.

We start with the proof of the first statement of this proposition. We denote by $\Pi^{(\eps, k), t}_0 \tau$ the component of $\Pi^{\eps, t}_0 \tau$ belonging to the $k$-th homogeneous Wiener chaos. Furthermore, we will use the following property of the Wiener integral \cite{Nua06}:
\begin{equ}[e:WienerBound]
 \EE\bigl[I^\eps_k(f)^2\bigr] \leq \Vert f \Vert_{H_\eps^{\otimes k}}\;, \qquad f \in H_\eps^{\otimes k}\;.
\end{equ}
Thus, from this property, \eqref{e:KernelTime} and the first bound in \eqref{e:CovarianceBounds}, we get
\begin{equs}
 \EE| \langle \Pi^{(\eps, k), t}_0 \tau&, \varphi^\lambda_0 \rangle_\eps |^{2} \lesssim \int_{\Lambda_\eps^d} \int_{\Lambda_\eps^d} |\varphi^\lambda_0 (x_1)|\, |\varphi^\lambda_0(x_2)|\, |\big(\CK^{(\eps, k)}\tau\big)_0(x_1, x_2)|\, d x_1 d x_2\\
&\lesssim \lambda^{-2 d} \sum_{\zeta \geq 0} \int_{\substack{|x_1| \leq \lambda \\ |x_2| \leq \lambda}} \bigl( \senorm{0, x_1} + \senorm{0, x_2}\bigr)^\zeta \senorm{0, x_1- x_2}^{2 \alpha - \zeta}\, d x_1 d x_2\\
&\lesssim \lambda^{-2 d} \sum_{\zeta \geq 0} \lambda^{d + \zeta} \int_{|x| \leq 2\lambda} \senorm{0, x}^{2 \alpha - \zeta}\, d x \lesssim \lambda^{2 \alpha}\;,\label{e:PiEstimate}
\end{equs}
for $\lambda \geq \eps$. Here, to have the proportionality constant independent of $\eps$, we need $2 \alpha - \zeta > - d$. Combining the bounds \eqref{e:PiEstimate} for each $k$ with stationarity of $\Pi^{\eps} \tau$, we obtain the first estimate in \eqref{e:GaussModelBound}, with a sufficiently small $\kappa > 0$.

Now, we will investigate the time regularity of the map $\Pi^\eps$. For $|t| \geq \lambda^{\s_0}$ we can use \eqref{e:PiEstimate} and brutally bound
\begin{equs}
\EE| \langle \delta^{0, t} \Pi^{(\eps, k)}_0 \tau, \varphi^\lambda_0 \rangle_\eps |^2 &\lesssim \EE| \langle \Pi^{(\eps, k), t}_0 \tau, \varphi^\lambda_0 \rangle_\eps |^2 + \EE| \langle \Pi^{(\eps, k), 0}_0 \tau, \varphi^\lambda_0 \rangle_\eps |^2\\
&\lesssim \lambda^{2 \alpha} \lesssim |t|^{2 \delta / s_0} \lambda^{2 \alpha - 2 \delta}\;,\label{e:PiTimeSimpleEstimate}
\end{equs}
for any $\delta \geq 0$, which is the required estimate. In the case $|t| < \lambda^{\s_0}$, the bound \eqref{e:WienerBound} and second bound in~\eqref{e:CovarianceBounds} yield
\begin{equs}
\EE&| \langle \delta^{0, t} \Pi^{(\eps, k)}_0 \tau, \varphi^\lambda_0\rangle_\eps |^2 \lesssim \int_{\Lambda_\eps^d} \int_{\Lambda_\eps^d} |\varphi^\lambda_0 (x_1)|\, |\varphi^\lambda_0 (x_2)|\, |\delta^{0, t} \big(\CK^{(\eps, k)}\tau\big)(x_1, x_2)|\, d x_1 d x_2\\
&\qquad \qquad + \int_{\Lambda_\eps^d} \int_{\Lambda_\eps^d} |\varphi^\lambda_0 (x_1)|\, |\varphi^\lambda_0 (x_2)|\, |\delta^{-t, 0} \big(\CK^{(\eps, k)}\tau\big)(x_1, x_2)|\, d x_1 d x_2\\
&\lesssim |t|^{2 \delta / \s_0} \lambda^{-2 d} \sum_{\zeta \geq 0} \int_{\substack{|x_1| \leq \lambda \\ |x_2| \leq \lambda}} \bigl(\senorm{t, x_1} + \senorm{t, x_2}\bigr)^\zeta \senorm{0, x_1- x_2}^{2 \alpha - 2 \delta - \zeta} d x_1 d x_2 \\
&\lesssim |t|^{2 \delta/ \s_0} \lambda^{2 \alpha - 2 \delta}\;,\label{e:PiTimeNotSimpleEstimate}
\end{equs}
where the integral is bounded as before for $2 \alpha - 2 \delta - \zeta > - d$. Combining the bounds \eqref{e:PiTimeSimpleEstimate} and \eqref{e:PiTimeNotSimpleEstimate} for each value of $k$ with stationarity of $\Pi^{\eps} \tau$, we obtain the second estimate in \eqref{e:GaussModelBound}. The required bounds on $\bigl(\Pi^\eps - \bar \Pi^\eps\bigr)\tau$ can be proved in a similar way.
\end{proof}

\begin{remark}\label{r:ContinousModelLift}
Assume that we are given an admissible continuous model $Z = (\Pi, \Gamma, \Sigma)$ on $\hat \ST$ such that the map $\Pi$ is given on $\hat{\CF}^-$ by the expansions \eqref{e:PiWiener} in which we replace all the discrete objects by their continuous counterparts. Then one can prove in the same way analogues to Theorem~\ref{t:ModelsConvergence} and Proposition~\ref{p:CovarianceConvergence} in the continuous case, i.e. when we use $\eps = 0$ and use continuous objects in place of the discrete ones.
\end{remark}

\subsection{Continuous inhomogeneous models}
\label{ss:ContinuousGauss}

In this section we will show how in some cases we can build a continuous inhomogeneous model from an admissible model in the sense of \cite[Def.~8.29]{Hai14}. 

For a white noise $\xi$ on a Hilbert space $H$ as in the beginning of the previous section, we assume that we are given an admissible model $\tilde{Z}=(\tilde{\Pi}, \tilde{\Gamma})$ in the sense of \cite[Def.~8.29]{Hai14} on the truncated regularity structure $\hat \ST$ such that for every $\tau \in \hat \CF$, every test function $\varphi$ on $\R^{d+1}$ and every pair of points $z, \bar z \in \R^{d+1}$, the maps $\langle\tilde{\Pi}_z \tau, \varphi\rangle$ and $\tilde{\Gamma}_{z \bar z} \tau$ belong to the inhomogeneous Wiener chaos of order $\VERT \tau \VERT$ (the quantity $\VERT \tau \VERT$ is defined in the beginning of Section~\ref{s:GaussModels}) with respect to $\xi$. Furthermore, we assume that for every $\tau \in \hat{\CF}$ there exist kernels $\CW^{(k)}\tau$ such that for every test function $\phi$ on $\R^{d+1}$ one has $\int_{\R^{d+1}} \phi(z) \bigl(\CW^{(k)} \tau\bigr)(z)\, dz \in H^{\otimes k}$, postulating that the integral is well-defined, and $\tilde{\Pi}_z \tau$ can be written as
\begin{equ}[e:GeneralPiWiener]
\langle \tilde{\Pi}_{z} \tau, \varphi_z\rangle = \sum_{k \leq \VERT \tau \VERT} I_k\Big(S_z^{\otimes k} \int_{\R^{d+1}} \varphi(\bar z)\, \bigl(\CW^{(k)}\tau\bigr)(\bar z)\, d\bar z\Big)\;,
\end{equ}
where $I_k$ is the $k$-th Wiener integral with respect to $\xi$, $\varphi_z$ is the recentered version of $\varphi$ and $\{S_z\}_{z \in \R^{d+1}}$ is the group of translations acting on $H$. Using the scalar product in $H^{\otimes k}$ rather than in $H_\eps^{\otimes k}$ and points from $\R^{d + 1}$, we assume that the respective modification of the right-hand side of \eqref{e:CovDef} is well defined and we introduce for these kernels the functions $\CK^{(k)}\tau$. In addition, we assume that they satisfy the continuous analogue of \eqref{e:KernelTime} and the first bound in \eqref{e:CovarianceBounds} (when $\eps = 0$). Then for every $\tau \in \hat{\CF}$ we can define a distribution $\Pi_{x}^{t} \tau \in \CS'(\R^d)$ by
\begin{equ}[e:ModelTranslation]
 \langle \Pi_{x}^{t} \tau, \varphi_x \rangle = \sum_{k \leq \VERT \tau \VERT} I_k\Big(S_{(t,x)}^{\otimes k} \int_{\R^{d}} \varphi(y)\, \bigl(\CW^{(k)}\tau\bigr)(t,y)\, dy\Big)\;,
\end{equ}
where $\phi$ is a test function on $\R^d$. In fact, the expression on the right-hand side of \eqref{e:ModelTranslation} is well-defined, because one can show in exactly the same way as in \eqref{e:PiEstimate} that for every test function $\phi$ on $\R^d$ one has
\begin{equ}
\Bigl|\int_{\R^d} \int_{\R^d} \varphi^\lambda_0 (x_1)\, \varphi^\lambda_0 (x_2)\, \big(\CK^{(k)}\tau\big)_0(x_1, x_2)\, d x_1 d x_2\Bigr| \lesssim \lambda^{2 \alpha}\;.
\end{equ}
Finally, defining the maps $\Gamma$ and $\Sigma$ by
\begin{equ}[e:GammaTranslation]
\Gamma^t_{x y} = \tilde{\Gamma}_{(t,x), (t,y)}\;, \qquad \Sigma^{s t}_{x} = \tilde{\Gamma}_{(s,x), (t,x)}\;,
\end{equ}
one can see that $(\Pi, \Gamma, \Sigma)$ is an admissible inhomogeneous model on $\hat \ST$.


\section{Convergence of the discrete dynamical $\Phi^4_3$ model}
\label{s:DPhi}

In this section we use the theory developed above to prove convergence of the solutions of \eqref{e:DPhiRenorm}, where $\Delta^\eps$ is the nearest-neighbour approximation of $\Delta$ and the discrete noise $\xi^\eps$ is defined in \eqref{e:SimpleDNoise} via a space-time white noise $\xi$. 

Example~\ref{ex:Laplacian} yields that Assumption~\ref{a:DOperator} is satisfied, and moreover $\xi^\eps$ is a discrete noise as in \eqref{e:DNoise}. The time-space scaling for the equation \eqref{e:Phi} is $\s=(2,1,1,1)$ and the kernels $K$ and $K^\eps$ are defined in Lemma~\ref{l:DGreenDecomposition} with the parameters $\beta = 2$ and $r > 2$, for the operators $\Delta$ and $\Delta^\eps$ respectively. 

The regularity structure $\ST=(\CT, \CG)$ for the equation \eqref{e:Phi}, introduced in Section~\ref{ss:RegStruct}, has the model space $\CT = \mathrm{span} \{\CF\}$, where
\begin{equs}[e:CFDef]
 \CF = \{\1, \Xi, \Psi, \Psi^2, \Psi^3, \Psi^2 X_i, \CI(\Psi^3) \Psi, \CI(\Psi^3) \Psi^2&, \CI(\Psi^2) \Psi^2, \CI(\Psi^2),\\
&\CI(\Psi) \Psi, \CI(\Psi) \Psi^2, X_i, \ldots \}\;,
\end{equs}
$\Psi \eqdef \CI(\Xi)$, $|\Xi| = \alpha \in \big(-\frac{18}{7}, -\frac{5}{2}\big)$ and the index $i$ corresponds to any of the three spatial dimensions, see \cite[Sec. 9.2]{Hai14} for a complete description of the model space $\CT$. The homogeneities $\CA$ of the symbols in $\CF$ are defined recursively by the rules \eqref{e:defDegree}. The bound $\alpha > -\frac{18}{7}$ is required, in order for the collection of symbols of negative degree generated by the 
procedure of \cite[Sec.~8]{Hai14} not to
depend on $\alpha$.

A two-parameter renormalisation subgroup $\mathfrak{R}^0 \subset \mathfrak{R}$ for this problem consists of the linear maps $M$ on $\CT$ defined in \cite[Equ.~9.3]{Hai14}.

In the proof of Theorem~\ref{t:Phi} in Section~\ref{s:ProofOfPhi} we will make use of the Gaussian models on $\ST$ built in \cite[Thm.~10.22]{Hai14}. As one can see from Remark~\ref{r:ContinousModelLift} and the continuous versions of the bounds \eqref{e:CovarianceBounds}, one can expect a concrete realisation of an abstract symbol $\tau$ to be a function in time if $|\tau| > -\frac{3}{2}$. In our case, the symbols $\Xi$ and $\Psi^3$ don't satisfy this condition, having homogeneities $\alpha < -\frac{5}{2}$ and $3 (\alpha + 2) < -\frac{3}{2}$ respectively. This was exactly the reason for introducing a truncated regularity structure in Section~\ref{sec:trunc}, which primarily means that we can remove these problematic symbols from $\ST$. More precisely, we introduce a new symbol $\bar{\Psi} \eqdef \CI(\Psi^3)$ and the set 
\begin{equ}
\gen{\CF} \eqdef \{\Psi, \bar \Psi\} \cup \poly{\CF}\;.
\end{equ}
Furthermore, we remove $\Xi$ and $\Psi^3$ from $\CF$ in \eqref{e:CFDef} and replace all the occurrences of $\CI(\Psi^3)$ by $\bar \Psi$, which gives
\begin{equ}
\HF = \{\1, \Psi, \Psi^2, \Psi^2 X_i, \Psi \bar{\Psi}, \Psi^2 \bar{\Psi}, \CI(\Psi^2) \Psi^2, \CI(\Psi^2), \CI(\Psi) \Psi, \CI(\Psi) \Psi^2, X_i, \ldots \}\;.
\end{equ}
Then the model spaces of the regularity structures $\gen{\ST}$ and $\hat{\ST}$ from Definition~\ref{d:TruncSets} are the linear spans of $\gen{\CF}$ and $\HF$ respectively, and the set $\hat{\CF}^{-}$ from \eqref{e:TMinus} is given in this case by
\begin{equ}[e:CFMinus]
\hat{\CF}^{-} = \{\Psi,\, \bar{\Psi},\, \Psi^2,\, \Psi^2 X_i,\, \Psi \bar{\Psi},\, \CI(\Psi^2) \Psi^2,\, \Psi^2 \bar{\Psi}\}\;.
\end{equ}

In the following lemma we show that the nonlinearities in \eqref{e:Phi} and \eqref{e:DPhiRenorm} satisfy the required assumptions, provided that the appearance of the renormalisation constant is being dealt with at the level of the corresponding models.

\begin{lemma}\label{l:PhiLipschiz}
Let $\hat \alpha \eqdef \min \hat{\CA}$ and let $a$ and $\lambda$ be as in \eqref{e:Phi}. Then, for any $\gamma > |2 \hat \alpha|$ and any $\eta \leq \hat \alpha$, the maps 
\begin{equ}[e:PhiNonlin]
 F(\tau) = F^\eps(\tau) = - \CQ_{\le 0} \bigl(a \tau + \lambda \tau^3\bigr) + \Xi
\end{equ}
 satisfy Assumptions \ref{a:Nonlin} and \ref{a:DNonlin} with 
\begin{equ}
F_0 = F^\eps_0 = \Xi - \lambda \Psi^3\;, \qquad I_0 = I^\eps_0 = \Psi - \lambda \bar{\Psi}\;,
\end{equ}
and $\bar{\gamma} = \gamma + 2\hat\alpha$, $\bar{\eta} = 3\eta$.
\end{lemma}

\begin{proof}
The space $\CT_\CU \subset \HT$ introduced in Section~\ref{ss:RegStruct} is spanned by polynomials and elements of the form $\CI (\tau)$. Thus, the fact that the function $\hat{F}$ defined in \eqref{e:NonlinAs} maps $\{I_0 + \tau : \tau \in \HT \cap \CT_\CU\}$ into $\HT$ is obvious. The bounds \eqref{e:Lipschitz} in the continuous and discrete cases can be proved in exactly the same way as in \cite[Prop.~6.12]{Hai14}, using Remarks~\ref{r:DistrMult} and \ref{r:DDistrMult} respectively.
\end{proof}

Our following aim is to define a discrete model $Z^\eps = (\Pi^\eps, \Gamma^\eps, \Sigma^\eps)$ on $\gen{\ST}$, and to extend it in the canonical way to $\hat{\ST}$ as in Remark~\ref{r:ModelLift}. To this end, we postulate, for $s, t \in \R$ and $x, y \in \Lambda_\eps^3$,
\begin{equ}[e:ModelOnPsi]
\bigl(\Pi^{\eps, t}_x \Psi\bigr)(y) = \bigl(K^\eps \star_\eps \xi^\eps\bigr)(t,y)\;, \qquad \Gamma^{\eps, t}_{x y} \Psi = \Psi\;, \qquad \Sigma^{\eps, s t}_{x} \Psi = \Psi\;.
\end{equ}
Furthermore, we denote the function $\bar{\psi}^\eps(t,x) \eqdef \bigl(K^\eps \star_\eps \bigl(\Pi^{\eps, t}_x \Psi\bigr)^3\bigr)(t,x)$ and set
\begin{equs}
\bigl(\Pi^{\eps, t}_x \bar{\Psi}\bigr)(y) = \bar{\psi}^\eps(t,y) - \bar{\psi}^\eps&(t,x)\;,\qquad \Gamma^{\eps, t}_{x y} \bar{\Psi} = \bar{\Psi} - \left( \bar{\psi}^\eps(t,y) - \bar{\psi}^\eps(t,x)\right) \1\;,\\
\Sigma^{\eps, s t}_{x} \bar{\Psi} = \bar{\Psi} &- \left( \bar{\psi}^\eps(t,x) - \bar{\psi}^\eps(s,x)\right) \1\;.\label{e:DModelSimple}
\end{equs}
Postulating the actions of these maps on the abstract polynomials in the standard way, we canonically extend $Z^\eps$ to the whole $\hat{\ST}$.

Furthermore, we define the renormalisation constants\footnote{One can show that $C^{(\eps)}_1 \sim \eps^{-1}$ and $C^{(\eps)}_2 \sim \log \eps$.}
\begin{equ}[e:RenormConsts]
C^{(\eps)}_1 \eqdef \int_{\R \times \Lambda_\eps^3} \bigl(K^\eps(z) \bigr)^2\, dz\;, \quad C^{(\eps)}_2 \eqdef 2 \int_{\R \times \Lambda_\eps^3} \bigl(K^\eps \star_\eps K^\eps\bigr)(z)^2\, K^\eps(z)\, dz\;,
\end{equ}
and use them to define the renormalisation map $M^\eps$ as in \cite[Sec. 9.2]{Hai14}. Finally, we define the renormalised model $\hat Z^\eps$ for $Z^\eps$ and $M^\eps$ as in Remark~\ref{r:RenormModel}. Using the model $\hat Z^\eps$ in~\eqref{e:DSPDESolution} we obtain a solution to the discretised $\Phi^4_3$ equation \eqref{e:DPhiRenorm} with
\begin{equ}
C^{(\eps)} \eqdef 3 \lambda C^{(\eps)}_1 - 9 \lambda^2 C^{(\eps)}_2\;,
\end{equ}
where $\lambda$ is the coupling constant from \eqref{e:Phi}. Before giving a proof of Theorem~\ref{t:Phi} we provide some technical results.

\subsection{Discrete functions with prescribed singularities}
\label{ss:SingularFunctions}

It follows from Proposition~\ref{p:CovarianceConvergence} that the ``strength'' of singularity of a kernel determines the regularity of the respective distribution. In this section we provide some properties of singular discrete functions. As usual we fix a scaling $\s=(\s_0, 1, \ldots, 1)$ of $\R^{d+1}$ with $\s_0 \geq 1$.

For a function $K^\eps$ defined on $\R \times \Lambda_\eps^d$ and supported in a ball centered at the origin, we denote by $D_{i, \eps}$ the finite difference derivative, i.e.
\begin{equ}
D_{i, \eps} K^\eps(t,x) \eqdef \eps^{-1} \left( K^\eps(t, x + \eps e_i) - K^\eps(t,x) \right)\;,
\end{equ}
where $\{e_i\}_{i = 1 \ldots d}$ is the canonical basis of $\R^d$, and for $k = (k_0, k_1, \ldots, k_d) \in \N^{d+1}$ we define $D_\eps^k \eqdef D_t^{k_0} D^{k_1}_{1,\eps} \ldots D^{k_d}_{d,\eps}$. We allow the function $K^\eps$ to be non-differentiable in time only on the set $P_0 \eqdef \{(0, x) : x \in \Lambda_\eps^d\}$. Furthermore, we define for $\zeta \in \R$ and $m \geq 0$ the quantity
\begin{equ}[e:SingularKernel]
\wnorm{K^\eps}_{\zeta; m} \eqdef \max_{|k|_\s \leq m} \sup_{z \notin P_0} \frac{\bigl| D^{k}_{\eps} K^\eps(z) \bigr|}{\senorm{z}^{(\zeta - |k|_\s) \wedge 0}}\;,
\end{equ}
where $z \in \R \times \Lambda_\eps^d$, $k \in \N^{d + 1}$ and $\senorm{z} \eqdef \snorm{z} \vee \eps$.

 By analogy with Remark~\ref{r:Uniformity}, we always consider a sequence of functions parametrised by $\eps = 2^{-N}$ with $N \in \N$, and we assume the bounds  to hold for all $\eps$ with proportionality constants independent of $\eps$. Thus, if $\wnorm{K^\eps}_{\zeta; m} < \infty$, then we will say that $K^\eps$ is of order $\zeta$. 
 
\begin{remark}\label{r:DiscontinuousFunctions}
We stress the fact that by our assumptions the functions $K^\eps$ are defined also at the origin. In particular, $K^\eps$ can have a discontinuity at $t = 0$ and its time derivative behaves in the worst case as the Dirac delta function at the origin.
\end{remark}
 
 The following result provides bounds on products and discrete convolutions $\star_\eps$.

\begin{lemma}\label{l:FuncsConv}
Let functions $K^\eps_1$ and $K^\eps_2$ be of orders $\zeta_1$ and $\zeta_2$ respectively. Then we have the following results:
\begin{itemize}
 \item If $\zeta_1, \zeta_2 \leq 0$, then $K^\eps_1 K^\eps_2$ is of order $\zeta_1 + \zeta_2$ and for every $m \geq 0$ one has
\begin{equ}[e:FuncsProd]
\wnorm{K^\eps_1 K^\eps_2}_{\zeta_1 + \zeta_2; m} \lesssim \wnorm{K^\eps_1}_{\zeta_1; m} \wnorm{K^\eps_2}_{\zeta_2; m}\;.
\end{equ}
Moreover, if both $K^\eps_1$ and $K^\eps_2$ are continuous in the time variable on whole $\R$, then $K^\eps_1 K^\eps_2$ is continuous as well.
\item If $\zeta_1 \wedge \zeta_2 > - |\s|$ and $\bar{\zeta} \eqdef \zeta_1 + \zeta_2 + |\s| \notin \N$, then $K^\eps_1 \star_\eps K^\eps_2$ is continuous in the time variable and one has the bound
\begin{equ}[e:FuncsConv]
\wnorm{K^\eps_1 \star_\eps K^\eps_2}_{\bar{\zeta}; m} \lesssim \wnorm{K^\eps_1}_{\zeta_1; m} \wnorm{K^\eps_2}_{\zeta_2; m}\;.
\end{equ}
\end{itemize}
In all these estimates the proportionality constants depend only on the support of the functions $K^\eps_i$ and are independent of $\eps$.
\end{lemma}

\begin{proof}[Proof of Lemma~\ref{l:FuncsConv}]
The bound \eqref{e:FuncsProd} follows from the Leibniz rule for the discrete derivative:
\begin{equ}[e:DLeibniz]
D_\eps^k \bigl(K^\eps_1 K^\eps_2\bigr)(z) = \sum_{l \leq k} \binom{k}{l} D_\eps^l K^\eps_1(z)\, D_\eps^{k-l} K^\eps_2\bigl(z + (0, \eps l)\bigr)\;,
\end{equ}
where $k, l \in \N^d$, as well as from the standard Leibniz rule in the time variable. The bound \eqref{e:FuncsConv} can be proved similarly to \cite[Lem.~10.14]{Hai14}, but using the Leibniz rule \eqref{e:DLeibniz}, summation by parts for the discrete derivative and the fact that the products 
\begin{equ}
(x)_{k, \eps} \eqdef \prod_{i=1}^d \prod_{0 \leq j < k_i}(x_i - \eps j)
\end{equ}
with $k \in \N^d$ play the role of polynomials for the discrete derivative. 

When bounding the time derivative of $K^\eps_1 \star_\eps K^\eps_2$, we convolve in the worst case a function which behaves as Dirac's delta at the origin with another one which has a jump there (see Remark~\ref{r:DiscontinuousFunctions}). This operation gives us a function whose derivative can have a jump at the origin, but is not Dirac's delta. This fact explains why $K^\eps_1 \star_\eps K^\eps_2$ is continuous in time.
\end{proof}

The following lemma, whose proof is almost identical to that of \cite[Lem.~10.18]{Hai14}, provides a bound on an increment of a singular function.

\begin{lemma}\label{l:FuncIncr}
Let a function $K^\eps$ be of order $\zeta \leq 0$. Then for every $\kappa \in [0,1]$, $t \in \R$ and $x_1, x_2 \in \Lambda_\eps^d$ one has
\begin{equ}
 |K^\eps(t, x_1) - K^\eps(t, x_2)| \lesssim |x_1 - x_2|^\kappa \Bigl(\senorm{t, x_1}^{\zeta - \kappa} + \senorm{t, x_2}^{\zeta - \kappa} \Bigr) \wnorm{K^\eps}_{\zeta; 1}\;.
\end{equ}
\end{lemma}

For a discrete singular function $K^\eps$, we define the function $\mathscr{R}_\eps K^\eps$ by
\begin{equ}[e:FuncRenorm]
\left(\mathscr{R}_\eps K^\eps\right)(\varphi) \eqdef \int_{\R \times \Lambda_\eps^d} K^\eps(z) \left( \varphi(z) - \varphi(0) \right)\, dz\;,
\end{equ}
for every compactly supported test function $\varphi$ on $\R^{d+1}$. The following result can be proved similarly to \cite[Lem.~10.16]{Hai14} and using the statements from the proof of Lemma~\ref{l:FuncsConv}.

\begin{lemma}\label{l:FuncRenorm}
Let functions $K^\eps_1$ and $K^\eps_2$ be of orders $\zeta_1$ and $\zeta_2$ respectively with $\zeta_1 \in \bigl(-|\s|-1, -|\s|\bigr]$ and $\zeta_2 \in \bigl(-2 |\s|-\zeta_1, 0\bigr]$. Then the function $\bigl(\mathscr{R}_\eps K^\eps_1\bigr) \star_\eps K^\eps_2$ is continuous in time of order $\bar{\zeta} \eqdef \zeta_1 + \zeta_2 + |\s|$ and, for any $m \geq 0$, one has
\begin{equ}
\wnorm{\bigl(\mathscr{R}_\eps K^\eps_1\bigr) \star_\eps K^\eps_2}_{\bar{\zeta}; m} \lesssim \wnorm{K^\eps_1}_{\zeta_1; m} \wnorm{K^\eps_2}_{\zeta_2; m + \s_0}\;.
\end{equ}
\end{lemma}

The following result shows how certain convolutions change singular functions. Its proof is similar to \cite[Lem.~10.17]{Hai14}.

\begin{lemma}\label{l:SmoothConv}
Let for some $\bar \eps \in [\eps, 1]$ the function $\psi^{\bar \eps, \eps} : \R \times \Lambda_\eps^d \to \R$ be smooth in the time variable, supported in the ball $B(0, R \bar \eps) \subset \R^{d+1}$ for some $R \geq 1$, and satisfies 
\begin{equ}[e:SmoothConv]
\int_{\R \times \Lambda_\eps^d} \psi^{\bar \eps, \eps}(z)\, dz = 1\;, \qquad |D^k_\eps \psi^{\bar \eps, \eps}(z)| \lesssim \bar{\eps}^{-|\s| - |k|_\s}\;,
\end{equ}
for all $z \in \R \times \Lambda_\eps^d$ and $k \in \N^{d+1}$, where the proportionality constant in the bound can depend on $k$. If $K^\eps$ is of order $\zeta \in (-|\s|, 0)$, then for all $\kappa \in (0,1]$ one has
\begin{equ}
\wnorm{K^\eps - K^\eps \star_\eps \psi^{\bar \eps, \eps}}_{\zeta - \kappa; m} \lesssim \bar{\eps}^\kappa \wnorm{K^\eps}_{\zeta; m + \s_0}\;.
\end{equ}
\end{lemma}


\subsection{Convergence of lattice approximations of the $\Phi^4_3$ measure}

In this section we provide some properties of the lattice approximations $\mu_\eps$ of the $\Phi^4_3$ measure, defined in \eqref{e:mu_eps}, which will be used in the proof of Corollary~\ref{c:Phi}. We start with tightness and moment estimates.

\begin{proposition}\label{prop:mu_tight}
If $a > 0$ and the coupling constant $\lambda$ in \eqref{e:DAction} is small enough, then for every $\alpha < -\frac{1}{2}$ the sequence $\mu_\eps$ is tight in $\CC^\alpha$ as $\eps \to 0$ with uniformly bounded
moments of all orders.
\end{proposition}

\begin{proof}
The estimate \cite[Eq.~8.2]{BFS83} implies that the $2 n$-th moment of $\mu_\eps$ is bounded by the second moment (up to a multiplier depending on $n$). Moreover, it follows from \cite[Thm.~6.1]{BFS83} that for any test function $\phi \in \CC_0^{\infty}(\R^3)$ one has
\begin{equ}
\int \Phi^\eps(\phi)^2 \mu_\eps(d \Phi^\eps) = \int \Phi^\eps(\phi)^2 \hat{\mu}_\eps(d \Phi^\eps) + \mathcal{O} \bigl(\lambda^2 \Vert \phi \Vert^2_{L^2}\bigr)\;,
\end{equ}
where $\hat{\mu}_\eps$ is the Gaussian measure given by \eqref{e:mu_eps} and \eqref{e:DAction} with $\lambda = C^{(\eps)} = 0$. Since the covariance of $\hat{\mu}_\eps$ is the kernel of $(a - \Delta^\eps)^{-1}$ where $\Delta^\eps$ is the nearest-neighbour approximation of the Laplacian $\Delta$ (see \cite[Eq.~3.2]{BFS83}), one has the bound
\begin{equ}
 \int \Phi^\eps(\phi^{\nu})^2 \hat{\mu}_\eps(d \Phi^\eps) \lesssim \nu^{-1 - \kappa}\;,
\end{equ}
for any $\kappa > 0$ and any scaling parameter $\nu \in [\eps, 1]$. This yields the respective bounds on the moments of $\mu_\eps$ from which the claim follows.
\end{proof}

The following result shows that the measures $\mu_\eps$ in fact converge as $\eps \to 0$.

\begin{proposition}\label{eq:mu_limit}
The measures $\mu_\eps$ on $\CC^\alpha$ converge to the $\Phi^4_3$ measure \eqref{e:PhiMeasure}.
\end{proposition}

\begin{proof}
By Proposition~\ref{prop:mu_tight}, we can choose a subsequence of $\mu_\eps$ weakly converging 
to a limit $\mu$. Combining this with 
\cite[Thm~2.1]{Park} (see also \cite{ParkOld})
shows that $\mu$ coincides with the $\Phi^4_3$ measure \eqref{e:PhiMeasure} constructed
in \cite{Fel74}.
\end{proof}


\subsection{Proof of the convergence result}
\label{s:ProofOfPhi}

Using the results from the previous section, we are ready to prove Theorem~\ref{t:Phi}.

\begin{proof}[Proof of Theorem~\ref{t:Phi}]
In order to prove the claim, we proceed as in \cite{Park} and introduce intermediate equations driven by a smooth noise. Precisely, we take a function $\psi : \R^{4} \to \R$ which is smooth, compactly supported and integrates to $1$, and for some $\bar{\eps} \in [\eps,1]$ we define $\psi^{\bar{\eps}}(t, x) \eqdef \bar{\eps}^{-|\s|} \psi\bigl(\bar{\eps}^{-2} t, \bar{\eps}^{-1} x\bigr)$ and the mollified noise $\xi^{\bar \eps, 0} \eqdef \xi \star \psi^{\bar{\eps}}$. Then we denote by $\Phi^{\bar \eps, 0}$ the global solution of 
\begin{equ}
\partial_t \Phi^{\bar \eps, 0} = \Delta \Phi^{\bar \eps, 0} + \bigl(C^{(\bar \eps, 0)} - a\bigr) \Phi^{\bar \eps, 0} - \lambda \bigl(\Phi^{\bar \eps, 0}\bigr)^3 + \xi^{\bar \eps, 0}\;, \qquad \Phi^{\bar \eps, 0}(0, \cdot) = \Phi_0(\cdot)\;,
\end{equ}
where $C^{(\bar \eps, 0)} = 3 \lambda C^{(\bar \eps, 0)}_1 - 9 \lambda^2 C^{(\bar \eps, 0)}_2$, and $C^{(\bar \eps, 0)}_1$ and $C^{(\bar \eps, 0)}_2$ are as in \cite[Thm.~10.22 and Eq. 9.21]{Hai14}.

Let $\tilde{Z}^{\bar \eps, 0}$ and $\tilde{Z}$ be the models on $\ST$ built in \cite[Thm.~10.22]{Hai14} via the noises $\xi^{\bar \eps, 0}$ and $\xi$ respectively. We will be interested only in their restrictions to the truncated regularity structure $\hat{\ST}$. It follows from the proof of the latter theorem that we are exactly in the setting of Section~\ref{ss:ContinuousGauss}, and we can define respective inhomogeneous models $\hat Z^{\bar \eps, 0}$ and $\hat Z$ on $\hat{\ST}$ as in \eqref{e:ModelTranslation} and \eqref{e:GammaTranslation}. Furthermore, Remark~\ref{r:ContinousModelLift} and the bounds obtained in the proof of \cite[Thm.~10.22]{Hai14} on the elements in the expansions \eqref{e:ModelTranslation} of the models yield the following bounds:
\begin{equ}[e:ModelConvFirst]
\EE \left[\VERT \hat Z \VERT_{\delta, \gamma; T}\right]^p \lesssim 1\;, \qquad \EE \left[\VERT \hat Z^{\bar \eps, 0}; \hat Z \VERT_{\delta, \gamma; T}\right]^p \lesssim \bar{\eps}^{\theta p}\;,
\end{equ}
uniformly in $\bar \eps \in (0,1]$, for any $T > 0$, $p \geq 1$ and for sufficiently small values of $\delta > 0$ and $\theta > 0$. Using Theorem~\ref{t:FixedMap} and Lemma~\ref{l:PhiLipschiz}, we define the solution $\Phi$ to the equation \eqref{e:Phi} as in Definition~\ref{d:SPDESolution} by solving the respective abstract equation \eqref{e:AbstractEquation} with the nonlinearity $F$ from \eqref{e:PhiNonlin} and the inhomogeneous model $\hat Z$.

In order to discretise the noise $\xi^{\bar \eps, 0}$, we define the function
\begin{equ}
\psi^{\bar{\eps}, \eps}(t, x) \eqdef \eps^{-d} \int_{\R^{d}} \psi^{\bar{\eps}}(t, y)\, \1_{|y - x| \leq \eps/2}\, dy\;, \qquad (t,x) \in \R \times \Lambda_\eps^d\;,
\end{equ}
and the discrete noise $\xi^{\bar \eps, \eps} \eqdef \psi^{\bar{\eps}, \eps} \star_\eps \xi^\eps$, where $\xi^\eps$ is given in \eqref{e:SimpleDNoise}. We take the function $\psi^{\bar{\eps}, \eps}$ in this form, because it satisfies the first identity in \eqref{e:SmoothConv}, which in general is not true for $\psi^{\bar{\eps}}$. We define the discrete model $Z^{\bar \eps, \eps}$ by substituting each occurrence of $\xi^\eps$, $C^{(\eps)}_1$ and $C^{(\eps)}_2$ in the definition of $Z^{\eps}$ by $\xi^{\bar \eps, \eps}$, $C^{(\bar \eps, \eps)}_1$ and $C^{(\bar \eps, \eps)}_2$ respectively, where $C^{(\bar \eps, \eps)}_1$ is defined as in \eqref{e:RenormConsts}, but via the kernel $K^{\bar \eps, \eps} \eqdef K^\eps \star_\eps \psi^{\bar \eps, \eps}$, and $C^{(\bar \eps, \eps)}_2$ is defined by replacing $K^\eps \star_\eps K^\eps$ by $K^{\bar \eps, \eps} \star_\eps K^{\bar \eps, \eps}$ in the second expression in \eqref{e:RenormConsts}. Furthermore, using $\wnorm{K^\eps}_{-3; r} \leq C$, which follows from Lemma~\ref{l:DGreenDecomposition} and Remark~\ref{r:KernelDiff}, and proceeding exactly as in the proof of \cite[Thm.~10.22]{Hai14}, but exploiting Proposition~\ref{p:CovarianceConvergence} and the results from Section~\ref{ss:SingularFunctions} instead of their continuous counterparts, we obtain the bounds \eqref{e:GaussModelBoundSigmaGamma} for each $\tau \in \gen{\CF} \setminus \poly{\CF}$, and \eqref{e:GaussModelBound} for each $\tau \in \hat{\CF}^{-}$, uniformly in $\eps \leq \bar \eps$ and for $\delta > 0$ small enough. We also obtain the respective bounds on the differences $Z^{\bar \eps, \eps} - Z^{\eps}$, with the proportionality constants of orders $\bar \eps^{2\theta}$ with $\theta > 0$ sufficiently small. For this, we can use Lemma~\ref{l:SmoothConv}, because $\psi^{\bar \eps, \eps}$ satisfies the required conditions, which follows from the properties of $\psi$. Thus, Theorem~\ref{t:ModelsConvergence} yields
\begin{equ}[e:DPhiModelBound]
\EE \left[\VERT Z^\eps \VERT_{\delta, \gamma; T}^{(\eps)}\right]^p \lesssim 1\;, \qquad \EE \left[\VERT Z^{\bar \eps, \eps}; Z^\eps \VERT_{\delta, \gamma; T}^{(\eps)}\right]^p \lesssim \bar{\eps}^{\theta p}\;,
\end{equ}
uniformly in $\eps \leq \bar \eps$, for any $T > 0$ and $p \geq 1$. We denote by $\Phi^{\bar \eps, \eps}$ the solution of \eqref{e:DPhiRenorm}, driven by the noise $\xi^{\bar \eps, \eps}$, with the renormalisation constant $C^{(\bar \eps, \eps)} \eqdef 3 \lambda C^{(\bar \eps, \eps)}_1 - 9 \lambda^2 C^{(\bar \eps, \eps)}_2$. 

For every $K > 0$ we define the following stopping time:
\begin{equ}
 \tau_K \eqdef \inf\bigl\{T > 0 : \Vert \Phi \Vert_{\CC^{\delta, \alpha}_{\bar{\eta}, T}} \geq K\bigr\}\;,
\end{equ}
where the values of $\delta$, $\alpha$ and $\bar \eta$ are as in the statement of the theorem. Then we have the limit in probability $\lim_{K \to \infty} \tau_K = T_\star$, where $T_\star$ is the random lifetime of $\Phi$. Our aim is now to prove that
\begin{equ}[e:SolutionsConvDouble]
\lim_{K \to \infty} \lim_{\eps \to 0} \P \Bigl[\Vert \Phi; \Phi^\eps \Vert^{(\eps)}_{\CC^{\delta, \alpha}_{\bar{\eta}, \tau_K}} \geq c\Bigr] = 0\;,
\end{equ}
for every constant $c > 0$. Then the claim \eqref{e:PhiConvergence} will follow after choosing $T_\eps$ as a suitable diagonal sequence. 

In order to have a priori bounds on the processes and models introduced above, we define for every $K >0$ the following stopping times:
\begin{equs}
\sigma^\eps_{K} &\eqdef \inf\bigl\{T > 0 : \Vert \Phi \Vert_{\CC^{\delta, \alpha}_{\bar{\eta}, T}} \geq K ~\text{ or }~ \VERT \hat Z \VERT_{\delta, \gamma; T} \geq K\,, \text{ or }~ \VERT \hat Z^\eps \VERT^{(\eps)}_{\delta, \gamma; T} \geq K\bigr\}\;,\\
\sigma^{\bar \eps, \eps} &\eqdef \inf\bigl\{T > 0 : \Vert \Phi - \Phi^{\bar \eps, 0} \Vert_{\CC^{\delta, \alpha}_{\bar{\eta}, T}} \geq 1 ~\text{ or }~ \Vert \Phi^\eps - \Phi^{\bar \eps, \eps} \Vert^{(\eps)}_{\CC^{\delta, \alpha}_{\bar{\eta}, T}} \geq 1\,, \\
& ~\text{ or }~ \Vert \Phi^{\bar \eps, 0}; \Phi^{\bar \eps, \eps} \Vert^{(\eps)}_{\CC^{\delta, \alpha}_{\bar{\eta}, T}} \geq 1\,, ~\text{ or }~ \VERT \hat Z; \hat Z^{\bar \eps, 0} \VERT_{\delta, \gamma; T} \geq 1\,, ~\text{ or }~ \VERT \hat Z^\eps; \hat Z^{\bar \eps, \eps} \VERT^{(\eps)}_{\delta, \gamma; T} \geq 1\bigr\}\;,
\end{equs}
as well as $\varrho_K^{\bar \eps, \eps} \eqdef \sigma^\eps_{K} \wedge \sigma^{\bar \eps, \eps}$. Then, choosing two constants $\bar K > K$ and using the latter stopping time and the triangle inequality, we get the following bound:
\begin{equs}
 \P \Bigl[\Vert \Phi; \Phi^\eps &\Vert^{(\eps)}_{\CC^{\delta, \alpha}_{\bar{\eta}, \tau_K}} \geq c\Bigr] \leq \P \Bigl[\Vert \Phi - \Phi^{\bar \eps, 0} \Vert_{\CC^{\delta, \alpha}_{\bar{\eta}, \varrho_{\bar K}^{\bar \eps, \eps}}} \geq c\Bigr] + \P \Bigl[\Vert \Phi^{\bar \eps, 0}; \Phi^{\bar \eps, \eps} \Vert^{(\eps)}_{\CC^{\delta, \alpha}_{\bar{\eta}, \varrho_{\bar K}^{\bar \eps, \eps}}} \geq c\Bigr]\\
 &+ \P \Bigl[\Vert \Phi^{\bar \eps, \eps} - \Phi^{\eps} \Vert^{(\eps)}_{\CC^{\delta, \alpha}_{\bar{\eta}, \varrho_{\bar K}^{\bar \eps, \eps}}} \geq c\Bigr] + \P \bigl[\varrho_{\bar K}^{\bar \eps, \eps} < \sigma^\eps_{\bar K}\bigr] + \P\bigl[\sigma^\eps_{\bar K} < \tau_K\bigr]\;. \label{e:SolutionsConv}
\end{equs}
We will show that if we take the limits $\eps, \bar \eps \to 0$ and $K, \bar K \to \infty$, then all the terms on the right-hand side of \eqref{e:SolutionsConv} vanish and we obtain the claim \eqref{e:SolutionsConvDouble}.

It follows from the definition of $\varrho_{\bar K}^{\bar \eps, \eps}$ that $\VERT \hat Z \VERT_{\delta, \gamma; \varrho_{\bar K}^{\bar \eps, \eps}}$ and $\VERT\hat Z^{\bar \eps, 0} \VERT_{\delta, \gamma; \varrho_{\bar K}^{\bar \eps, \eps}}$ are bounded by constants proportional to $\bar K$. Hence, Theorems~\ref{t:DSolutions} and \ref{t:DReconstruct}, and the bounds~\eqref{e:ModelConvFirst} yield
\begin{equ}
\lim_{\bar \eps \to 0} \P \Bigl[\Vert \Phi - \Phi^{\bar \eps, 0} \Vert_{\CC^{\delta, \alpha}_{\bar{\eta}, \varrho_{\bar K}^{\bar \eps, \eps}}} \geq c\Bigr] = 0\;,
\end{equ}
uniformly in $\eps$. Similarly, we can use Theorems~\ref{t:DSolutions} and \ref{t:DReconstruct}, and the bounds on the discrete models \eqref{e:DPhiModelBound} to obtain the uniform in $\eps$ convergence
\begin{equ}
\lim_{\bar \eps \to 0} \P \Bigl[\Vert \Phi^\eps - \Phi^{\bar \eps, \eps} \Vert^{(\eps)}_{\CC^{\delta, \alpha}_{\bar{\eta}, \varrho_{\bar K}^{\bar \eps, \eps}}} \geq c\Bigr] = 0\;.
\end{equ}

Now, we turn to the second term in \eqref{e:SolutionsConv}. It follows from our definitions that we have $\xi^{\bar \eps, \eps} = \varrho^{\bar \eps, \eps} \star \xi$, where 
\begin{equ}
\varrho^{\bar \eps, \eps}(t,x) \eqdef \eps^{-d} \int_{\Lambda_{\eps}^d} \psi^{\bar \eps, \eps}(t,y)\, \1_{|y - x| \leq \eps/2}\, dy\;.
\end{equ}
Moreover, for $z = (t,x) \in \R \times \Lambda_{\eps}^d$ one has the identity
\begin{equ}
\bigl(\psi^{\bar \eps} - \varrho^{\bar \eps, \eps}\bigr)(z) = \eps^{-2d} \int_{\Lambda_\eps^d} \int_{\R^{d}} \left( \psi^{\bar \eps}(t,x) - \psi^{\bar \eps}(t,u) \right) \1_{|u-y| \leq \eps/2} \1_{|y - x| \leq \eps/2} d u\, dy,
\end{equ}
from which we immediately obtain the bound
\begin{equ}
\sup_{z \in \R \times \Lambda_{\eps}^d} \bigl| D_t^k \bigl(\psi^{\bar \eps} - \varrho^{\bar \eps, \eps}\bigr)(z)\bigr| \lesssim \eps \bar{\eps}^{-|\s| - k \s_0 - 1}\;,
\end{equ}
for every $k \in \N$. Hence, using the a priori bounds on the solutions, which follow from the definition of $\varrho_{\bar K}^{\bar \eps, \eps}$, we can use the standard result from numerical analysis of PDEs (see e.g. \cite[Ch.~6]{Lui11}) that the second term in \eqref{e:SolutionsConv} vanishes as $\eps \to 0$, as soon as $\bar \eps$ is fixed.

The limit $\lim_{\bar \eps \to 0} \lim_{\eps \to 0} \P \bigl[\varrho_{\bar K}^{\bar \eps, \eps} < \sigma^\eps_{\bar K}\bigr] = 0$ follows immediately from the definition of the involved stopping times, the bounds \eqref{e:ModelConvFirst} and \eqref{e:DPhiModelBound}, and the convergences we have just proved. Finally, it follows from \eqref{e:ModelConvFirst} that 
\begin{equ}
 \lim_{\bar K \to \infty} \P\bigl[\sigma^\eps_{\bar K} < \tau_K\bigr] = 0\;,
\end{equ}
for a fixed $K$ and uniformly in $\eps$, which finishes the proof.
\end{proof}

\begin{proof}[Proof of Corollary~\ref{c:Phi}]
Let $\xi$ be space-time white noise on some probability space 
$(\Omega, \CF, \mathbf{P})$, and let its discretisation $\xi^\eps$ be given 
by \eqref{e:SimpleDNoise}. Let furthermore $\Phi^\eps_0$ be a random variable on 
the same probability space which is independent of $\xi$ and 
such that the solution to \eqref{e:DPhiRenorm} with 
the nearest neighbours approximate Laplacian $\Delta^\eps$ and driven by $\xi^\eps$ 
is stationary. We denote by $\mu_\eps$ its stationary distribution \eqref{e:mu_eps}, which we view
as a measure on $\CC^{\alpha}$ with $\alpha$ as in \eqref{e:PhiConvergence}, 
by extending it in a piecewise constant fashion. It then 
follows from Proposition~\ref{prop:mu_tight} that if we view $\Phi^\eps_0$ as an element of $\CC^{\alpha}$ 
by piecewise constant extension, we can and will assume by Skorokhod's representation 
theorem that 
$\Phi^\eps_0$ converges almost surely as $\eps \to 0$ to a limit $\Phi_0 \in \CC^\alpha$. In order to use Skorokhod's representation 
theorem \cite{Kal02}, the underlying spaces have to be separable which isn't the case for $\CC^{\alpha}$, but this
is irrelevant since our random variables belong almost surely to the closure of smooth functions under the seminorm \eqref{e:AlphaNorm} which is separable.

Before we proceed, we introduce the space $\bar \CC \eqdef \CC^{0, \alpha}_{\bar \eta}\bigl([0,1], \Torus^3\bigr) \cup \{\infty\}$ (the latter H\"{o}lder space is a subspace of $\CC^{0, \alpha}_{\bar \eta}\bigl([0,1], \R^3\bigr)$ defined below \eqref{e:SpaceTimeNorm}, containing the spatially periodic distributions), for $\alpha$ and $\bar \eta$ as in \eqref{e:PhiConvergence}, and equipped with the  metric such that 
\begin{equs}
d (\zeta, \infty) &\eqdef d (\infty, \zeta) \eqdef \bigl( 1 + \Vert \zeta \Vert_{\CC^{0, \alpha}_{\bar \eta, 1}} \bigr)^{-1}\;, \quad \zeta \neq \infty\;,\\
 d (\zeta_1, \zeta_2) &\eqdef \min \bigl\{ \Vert \zeta_1 - \zeta_2 \Vert_{\CC^{0, \alpha}_{\bar \eta, 1}}, d (\zeta_1, \infty) + d (\zeta_2, \infty) \bigr\}\;, \quad \zeta_i \neq \infty\;.
\end{equs}
Denote now by $\Phi^\eps$ the  
solution to \eqref{e:DPhiRenorm} with initial condition
$\Phi^\eps_0$ and by $\Phi$ the solution to \eqref{e:Phi} with initial condition $\Phi_0$.
We can view these as $\bar \CC$-valued random variables by postulating that
$\Phi = \infty$ if its lifetime is smaller than $1$. (The lifetime of $\Phi^\eps$ 
is always infinite for fixed $\eps$.) 

Since the assumptions of Theorem~\ref{t:Phi} are fulfilled, 
the convergence \eqref{e:PhiConvergence} holds and, since solutions blow up at time
$T_\star$, this implies that 
$d(\Phi^\eps, \Phi) \to 0$ in probability, as $\eps \to 0$. (The required continuity in time obviously holds for every $\Phi^\eps$ and $\Phi$.) 
In order to conclude, it remains to show that $\P(\Phi = \infty) = 0$.
In particular, since the only point of discontinuity of the evaluation 
maps $\Phi \mapsto \Phi(t,\cdot)$ 
on $\bar \CC$ is $\infty$, this would then immediately show not only that 
solutions $\Phi$ live up to time $1$ (and therefore any time) almost surely,
but also that $\mu$ is invariant for $\Phi$. To show that $\Phi \neq \infty$ a.s., it suffices to prove that there is no atom of the measure $\mu$ at the point $\infty$. Precisely, our aim is to show that for every $\bar \eps > 0$ there exists a constant $C_{\bar \eps} > 0$ such that 
\begin{equ}[e:Tightness]
\P\Bigl( \Vert \Phi^\eps \Vert_{\CC^{0, \alpha}_{\bar{\eta}, 1}} \geq C_{\bar \eps} \Bigr) \leq \bar \eps\;.
\end{equ}

We fix $\bar \eps > 0$ in what follows and work with a generic constant $C_{\bar \eps} > 0$, whose value will be chosen later. For integers $K \ge 2$ and $i \in \{0,\ldots,K-2\}$, we denote
\begin{equs}
 Q^\eps_{K, i} \eqdef \Vert \Phi^\eps \Vert_{\CC^{0, \alpha}_{\bar{\eta}, [i/K, (i + 2)/K]}}\;,
\end{equs}
where the norm $\Vert \cdot \Vert_{\CC^{0, \alpha}_{\bar{\eta}, [T_1, T_2]}}$ is defined as below \eqref{e:SpaceTimeNorm}, but on the time interval $[T_1, T_2]$ and with a blow-up at $T_1$. Splitting the time interval $(0,1]$ in \eqref{e:SpaceTimeNorm} into subintervals of length $1/K$, and deriving estimates on each subinterval, one gets
\begin{equs}
\Vert \Phi^\eps \Vert_{\CC^{0, \alpha}_{\bar{\eta}, 1}} &\leq Q_{K, 0}^\eps + \sum_{i = 1}^{K - 1} (i+1)^{-\bar \eta / 2}\, Q^\eps_{K, i-1}
\leq \tilde C K^{-\bar \eta / 2} \sum_{i = 0}^{K-2} Q^\eps_{K, i}\;,
\end{equs}
if $\bar \eta \leq 0$, and for some $\tilde C$ independent of $K$ and $\eps$. 
Since, by stationarity, the random variables $Q^\eps_{K, i}$ all have the same law,
it follows that 
\begin{equs}
\P \Bigl( \Vert \Phi^\eps \Vert_{\CC^{0, \alpha}_{\bar{\eta}, 1}} \geq C_{\bar \eps} \Bigr) &\leq \P\Bigl( \tilde C K^{-\bar \eta / 2} \sum_{i = 0}^{K-2} Q^\eps_{K, i} \geq C_{\bar \eps}\Bigr)\\
&\leq K \P \Bigl( \Vert \Phi^\eps \Vert_{\CC^{0, \alpha}_{\bar{\eta}, 2/K}} \geq \tilde{C}^{-1} K^{\bar \eta / 2} C_{\bar \eps}\Bigr)\;,\label{e:TightBoundOne}
\end{equs}
To make the notation concise, we write $\tilde{C}_{K, \bar \eps} \eqdef \tilde{C}^{-1} K^{\bar \eta / 2} C_{\bar \eps}$. Furthermore, in order to have a uniform bound on the initial data and the model, we use the following estimate
\begin{equs}
\mathbf{P} \Bigl( \Vert \Phi^\eps \Vert_{\CC^{0, \alpha}_{\bar{\eta}, 2/K}} \geq \tilde{C}_{K, \bar \eps} \Bigr) \leq \mathbf{P}&\Bigl( \Vert \Phi^\eps \Vert_{\CC^{0, \alpha}_{\bar{\eta}, 2/K}} \geq \tilde{C}_{K, \bar \eps} \Big| \Vert \Phi^\eps_0 \Vert_{\CC^{\eta}} \leq L, \VERT Z^\eps \VERT^{(\eps)}_{\gamma; 1} \leq L \Bigr) \\
 &+ \mathbf{P}\Bigl(  \Vert \Phi^\eps_0 \Vert_{\CC^{\eta}} > L \Bigr) + \mathbf{P}\Bigl( \VERT Z^\eps \VERT^{(\eps)}_{\gamma; 1} > L \Bigr),\label{e:PhiInitDataBound}
\end{equs}
valid for every $L$, where $\eta$ and $\gamma > 0$ are as in the proof of Theorem~\ref{t:Phi}. 

Recalling that \cite[Sec.~8]{BFS83} yields uniform bounds on all moments of $\mu_\eps$, 
and using the first bound in \eqref{e:DPhiModelBound}, Markov's inequality implies that
\begin{equ}[e:twoterms]
\mathbf{P}\Bigl(  \Vert \Phi^\eps_0 \Vert_{\CC^{\eta}} > L \Bigr) \leq B_1 L^{-q}\;, \qquad \mathbf{P}\Bigl( \VERT Z^\eps \VERT^{(\eps)}_{\gamma; 1} > L \Bigr) \leq B_2 L^{-q}\;,
\end{equ}
for any $q \geq 1$, and for constant $B_1$ and $B_2$ independent of $\eps$ and $L$.

Turning to the first term in \eqref{e:PhiInitDataBound}, it follows from the fixed point argument in the proof of Theorem~\ref{t:DSolutions} and the bound \eqref{e:DReconstructSpace}, that there exists $\tilde p \geq 1$ such that one has the bound
\begin{equ}
\Vert \Phi^\eps \Vert_{\CC^{0, \alpha}_{\bar{\eta}, 2/K}} \leq B_3 L^{3}\;,
\end{equ}
with $B_3$ being independent of $\eps$ and $L$, as soon as $\Vert \Phi^\eps_0 \Vert_{\CC^{\eta}} \leq L$, $\VERT Z^\eps \VERT^{(\eps)}_{\gamma; 1} \leq L$, $K \ge L^{\tilde p}$ and $L \geq 2$. In particular, the first term vanishes if we can ensure
that 
\begin{equ}[e:constrC]
\tilde{C}_{K, \bar \eps} \ge B_3 L^{3}\;.
\end{equ}
Choosing first $L$ large enough so that the contribution of the two terms in \eqref{e:twoterms} is smaller than $\bar \eps / 2$, then $K$ large enough so that $K \ge L^{\tilde p}$,
and finally $C_{\bar \eps}$ large enough so that \eqref{e:constrC} holds, the claim follows.

Let $\hat Z$ be the model from the proof of Theorem~\ref{t:Phi} and let 
\begin{equ}
\bar \CS_t \colon \bar \CC^\eta  \times \MM \to \bar \CC^\eta
\end{equ}
be the map $\bar \CS_t = \CR_t \CS_t$ from Theorem~\ref{t:FixedMap} 
yielding the maximal solution up to time $t$, i.e.\ 
$\Phi_t = \bar \CS_{t} (\Phi_0, \hat Z)$, with the conventions that 
$\bar \CS_t(\infty, \hat Z) = \infty$ and $\bar \CS_t(\Phi_0, \hat Z) = \infty$ if the maximal
existence time $T_\star$ is less than $t$. Here, $\MM$ denotes the space of all admissible
models as in Section~\ref{ss:ContinuousGauss}.
It follows from \eqref{e:IntegralIdentity}, the locality of the reconstruction map and the
locality of the construction of the model
that $\bar \CS_{t} (\Phi_0, \hat Z)$ depends on the underlying white noise 
only on the time interval $[0, t]$. Moreover, as a consequence of \cite[Prop.~7.11]{Hai14},
one has
\begin{equ}
\bar \CS_{s + t} (\Phi_0, \hat Z) = \bar\CS_{t} \bigl(\bar\CS_{s} (\Phi_0, \hat Z), \hat Z_s\bigr)\,,
\end{equ}
where $\hat Z_s$ is the natural time shift by $s$ of the model $\hat Z$. Since the underlying noise
is white in time, we conclude that the process $\Phi$ is Markov. 
The fact that the measure $\mu$ is reversible for $\Phi$ follows immediately from 
the fact that $\mu_\eps$ is reversible for the discretised process $\Phi^\eps$.
\end{proof}


\bibliographystyle{./Martin}
\bibliography{./bibliography}

\end{document}